\documentclass[12pt]{article}

\usepackage{a4}
\usepackage{epsfig} 
\usepackage{amsmath,lipsum} 
\usepackage{amssymb}  
\usepackage{subfigure}
\usepackage{color}
\usepackage{dsfont,hyperref}

 \usepackage{arydshln}
\usepackage[title]{appendix}

\newtheorem{lemma}{Lemma}
\newtheorem{definition}{Definition}
\newtheorem{theorem}{Theorem}

\newtheorem{remark}{Remark}

\newtheorem{example}{Example}
\newtheorem{procedure}{Procedure}
\newtheorem{assumption}{Assumption}
\newcommand{\tabincell}[2]{\begin{tabular}{@{}#1@{}}#2\end{tabular}}

\newenvironment{proof}{~~\textit{Proof:}}{\hfill$\blacksquare$}

\begin{document}

\begin{center}
\LARGE ``Series-Parallel Mechanical Circuit Synthesis of a Positive-Real Third-Order Admittance Using at Most Six Passive Elements for Inerter-Based Control'' with the Supplementary Material
\end{center}

\vspace{1.5cm}

\noindent This report includes the original manuscript (pp.~2--47) and the supplementary material (pp.~48--67) of ``Series-Parallel Mechanical Circuit Synthesis of a Positive-Real Third-Order Admittance Using at Most Six Passive Elements for Inerter-Based Control''.

\vspace{1cm}

\noindent Authors: Kai Wang, Michael Z. Q. Chen, and Fei Liu

\newpage



\title{Series-Parallel Mechanical Circuit Synthesis of a
Positive-Real Third-Order Admittance Using at Most
Six Passive Elements for Inerter-Based Control}
\author{Kai~Wang$^{1}$,  ~Michael~Z.~Q.~Chen$^{2,}$\footnote{Corresponding author:  Michael Z. Q. Chen, mzqchen@outlook.com \newline
$^{1}$ Key Laboratory of Advanced Process Control for Light Industry (Ministry of Education), School of Internet of Things Engineering, Jiangnan University, Wuxi 214122, P. R. China (e-mail: kaiwang@jiangnan.edu.cn). \newline
$^{2}$ School of Automation, Nanjing University of Science and Technology, Nanjing 210094, P. R. China (e-mail: mzqchen@outlook.com).
\newline
This work is supported by the Natural Science Foundation of Jiangsu Province under grant BK20211234, and  the National Natural Science Foundation of China under grants 61873129 and 61833007.
},  ~ and ~ Fei Liu$^{1}$}
\date{}
\maketitle

\begin{abstract}
This paper investigates the circuit synthesis problem for a certain positive-real bicubic (third-order) admittance with a simple pole at the origin  ($s=0$) to be realizable as a  one-port series-parallel damper-spring-inerter circuit consisting of at most six elements, where the results can be directly applied to the design and physical realization of  inerter-based  control systems.
Necessary and sufficient conditions for such a specific bicubic admittance to be realizable by a one-port passive series-parallel mechanical circuit containing at most six elements are derived, and a group of  mechanical circuit configurations covering the whole set of realizability conditions are presented together with element value expressions.
The conditions and element value expressions are related to the admittance coefficients and the roots of certain algebraic equations.
The circuit synthesis results of this paper  are illustrated by several numerical examples including the   control system design  of a train suspension system.
Any realization circuit in this paper contains much fewer passive elements than the ten-element realization  circuit by the well-known Bott-Duffin circuit synthesis approach. The investigations of this paper can contribute to the theory of  circuit synthesis and many other related fields.

\medskip

\noindent{\em Keywords:} Passivity,  positive-real function, circuit synthesis,  inerter-based  control, parameters optimization.
\end{abstract}








\section{Introduction}  \label{eq: introduction}

\emph{Passive circuit synthesis}\cite{AV06}--\cite{CWC19} is the theory of physically realizing passive network systems, which are described by admittances, impedances, driving-point behavioural approach, etc., as passive   circuits containing only passive elements.\footnote{The phrase ``circuit synthesis'' is also called ``network synthesis'' in the research of this field. Moreover, the phrases ``mechanical circuit'', ``electrical circuit'', etc. of this paper can also be called ``mechanical network'', ``electrical network'', etc.} For any one-port linear time-invariant electrical network, the driving-point behaviour about the port voltage $v$ and current $i$
can be described as $\alpha(\frac{d}{dt}) v = \beta(\frac{d}{dt}) i$, where $\alpha, \beta \in \mathbb{R}[s]$   are real-coefficient  polynomials.
Then, the admittance  defined as $Y(s) := \hat{i}(s)/\hat{v}(s)$  can be expressed as a real-rational function  $Y(s) = \alpha(s)/\beta(s)$.
For a real-rational function $H(s)$, if $H(s)$ is analytic for $\Re(s) > 0$ and satisfies $\Re(H(s)) \geq 0$ for $\Re(s) > 0$, then $H(s)$ is defined to be \emph{positive-real} \cite{AV06}.
The admittance $Y(s)$ (resp. impedance $Z(s)$) of any one-port linear time-invariant passive circuit  must be positive-real \cite{AV06}.
By using the \emph{Bott-Duffin circuit synthesis procedure} \cite{BD49}, any positive-real admittance (resp. impedance) can be realized by a one-port linear time-invariant passive electrical circuit consisting of resistors, inductors, and capacitors (also called RLC circuits) \cite{You15,CWC19}.
However,
the driving-point behavior of the Bott-Duffin circuit realization is not controllable and the number of reactive elements (inductors and capacitors) is much larger than the \emph{McMillan degree} \cite[Chapter~3.6]{AV06} of the admittance or impedance function (see  \cite{HS17}). This means that the Bott-Duffin circuit synthesis procedure may generate several redundant elements and appear nonminimal.
Moreover, since
the Bott-Duffin circuit synthesis procedure is not in an explicit form, it is not convenient to calculate the element values.

Nowadays,   one-port linear time-invariant  passive electrical circuits and mechanical circuits can be completely analogous with each other, where the current, voltage, resistors, inductors, and capacitors are respectively analogous to force, velocity,  dampers, springs, and inerters \cite{Smi02}. Therefore, the analysis and synthesis of passive electrical circuits
are actually equivalent with those of passive mechanical circuits, and one can always utilize one-port mechanical circuits consisting of dampers, springs, and inerters (also called one-port damper-spring-inerter circuits) to physically realize any two-terminal linear time-invariant passive mechanical system based on the theory of circuit synthesis.  Regarded as passive controllers,
one-port passive mechanical circuits consisting of dampers, spring, and inerters have been applied to the control of many vibration systems since the invention of inerters \cite{ELSS06}--\cite{DWSZ21},  where the system performances are shown to be enhanced compared with the conventional mechanical circuits consisting of only dampers and springs.
After determining a suitable passive controller,
 passive circuit synthesis  can be directly
applied to physically realize the controller as passive mechanical circuits, which makes the design process  more convenient and systematic.
Moreover, the control methodology based on passive mechanical systems containing inerters has the advantages of low cost and high reliability. Considering the constraints on space, weight, cost, etc.,
it is essential to restrict the complexity of mechanical circuits, which motivates the further investigation on passive   circuit synthesis problems of positive-real admittances (resp. impedances) by using the restricted number of elements, especially for low-order positive-real functions. During recent years, there have been many new results of passive circuit synthesis \cite{CS09(2)}--\cite{WC21},  but many unsolved problems still exist. For instance, the minimal complexity realization problems of positive-real biquadratic (second-order) and bicubic (third-order) impedances as damper-spring-inerter circuits have not been determined. Specifically, Kalman \cite{Kal10,Smi17}
has   highlighted the significance of investigating the minimal realization problems of passive circuits as a field of system theory.

The main task of solving a passive circuit synthesis problem mainly includes two parts, where the first part is to derive necessary and sufficient conditions for a class of positive-real
functions to be realizable as the admittances (or impedances) of  a specific class of passive circuits, and the other part is to determine a set of realization configurations
covering the conditions with element value expressions. The realizability conditions can be utilized as the optimization constraints in the passive controller design of mechanical   systems, such that the  complexity requirements of the realization circuits can be satisfied. After determining the passive controller,
the circuit synthesis results can be utilized to physically realize the the positive-real admittance (resp. impedance) as a damper-spring-inerter circuit.
In addition to mechanical control, passive circuit synthesis can have a long-term impact on many other related fields, such as circuit theory \cite{RV21},
circuit-antenna design \cite{LBHD11}, self-assembling circuit design \cite{DGYM20},  biomedical engineering \cite{Kes17},
fractional-order circuit systems \cite{Tav20,LLDYYIFL21},
negative imaginary systems \cite{XLP16}, modelling of spatially interconnected systems \cite{ZGSX19},  etc. Therefore, investigating passive circuit synthesis is both theoretically and practically meaningful.

The mechanical admittance in many vibration systems should contain a pole at the origin ($s = 0$) to provide static stiffness, such as the admittances of suspension struts (see  \cite{CHW15}), and
the mechanical circuits whose admittances are positive-real biquadratic or bicubic functions without a pole at $s = 0$ need to be connected in parallel with a spring to form such  admittances
(see  \cite{PS06,WLLSC09}). In  \cite{CWSL13}, the circuit synthesis problem for a class of admittances $Y (s) = \alpha (s) / \beta (s)$ containing a pole at $s = 0$ has been solved, where $\alpha(s)$ is a second-order polynomial and  $\beta(s)$ is a third-order polynomial with a root at $s = 0$.
More generally, the low-complexity realization problems of a class of bicubic admittances $Y(s) = \alpha(s)/\beta(s)$ containing a pole at $s = 0$ need to be further investigated (see \eqref{eq: specific bicubic admittance}), where $\alpha(s)$ and $\beta(s)$ are both third-order polynomials.
By the removal of the pole at $s = 0$ (extracting a parallel spring as in Fig.~\ref{fig: Spring-N2}), any positive-real admittance belonging to this class can be converted into a positive-real biquadratic function, and the circuit synthesis results for biquadratic functions can be applied to complete the realization. By the Bott-Duffin synthesis procedure,    ten elements are needed to realize the whole class of such positive-real admittances. However, the realization circuits may contain fewer elements without first removing the pole.
In \cite{WJ19},
the synthesis results of such an admittance as a one-port five-element damper-spring-inerter circuit was derived, but the dimension of the realizability condition set is less than that of the positive-real condition set.
Therefore, it is almost  impossible to obtain the optimal admittance satisfying the conditions in \cite{WJ19}, which is realizable with five elements, when the optimization constraint of the admittance is simply the positive-real condition in
the passive controller design of mechanical   systems. In order to completely solve  the minimal complexity circuit realizations of such a positive-real admittance, it is essential to further investigate the  realization problem of such an admittance as a $k$-element series-parallel circuit, where $k = 6, 7, ..., 10$.
This paper aims to solve the realization problems as one-port six-element series-parallel circuits, such that the realizability condition set is expanded and more general realizability cases are derived.

The investigations in this paper are highlighted in the following statements.
Necessary and sufficient conditions  are derived for the bicubic admittance with a simple pole at $s=0$ to be realizable as a one-port series-parallel damper-spring-inerter circuit consisting of at most six elements (see Theorem~\ref{theorem: main theorem 03}).
Moreover, it is proved that any admittance satisfying the conditions is realizable as such a circuit by the Foster preamble or one of the circuit configurations in Figs.~\ref{fig: Spring-Biquadratic-Network}--\ref{fig: classes 4 and 5} with element values being expressed.
The synthesis results of the above circuits that can be realized by the Foster preamble after completely removing the pole at the origin are first derived, and the circuit decomposition method and
the structure properties of the realization circuits described by graph theory are utilized to determine circuit configurations to cover all the other cases. By deriving the realization results of these configurations, the final results can be obtained.
The realizability conditions and element value expressions are related to the function coefficients and the roots of certain algebraic equations.
For the circuit synthesis results in this paper, it is more convenient to check the realizability and to achieve the realization by using computer softwares, and the realization
circuits contain much fewer elements than the  circuits  by the Bott-Duffin circuit synthesis procedure. The five-element series-parallel circuit synthesis results in \cite{WJ19} are completely included by the results of this paper as specific cases.
Numerical examples are presented for illustration (see Section~\ref{sec: examples}), and the results of passive controller design and the mechanical circuit realization for an inerter-based train-suspension control system are given based on the results of this paper to show the practical significance.

The contributions of this paper are as follows. The results in this paper can guarantee minimal complexity passive circuit realizations and
can contribute to solving other minimal complexity  circuit synthesis problems for  low-order positive-real functions. In addition to train suspension systems as illustrated in Section~\ref{sec: examples}, the circuit synthesis results in this paper can be directly applied to physically realize the
passive mechanical controllers as six-element series-parallel damper-spring-ineter circuits in many other inerter-based mechanical control systems, such as  mass chain systems, car suspension systems,
building vibration systems,   wind turbines, isolator systems,  etc.
In the design process, after determining
the optimal positive-real admittances of this low-order class that constitute the passive controller
based on the theory of optimization and control, one can utilize the algebraic conditions (Theorems~\ref{theorem: main theorem 01}--\ref{theorem: main theorem 03}) to check the realizability, and each admittance satisfying one of the conditions can be further realized as one of the circuit configurations in this paper. Moreover, the realizability conditions or configurations can be utilized as the optimization constraints in addition to the positive-real conditions.
The research of this paper can also have long-term impacts on
other fields, such as electronic engineering, biomedical engineering, etc.

Compared with   \cite{WJ19},
the investigation methods
in this paper are more general, and the algebraic calculations of the realizability results are much more complex.
The realizability results of biquadratic functions as five-element circuits in \cite{JS11} are utilized to derive the realizability results of $Y(s)$ as a one-port series-parallel damper-spring-inerter circuit containing at most six elements as in Fig.~\ref{fig: Spring-N2}, where the impedance of $N_2$ is  a biquadratic function (see Lemma~\ref{lemma: Spring-Biquadratic-Network}). Furthermore, the circuit decomposition approach and the theory of circuit graph are applied to determine the circuit configurations covering all the other cases (see Lemmas~\ref{lemma: Spring-Other-Network} and \ref{lemma: Other Series-Parallel Structures}), which are more general and effective than the enumeration method in \cite{WJ19}. Therefore, it is easier to generalize the investigations in this paper to solve the synthesis problems of circuits containing more elements. In addition, the recent work in \cite{WC21_JFI} investigates the five-element circuit synthesis problem for another class of positive-real bicubic functions, where the function does not contain any pole or zero on $j \mathbb{R} \cup \infty$. Therefore, the research problems, methodologies, and results in \cite{WC21_JFI} are different from those in this paper.

In this paper, one assumes that all the circuits are one-port (two-terminal) linear time-invariant  passive  mechanical circuits consisting of only dampers, spring, and inerters (also called one-port damper-spring-inerter circuits). If there is no specific statement, all the  elements  are of positive and finite values to guarantee the passivity.
All of the circuit synthesis results in this paper are completely applicable to RLC circuit synthesis by replacing dampers, springs, inerters with resistors, inductors, and capacitors, respectively.

\section{Notations}   \label{sec: notation}

Let $\mathbb{R}$ (resp. $j \mathbb{R}$, $\mathbb{C}$) denote the real number set (resp. imaginary number set, complex number set);
let $\mathbb{R}^n$ denote the $n$-dimensional vector set; let $\mathbb{R}[s]$ (resp. $\mathbb{R}(s)$)
denote the set of real-coefficient polynomials (resp. real-rational functions) in  the indeterminate $s$;
let $\mathbb{R}^{m\times n}$ (resp. $\mathbb{R}^{m\times n}[s]$, $\mathbb{R}^{m\times n}(s)$) denote the set of $m \times n$ matrices
with entries belonging to $\mathbb{R}$ (resp. $\mathbb{R}[s]$, $\mathbb{R}(s)$).
For  $\xi \in \mathbb{C}$,  $\Re(\xi)$  denotes its real part. For $z \in \mathbb{R}^n$, $\| z \|$ denotes its Euclidean norm.
For  $H \in \mathbb{R}(s)$ or $\mathbb{R}^{m\times n}(s)$,  $\delta(H(s))$ denotes its \emph{McMillan degree} \cite[Chapter~3.6]{AV06} and $\| H \|_2$ denotes its $\mathcal{H}_2$ norm.
Let $M^\mathrm{T}$ denotes the transpose of $M \in \mathbb{R}^{n}$, $\mathbb{R}^{m\times n}$,
$\mathbb{R}^{m\times n}[s]$, or $\mathbb{R}^{m\times n}(s)$.
Let $\mathbf{0}$ and  $I$  respectively denote    the zero matrix (or zero vector) and the identity matrix of appropriate dimension, and  let $\mathbf{0}_{m\times n}$ further denote $m \times n$ zero matrix (or zero vector).

For the real symmetric $3\times 3$ \emph{Bezoutian matrix} $\mathcal{B}(\alpha, \beta)$ \cite[Definition~8.24]{Fuh12} of two third-order polynomials $\alpha, \beta \in \mathbb{R}[s]$ expressed as
 $\alpha(s) := \alpha_3 s^3 + \alpha_2 s^2 + \alpha_1 s + \alpha_0$ and $\beta(s) := \beta_3 s^3 + \beta_2 s^2 + \beta_1 s$, the entries $\mathcal{B}_{ij}$ for $i, j = 1, 2, 3$
satisfy
\begin{equation*}
\frac{\alpha(s_2)\beta(s_1) - \beta(s_2)\alpha(s_1)}{s_2-s_1} = \sum_{i = 1}^{3}
\sum_{j = 1}^{3} \mathcal{B}_{ij} s_2^{i-1} s_1^{j-1}.
\end{equation*}
Therefore, one defines the notations:
$\mathcal{B}_{11} := -\alpha_0\beta_1$, $\mathcal{B}_{12} := -\alpha_0\beta_2$,
$\mathcal{B}_{13} := -\alpha_0\beta_3$, $\mathcal{B}_{22} := \mathcal{B}_{13} + \alpha_2\beta_1 - \alpha_1\beta_2$, $\mathcal{B}_{23} := \alpha_3\beta_1 - \alpha_1\beta_3$, and $\mathcal{B}_{33} := \alpha_3\beta_2 - \alpha_2\beta_3$.

Letting $\tilde{\beta} (s) := \beta(s)/s = \beta_3 s^2 + \beta_2 s + \beta_1$, one can formulate the Bezoutian matrix $\tilde{\mathcal{B}}(\alpha, \tilde{\beta})$ of $\alpha(s)$ and $\tilde{\beta}(s)$ whose entries $\tilde{\mathcal{B}}_{ij}$ for $i, j = 1, 2, 3$
satisfy
\begin{equation*}
\frac{\alpha(s_2)\tilde{\beta}(s_1) - \tilde{\beta}(s_2)\alpha(s_1)}{s_2-s_1} = \sum_{i = 1}^{3}
\sum_{j = 1}^{3} \tilde{\mathcal{B}}_{ij} s_2^{i-1} s_1^{j-1}.
\end{equation*}
Therefore, one defines the notations:
$\tilde{\mathcal{B}}_{13} := \alpha_3 \beta_1$,
$\tilde{\mathcal{B}}_{23} := \alpha_3 \beta_2$,
$\tilde{\mathcal{B}}_{33} := \alpha_3 \beta_3$,
$\tilde{\mathcal{B}}_{11} := \alpha_1 \beta_1 - \alpha_0 \beta_2$,
$\tilde{\mathcal{B}}_{12} := \alpha_2 \beta_1 - \alpha_0 \beta_3$,
$\tilde{\mathcal{B}}_{22} := \tilde{\mathcal{B}}_{13} + \alpha_2 \beta_2 - \alpha_1 \beta_3$.

Moreover, define the following notations: $\mathcal{M}_{23} := \alpha_3\beta_1 + \alpha_1\beta_3$, $\mathcal{M}_{33} := \alpha_3\beta_2 + \alpha_2\beta_3$,
$\tilde{\mathcal{M}}_{12} := \alpha_2 \beta_1 + \alpha_0 \beta_3$,
$\Delta_\alpha := \alpha_1 \alpha_2 - \alpha_0 \alpha_3$, and $\Delta_\beta := \beta_2^2 - 4 \beta_1 \beta_3$.

\section{Problem Formulation} \label{sec: problem formulation}

A real-rational function $H \in \mathbb{R}(s)$ is defined to be a \emph{positive-real} function if  $H(s)$ is analytic for $\Re(s) > 0$ and satisfies $\Re(H(s)) \geq 0$ for $\Re(s) > 0$ \cite{AV06}. Specifically, a real-rational function $H \in \mathbb{R}(s)$ is called a \emph{minimum  function} if $H(s)$ is positive-real and contains no zero and pole on $j \mathbb{R} \cup \infty$ \cite{You15}. The \emph{Foster preamble} \cite[pg.~161]{Van60} is the successive removal of the poles or zeros belonging to $j \mathbb{R} \cup \infty$ and the minimum constant of $\Re(Y(j \omega))$ or $\Re(Y^{-1}(j \omega))$, such that both the remaining impedance   and admittance are \emph{minimum functions} \cite[pg.~161]{Van60}  with lower  McMillan degrees  or one of the impedance and admittance is zero.
The \emph{Bott-Duffin circuit synthesis procedure} is the most effective passive circuit synthesis algorithm that can realize any given positive-real impedance (also admittance) to be a one-port series-parallel damper-spring-inerter (RLC) circuit, which consists of the Foster preambles and Bott-Duffin cycles \cite[Section~2.4]{CWC19}, \cite{BD49}. In comparison, the modified Bott-Duffin circuit synthesis procedures by Reza, Pantell, Fialkow, and Gerst   can reduce one element in each Bott-Duffin cycle but generate  the non-series-parallel circuit structures \cite{Hug17}. Moreover, some other procedures, such as Miyata synthesis procedure, can only be applied to some specific classes of positive-real functions \cite{Van60}.
As shown in Section~\ref{eq: introduction}, the Bott-Duffin circuit synthesis procedure cannot guarantee the complexity of the circuits realizing the given positive-real functions to be minimal
in many cases.
Therefore, it is essential to investigate the minimal complexity synthesis problems of damper-spring-inerter circuits for given classes of low-order positive-real functions, such as some investigations in \cite{MS19,JS11,Hug17,ZJWN17,WJ19,Hug20,WC21}.

The  design process of a general class of inerter-based control systems is shown in Appendix~\ref{sec: mechanical control}, where the positive-real admittances \eqref{eq: Qi} constitute the passive controller.
In this paper, we will investigate the minimal complexity passive circuit synthesis problems when the  McMillan degree of \eqref{eq: Qi} is three,  and the realizability results  as series-parallel damper-spring-inerter circuits containing no more than six elements will be derived.

As defined in Section~\ref{sec: notation}, $\alpha, \beta \in \mathbb{R}[s]$ are two third-order real-coefficient polynomials in $s$. When the    McMillan degree of \eqref{eq: Qi} is equal or less than three, a certain bicubic admittance $Y(s) \in \mathbb{R}(s)$ is formulated as
\begin{equation} \label{eq: specific bicubic admittance}
Y(s) = \frac{\alpha(s)}{\beta(s)} =  \frac{\alpha_3 s^3 + \alpha_2 s^2 + \alpha_1 s + \alpha_0}{\beta_3 s^3 + \beta_2 s^2 + \beta_1 s}.
\end{equation}
To guarantee the positive-realness of $Y(s)$ in \eqref{eq: specific bicubic admittance},   assume that all of coefficients are nonnegative, that is, $\alpha_i, \beta_j \geq 0$ for $i = 0, 1, 2, 3$ and $j = 1, 2, 3$.
If a given admittance $Y(s)$ in \eqref{eq: specific bicubic admittance} is positive-real and contains any pole or zero belonging to $j \mathbb{R} \cup \infty$ except the simple pole at the origin ($s=0$), then it can be verified that $Y(s)$ is realizable by a  series-parallel damper-spring-inerter circuit that contains at most five elements by making use of the Foster preamble.
Moreover, if the McMillan degree of $Y(s)$ in
\eqref{eq: specific bicubic admittance} is lower than three ($\delta(Y(s)) < 3$), which is equivalent to $R_0(\alpha, \beta) = 0$ with
\[
R_0(\alpha, \beta) := \left|\begin{array}{cccccc} \alpha_3 & \alpha_2 & \alpha_1 & \alpha_0 & 0 & 0 \\0 & \alpha_3 & \alpha_2 & \alpha_1 & \alpha_0 & 0 \\0 & 0 & \alpha_3 & \alpha_2 & \alpha_1 & \alpha_0 \\ \beta_3 &  \beta_2 &  \beta_1 &  0 & 0 & 0 \\0 & \beta_3 &  \beta_2 &  \beta_1 &  0 & 0 \\0 & 0 & \beta_3 &  \beta_2 &  \beta_1 &  0 \end{array}\right|,
\]
then any positive-real admittance $Y(s)$ in \eqref{eq: specific bicubic admittance} must be realizable by a one-port series-parallel damper-spring-inerter circuit that contains at most four elements.
Therefore, to exclude the above low-complexity realization cases that have been solved,
one will make the assumption for the admittance $Y(s)$ in \eqref{eq: specific bicubic admittance} as follows.
\begin{assumption} \label{assumption: 01}
For any admittance $Y(s)$ in \eqref{eq: specific bicubic admittance}, the coefficients are assumed to satisfy
$\alpha_i, \beta_j > 0$ for $i = 0, 1, 2, 3$ and $j = 1, 2, 3$, $\Delta_\alpha \neq 0$, and $R_0(\alpha, \beta)  \neq 0$.
\end{assumption}
It can be derived that the equation $\alpha(s) = 0$ where $\alpha_i > 0$ for $i = 0, 1, 2, 3$ does not have any root on $j \mathbb{R}$ if and only if $\Delta_\alpha \neq 0$. Moreover,
there does not exist any root on $j \mathbb{R} \setminus \{0\}$ for $\beta(s) = 0$ where $\beta_j > 0$, $j = 1, 2, 3$.
Therefore, Assumption~\ref{assumption: 01} can guarantee that the McMillan degree of $Y(s)$ is three, $Y(s)$ contains a simple pole at $s=0$,
and $Y(s)$ does not contain any other pole or zero belonging to $j \mathbb{R} \cup \infty$ except the simple pole at $s=0$.

By applying the Bott-Duffin circuit synthesis procedure, any positive-real admittance $Y(s)$ in
\eqref{eq: specific bicubic admittance} satisfying Assumption~\ref{assumption: 01} can be realized as a one-port damper-spring-inerter series-parallel circuit containing at most ten elements. The  five-element circuit synthesis results for such a class of admittances derived in \cite{WJ19} are only specific subcases of the positive-real condition.
Therefore, in order to completely solve the minimal complexity circuit realizations of the positive-real admittance $Y(s)$ in \eqref{eq: specific bicubic admittance}, it is both theoretically and practically significant to further investigate the  realization problem of such an admittance as a $k$-element series-parallel circuit, where $k = 6, 7, ..., 10$, such that the minimal number of elements $n_{\min}$ to realize the whole class of positive-real admittances in
\eqref{eq: specific bicubic admittance} satisfying Assumption~\ref{assumption: 01} can be determined ($n_{\min} \leq 10$).

The task of this paper is to solve the circuit synthesis problem for any given admittance $Y(s)$ in \eqref{eq: specific bicubic admittance} satisfying Assumption~\ref{assumption: 01} to be realizable by a one-port series-parallel damper-spring-inerter circuit that contains at most six passive elements, where the circuit synthesis results can guarantee the minimality of the circuit complexity.

\section{Main Results} \label{sec: main results}

This section will derive and present the main results of this paper in Theorems~\ref{theorem: main theorem 01}--\ref{theorem: main theorem 03}, where Lemmas~\ref{lemma: positive-realness}--\ref{lemma: lossless subnetwork N1} are the basic lemmas to derive the main results. Theorem~\ref{theorem: main theorem 01} shows the necessary and sufficient condition for admittance $Y(s)$ in \eqref{eq: specific bicubic admittance} satisfying Assumption~\ref{assumption: 01} to be realizable as a one-port
series-parallel damper-spring-inerter circuit containing at most six elements with the specific structure in Fig.~\ref{fig: Spring-N2} (one of the conditions in Lemmas~\ref{lemma: Spring-Biquadratic-Network} and
\ref{lemma: Case-6}--\ref{lemma: Case-8}). Assuming that any of the conditions in Theorem~\ref{theorem: main theorem 01} does not holds,
Theorem~\ref{theorem: main theorem 02} presents the necessary and sufficient condition for the realizability
as any other one-port series-parallel
damper-spring-inerter circuit containing at most six elements (one of the conditions in Lemmas~\ref{lemma: 2-1}--\ref{lemma: 5-2}). Finally, Theorem~\ref{theorem: main theorem 03} combines the results in Theorems~\ref{theorem: main theorem 01} and \ref{theorem: main theorem 02}. The results show that any  admittance satisfying the conditions is realizable as such a circuit by the Foster preamble or one of the circuit configurations in Figs.~\ref{fig: Spring-Biquadratic-Network}--\ref{fig: classes 4 and 5}.

\subsection{Basic Lemmas}

For the admittance $Y(s)$ in \eqref{eq: specific bicubic admittance} satisfying Assumption~\ref{assumption: 01}, the residue of the pole at $s = 0$ is calculated as
$\alpha_0/\beta_1$. Then, it follows that
\begin{equation}  \label{eq: Zb biquadratic impedance}
Z_b(s) := \left( Y(s) - \frac{\alpha_0}{\beta_1 s}  \right)^{-1}
= \frac{a_2 s^2 + a_1 s + a_0}{d_2 s^2 + d_1 s + d_0},
\end{equation}
where
\begin{equation}  \label{eq: a d}
\begin{split}
a_2 = k \beta_1\beta_3, ~~
a_1 = k \beta_1 \beta_2, ~~
a_0 = k \beta_1^2,  ~~
d_2 = k \tilde{\mathcal{B}}_{13},  ~~
d_1 = k \tilde{\mathcal{B}}_{12}, ~~
d_0 = k \tilde{\mathcal{B}}_{11},
\end{split}
\end{equation}
for any $k > 0$. Then, the following lemma presenting a necessary and sufficient condition for   $Y(s)$ in
\eqref{eq: specific bicubic admittance} to be positive-real can be derived, which is equivalent to the result in
 \cite{CS09(2)}.

\begin{lemma}  \label{lemma: positive-realness}
Any admittance $Y(s)$ in \eqref{eq: specific bicubic admittance} satisfying Assumption~\ref{assumption: 01}
is positive-real, if and only if
$\tilde{\mathcal{B}}_{11} \geq 0$,
$\tilde{\mathcal{B}}_{12} \geq 0$, and
$2 \tilde{\mathcal{B}}_{13} - \tilde{\mathcal{B}}_{22} -
2 \sqrt{\tilde{\mathcal{B}}_{11} \tilde{\mathcal{B}}_{33}} \leq 0$.
\end{lemma}
\begin{proof}
By the results in   \cite[pg.~34]{Bah84}, if $Y(s)$ is a positive-real function, then  $1/Z_b(s) = Y(s) - \alpha_0/(\beta_1 s)$ is positive-real, which by the definition of positive-realness further implies that $Z_b(s)$ as in \eqref{eq: Zb biquadratic impedance} is positive-real. Conversely, the positive-realness of $Z_b(s)$ in the form of \eqref{eq: Zb biquadratic impedance} can imply that $Y(s) = 1/Z_b(s) + \alpha_0/(\beta_1 s)$ is positive-real. Together with  the positive-realness condition of biquadratic functions (see \cite{CS09(2)}), this lemma can be proved.
\end{proof}

\begin{definition}  \label{def: 01}
For any $n$-port  damper-spring-inerter circuit,
the \emph{graph} $\mathcal{G}(\mathcal{V},\mathcal{E})$ (or called \emph{network graph} in  \cite[pg.~28]{CWC19}) of the circuit   is the linear graph whose vertex set $\mathcal{V}$ consists of all the vertices
representing the velocity nodes and whose edge set $\mathcal{E}$ contains all the edges representing the
circuit elements (damper, spring, or inerter), the \emph{port graph} $\mathcal{G}_p(\mathcal{V}_p,\mathcal{E}_p)$ of the circuit is the linear graph whose vertices and edges respectively represent all the external terminals and ports, where $\mathcal{V}_p \subset \mathcal{V}$,
and the \emph{augmented graph} $\mathcal{G}_a(\mathcal{V}_a,\mathcal{E}_a)$ of the circuit  is the union of  $\mathcal{G}(\mathcal{V},\mathcal{E})$ and  $\mathcal{G}_p(\mathcal{V}_p,\mathcal{E}_p)$, where $\mathcal{V}_a = \mathcal{V}$ and $\mathcal{E}_a = \mathcal{E} \cup \mathcal{E}_p$.
\end{definition}

Based on Definition~\ref{def: 01}, the following assumption for the one-port damper-spring-inerter circuits can be made.
\begin{assumption}  \label{assumption: 02}
For a one-port damper-spring-inerter circuit, its graph  is \emph{connected} \cite[pg.~15]{SR61} and its augmented graph is \emph{nonseparable} \cite[pg.~35]{SR61}.
\end{assumption}

If Assumption~\ref{assumption: 02} does not hold, then the circuit can be equivalent to another one-port damper-spring-inerter circuit that contains fewer elements and satisfies Assumption~\ref{assumption: 02}.

For any one-port (two-terminal) damper-spring-inerter circuit,
$\mathcal{P}(a,a')$ denotes the \emph{path} \cite[pg.~14]{SR61} of the graph $\mathcal{G}(\mathcal{V},\mathcal{E})$ whose end vertices   $a, a' \in \mathcal{V}$ represent two external terminals of the circuit; $\mathcal{C}(a,a')$ denotes the \emph{cut-set} \cite[pg.~28]{SR61}, such that removing the edges of $\mathcal{C}(a,a')$  can  partition  $\mathcal{G}(\mathcal{V},\mathcal{E})$ as two connected subgraphs respectively containing vertices $a$ and $a'$.
Furthermore,
denote the path $\mathcal{P}(a,a')$   whose all the edges represent inerters (resp. springs) as $b$-$\mathcal{P}(a,a')$ (resp. $k$-$\mathcal{P}(a,a')$), and denote the cut-set $\mathcal{C}(a,a')$ whose all the edges correspond to inerters (resp.  springs) as $b$-$\mathcal{C}(a,a')$ (resp. $k$-$\mathcal{C}(a,a')$).
Then, the following lemmas present the constraints on the circuit realizations of $Y(s)$.

\begin{lemma} \cite[Theorem 2]{Ses59} \label{lemma: graph constraint}
If any admittance $Y(s)$ in \eqref{eq: specific bicubic admittance} satisfying Assumption~\ref{assumption: 01}
can be realized by the one-port damper-spring-inerter circuit satisfying Assumption~\ref{assumption: 02}, then the graph of the circuit
must have $k$-$\mathcal{P}(a,a')$ and cannot not have any of $b$-$\mathcal{P}(a,a')$, $k$-$\mathcal{C}(a,a')$, or $b$-$\mathcal{C}(a,a')$.
\end{lemma}

\begin{lemma}  \label{lemma: lossless subnetwork N1}
Any  admittance $Y(s)$ in \eqref{eq: specific bicubic admittance} satisfying Assumption~\ref{assumption: 01} cannot be realized
by the one-port passive circuit that is the parallel or series of  a spring-inerter circuit $N_l$
and a passive circuit $N_p$, where the McMillan degree of
the admittance $Y_l(s)$ of the spring-inerter circuit $N_l$ is at least two.
\end{lemma}
\begin{proof}
Since the McMillan degree of   $Y_l(s)$  is at least two, $Y_l(s)$ must contain at least one pair of finite poles or zeros on $j \mathbb{R} \setminus \{0\}$
referring to the general form of the admittances of spring-inerter  circuits
(see \cite[pg.~51]{Bah84} and  \cite[pg.~15]{CWC19}). This implies that $Y(s)$ must contain such poles or zeros on $j \mathbb{R} \setminus \{0\}$, which contradicts the assumption.
\end{proof}

Since the McMillan degree of any positive-real admittance cannot exceed the number of energy storage elements (spring or inerter) needed to realize the function \cite[pg.~370]{AV06},
any one-port passive circuit that can realize
the positive-real admittance $Y(s)$ in \eqref{eq: specific bicubic admittance} satisfying Assumption~\ref{assumption: 01} contains at least three energy storage elements.   By
Lemmas~\ref{lemma: graph constraint} and \ref{lemma: lossless subnetwork N1}, one can also indicate that the least number of dampers for the series-parallel realizations is two. Therefore, the total number of elements is at least five.

\subsection{Series-Parallel Circuit Realization Containing a Parallel Spring}

This subsection will first investigate the synthesis problem of a specific class of one-port series-parallel damper-spring-inerter circuits consisting of at most six elements, any of which
is the parallel structure of a spring $k_1$ and a circuit $N_2$ (see Fig.~\ref{fig: Spring-N2}).
The main result of this subsection is shown in Theorem~\ref{theorem: main theorem 01}, where the necessary and sufficient condition for the realizability is the union of the conditions in Lemmas~\ref{lemma: Spring-Biquadratic-Network} and
\ref{lemma: Case-6}--\ref{lemma: Case-8}, and the admittance can be realized by the Foster preamble or as one of the mechanical circuit configurations in Figs.~\ref{fig: Spring-Biquadratic-Network} and
\ref{fig: Spring-Biquadratic-Network-partial-removal} with element values being expressed.

As defined in  \cite{JS11}, a real-rational function $H(s)$ is defined to be \emph{regular} if $H(s)$ is positive-real and the minimal value of $\Re(H(j\omega))$ or $\Re(1/H(j\omega))$ is at $\omega = 0$ or $\omega = \infty$. As shown in  \cite[Lemma~5]{JS11}, a biquadratic function $Z_b(s)$ in  \eqref{eq: Zb biquadratic impedance} with $a_i, d_j \geq 0$ for $i, j = 0, 1, 2$ is regular, if and only if  one of the cases holds:
1. $a_2 d_0 - a_0 d_2 \geq 0$ and either $d_1 (a_1 d_0 - a_0 d_1) - d_0 (a_2 d_0 - a_0 d_2) \geq 0$ or $a_1 (a_2 d_1 - a_1 d_2) - a_2 (a_2 d_0 - a_0 d_2) \geq 0$; 2. $a_2 d_0 - a_0 d_2 \leq 0$ and either $d_2 (a_2 d_0 - a_0 d_2) - d_1 (a_2 d_1 - a_1 d_2) \geq 0$ or $a_0 (a_2 d_0 - a_0 d_2) - a_1 (a_1 d_0 - a_0 d_1) \geq 0$. Then, the following lemma can be derived.

\begin{lemma}  \label{lemma: Spring-Biquadratic-Network}
Any   admittance $Y(s)$ in \eqref{eq: specific bicubic admittance} satisfying Assumption~\ref{assumption: 01} can be realized by a one-port series-parallel damper-spring-inerter circuit containing at most six elements as in Fig.~\ref{fig: Spring-N2}, where the impedance of $N_2$ is  a biquadratic function, if and only if $Y(s)$ satisfies one of the five conditions:
\begin{enumerate}
  \item[1.] $Z_b(s)$ as in \eqref{eq: Zb biquadratic impedance} is a regular  function;
  \item[2.] $d_1 > 0$, $d_0 > 0$, and $a_2 (a_1d_1-a_2d_0)^2 - a_1^2 d_2 (a_1d_1-a_2d_0) + a_1^2 a_0 d_2^2   = 0$;
  \item[3.] $d_1 > 0$, $d_0 > 0$, and $d_2 (a_1d_1-a_0d_2)^2 - a_2d_1^2(a_1d_1-a_0d_2) + a_2^2d_1^2d_0 = 0$;
  \item[4.] $d_1 > 0$, $d_0 > 0$, and $a_0 (a_1d_1 - a_0d_2)^2 - a_1^2d_0(a_1d_1 - a_0d_2) + a_2a_1^2d_0^2 = 0$;
  \item[5.] $d_1 > 0$, $d_0 > 0$, and $d_0 (a_1d_1 - a_2d_0)^2 - a_0d_1^2(a_1d_1 - a_2d_0) + a_0^2 d_2 d_1^2 = 0$,
\end{enumerate}
where $a_i, d_j$ for $i, j = 0, 1, 2$ satisfy \eqref{eq: a d} for any $k > 0$. Moreover, if Condition~1 holds,
then
the Foster preamble can be utilized to realize  $Y(s)$ as the required circuit
after extracting the parallel spring $k_1 = \alpha_0/\beta_1$ to remove the pole at $s = 0$.
If   one of Conditions~2--5 holds, then $Y(s)$ can be realized by one of the  circuit configurations in Fig.~\ref{fig: Spring-Biquadratic-Network}, where the element values are expressed in Table~\ref{table: Element values of Cases2-5}.
\end{lemma}
\begin{proof}
It is clear that $Y(s)$ is realizable by the required circuit in this lemma, if and only if the biquadratic impedance $Z_b(s)$ calculated in \eqref{eq: Zb biquadratic impedance} is realizable by a one-port five-element series-parallel circuit. As shown in  \cite{JS11},  $Z_b(s)$ is realizable by such a class of circuits, if and only if $Z_b(s)$ is regular (Condition~1), or $Z_b(s)$ is realizable by a one-port five-element series-parallel circuit that contains three energy storage elements, which is the circuit $N_2$ in parallel with spring $k_1$ for any of the mechanical circuit configurations in Fig.~\ref{fig: Spring-Biquadratic-Network}.
Together with the   results in  \cite{Lad64}, Conditions~2--5 of this lemma and the element value expressions in
Table~\ref{table: Element values of Cases2-5}
 can be obtained.
\end{proof}

\begin{figure}[thpb]
      \centering
      \includegraphics[scale=1.2]{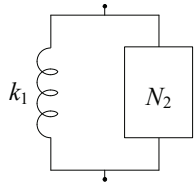}
      \caption{The one-port passive circuit that is the parallel structure of a spring $k_1$ and a series-parallel damper-spring-inerter circuit $N_2$, where $N_2$ contains at most five elements.}
      \label{fig: Spring-N2}
\end{figure}

\begin{figure}[thpb]
      \centering
      \subfigure[]{
      \includegraphics[scale=1.0]{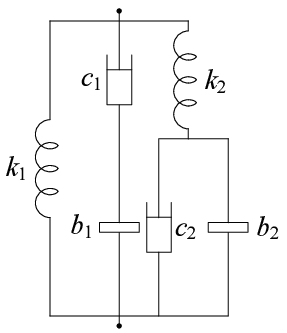}
      \label{fig: Case-2}} \hspace{-0.6cm}
      \subfigure[]{
      \includegraphics[scale=1.0]{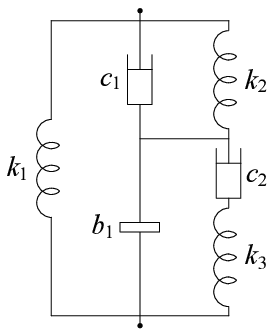}
      \label{fig: Case-3}} \hspace{-0.4cm}
      \subfigure[]{
      \includegraphics[scale=1.0]{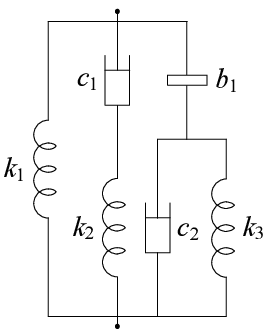}
      \label{fig: Case-4}}  \hspace{-0.5cm}
      \subfigure[]{
      \includegraphics[scale=1.0]{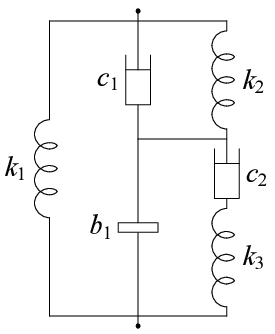}
      \label{fig: Case-5}}
      \caption{The one-port six-element series-parallel circuit configurations corresponding to Conditions~2--5 of Lemma~\ref{lemma: Spring-Biquadratic-Network}, respectively. All of the configurations belong to the structure in Fig.~\ref{fig: Spring-N2}. Here, the element values of the configurations in (a)--(d) are positive and finite.}
      \label{fig: Spring-Biquadratic-Network}
\end{figure}

\begin{table}[!t]
\renewcommand{\arraystretch}{1.1}
\caption{The element values   of the configurations in Fig.~\ref{fig: Spring-Biquadratic-Network}, where
$a_i, d_j$ for $i,j = 0, 1, 2$ satisfy  \eqref{eq: a d}.}
\label{table: Element values of Cases2-5}
\centering
\begin{tabular}{c   c}
\hline
Configuration & Element value expressions    \\
\hline
Fig.~\ref{fig: Case-2} &    \tabincell{c}{$c_1 = \displaystyle\frac{d_2}{a_2}$, $c_2 = \displaystyle\frac{d_0}{a_0}$, $k_1 =  \displaystyle\frac{\alpha_0}{\beta_1}$,  $k_2 = \displaystyle\frac{d_0}{a_1}$,   $b_1 =  \displaystyle\frac{a_1d_1-a_2d_0}{a_1a_0}$, $b_2 = \displaystyle\frac{a_2d_0}{a_1a_0}$}   \\
Fig.~\ref{fig: Case-3} & \tabincell{c}{$c_1 = \displaystyle\frac{d_2}{a_2}$,
$c_2 = \displaystyle\frac{d_0}{a_0}$, $k_1 =  \displaystyle\frac{\alpha_0}{\beta_1}$,
$k_2 = \displaystyle\frac{d_1d_0}{a_1d_1-a_0d_2}$, $k_3 = \displaystyle\frac{d_1d_0}{a_0d_2}$,
$b_1 = \displaystyle\frac{d_1}{a_0}$}  \\
Fig.~\ref{fig: Case-4} & \tabincell{c}{$c_1 = \displaystyle\frac{d_0}{a_0}$, $c_2 = \displaystyle\frac{d_2}{a_2}$, $k_1 = \displaystyle\frac{\alpha_0}{\beta_1}$, $k_2 = \displaystyle\frac{a_1d_1 - a_0d_2}{a_2a_1}$,   $k_3 = \displaystyle\frac{a_0d_2}{a_2a_1}$, $b_1 = \displaystyle\frac{d_2}{a_1}$}  \\
Fig.~\ref{fig: Case-5} & \tabincell{c}{$c_1 = \displaystyle\frac{d_0}{a_0}$,
$c_2 = \displaystyle\frac{d_2}{a_2}$, $k_1 = \displaystyle\frac{\alpha_0}{\beta_1}$,
$k_2 = \displaystyle\frac{d_1}{a_2}$,   $b_1 = \displaystyle\frac{d_2 d_1}{a_1d_1-a_2d_0}$,
$b_2 = \displaystyle\frac{d_2d_1}{a_2d_0}$}  \\
\hline
\end{tabular}
\end{table}

The above lemma shows the realization results by completely removing the pole at $s=0$, and the following lemma
further presents the possible configurations generated by the partial removal of such a pole in order to provide the spring $k_1$ in Fig.~\ref{fig: Spring-N2}.

\begin{lemma}  \label{lemma: Spring-Other-Network}
For any admittance $Y(s)$ in \eqref{eq: specific bicubic admittance} that satisfies Assumption~\ref{assumption: 01} and does not satisfy
the conditions of  Lemma~\ref{lemma: Spring-Biquadratic-Network},
$Y(s)$  is realizable by a  one-port series-parallel damper-spring-inerter circuit containing at most six elements as in Fig.~\ref{fig: Spring-N2}, if and only if $Y(s)$ can be realized by one of the circuit configurations in
Fig.~\ref{fig: Spring-Biquadratic-Network-partial-removal}.
\end{lemma}
\begin{proof}
The details of this proof can be referred to Appendix~\ref{appendix: A}.  In the proof, Lemmas~\ref{lemma: graph constraint} and \ref{lemma: lossless subnetwork N1} are utilized to establish the realization constraints of $N_2$ as in Fig.~\ref{fig: Spring-N2}, which are further applied to derive the configurations in Fig.~\ref{fig: Spring-Biquadratic-Network-partial-removal}.
\end{proof}

\begin{figure}[thpb]
      \centering
      \subfigure[]{
      \includegraphics[scale=1.0]{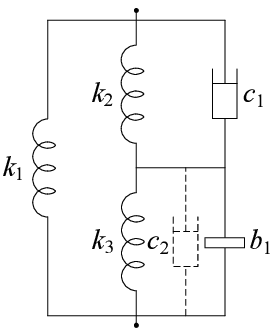}
      \label{fig: Case-6}}
      \subfigure[]{
      \includegraphics[scale=1.0]{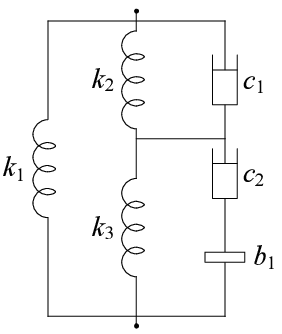}
      \label{fig: Case-7}}
      \subfigure[]{
      \includegraphics[scale=1.0]{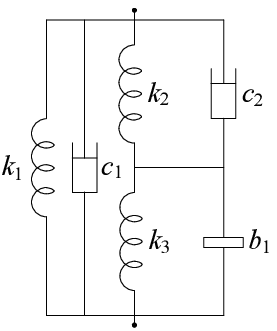}
      \label{fig: Case-8}}
      \caption{The one-port series-parallel circuit configurations containing at most six elements, which can realize the admittance $Y(s)$ in
      \eqref{eq: specific bicubic admittance}  that satisfies Assumption~\ref{assumption: 01} and does not satisfy
the conditions of  Lemma~\ref{lemma: Spring-Biquadratic-Network}. Here, the element value of the damper $c_2$ in (a) satisfies $c_2 \geq 0$ (this damper is open-circuited when $c_2 = 0$), and all the other element values  in (a)--(c) are positive and finite.}
      \label{fig: Spring-Biquadratic-Network-partial-removal}
\end{figure}

Then, the following Lemmas~\ref{lemma: Case-6}--\ref{lemma: Case-8}  present the realizability conditions and element value expressions of the configurations in Fig.~\ref{fig: Spring-Biquadratic-Network-partial-removal}.

\begin{lemma}  \label{lemma: Case-6}
Any admittance $Y(s)$ in \eqref{eq: specific bicubic admittance} satisfying Assumption~\ref{assumption: 01} can be realized by the circuit configuration in Fig.~\ref{fig: Case-6} ($c_2$ can be zero),    if and only if there exits a positive root $x > 0$ for the equation
\begin{subequations}
\begin{equation} \label{eq: Case-6 equation}
\begin{split}
\Delta_\beta x^2 + 2(\beta_2 \beta_3 \mathcal{B}_{23} - (\beta_2^2 - 2\beta_1\beta_3) \mathcal{B}_{33} - 2\beta_3^2 \mathcal{B}_{13})x
+  (\beta_2 \mathcal{B}_{33} - \beta_3 \mathcal{B}_{23})^2= 0,
\end{split}
\end{equation}
such that
\begin{equation} \label{eq: Case-6 condition01}
0< x \leq \tilde{\mathcal{B}}_{23},
\end{equation}
\begin{equation}  \label{eq: Case-6 condition02}
\beta_2 \mathcal{B}_{33} - \beta_3 \mathcal{B}_{23}  < \beta_2 x <   \beta_2 \mathcal{B}_{33} + \beta_3 \mathcal{M}_{23},
\end{equation}
and
\begin{equation}  \label{eq: Case-6 condition03}
2 x^2 - (\tilde{\mathcal{B}}_{23} + 2 \mathcal{B}_{33})x + \alpha_3
(\beta_2 \mathcal{B}_{33} - \beta_3 \mathcal{B}_{23}) > 0.
\end{equation}
\end{subequations}
Moreover, the element values can be expressed as
\begin{equation}  \label{eq: Case-6 element values}
\begin{split}
c_1  = \frac{\alpha_3}{\beta_3}, ~~
c_2 = \frac{\alpha_3 (\tilde{\mathcal{B}}_{23} - x)}{\beta_3 x},  ~~
k_1  = \frac{2 x^2 - (\tilde{\mathcal{B}}_{23} + 2 \mathcal{B}_{33})x + \alpha_3
(\beta_2 \mathcal{B}_{33} - \beta_3 \mathcal{B}_{23})}{2 \beta_3^2 x},  \\
k_2  = \frac{\alpha_3 (\beta_2 x - (\beta_2 \mathcal{B}_{33} - \beta_3 \mathcal{B}_{23}))}{2 \beta_3^2 x},     ~~
k_3  = \frac{\alpha_3(-\beta_2 x + (\beta_2 \mathcal{B}_{33} + \beta_3 \mathcal{M}_{23}))}{2 \beta_3^2 x}, ~~b_1  = \frac{\alpha_3^2}{x}.
\end{split}
\end{equation}
\end{lemma}
\begin{proof}
The details of the proof are presented in Appendix~\ref{appendix: B}.
\end{proof}

\begin{lemma} \label{lemma: Case-7}
Any admittance $Y(s)$ in \eqref{eq: specific bicubic admittance} satisfying Assumption~\ref{assumption: 01} can be realized by the circuit configuration in  Fig.~\ref{fig: Case-7}, if and only if
there exists positive roots $x > 0$ and $y > 0$ for the equations
\begin{subequations}
\begin{equation} \label{eq: Case-7 equation}
\begin{split}
\tilde{\mathcal{B}}_{11}  \tilde{\mathcal{B}}_{12} y^2
+
(\alpha_0 - \beta_1 x) \big((\beta_3 \tilde{\mathcal{B}}_{11} + \beta_1 \tilde{\mathcal{B}}_{13}) x
-
(\alpha_2 \tilde{\mathcal{B}}_{11} + \alpha_0 \tilde{\mathcal{B}}_{13}) \big) y
+  2 \alpha_3 (\alpha_0 - \beta_1 x)^3 = 0
\end{split}
\end{equation}
and
\begin{equation} \label{eq: Case-7 equation02}
\begin{split}
\beta_1   \tilde{\mathcal{B}}_{11}^3  y^4 + \tilde{\mathcal{B}}_{11}^2 & (\alpha_0 - \beta_1 x)  (\beta_1 \beta_2 x + \mathcal{B}_{12} - 2 \tilde{\mathcal{B}}_{11}) y^3  \\
&+ \tilde{\mathcal{B}}_{11} (\alpha_0 - \beta_1 x)^2   \big( \tilde{\mathcal{B}}_{11} (\alpha_1 -\beta_2 x)
+ \beta_3 (\alpha_0-\beta_1 x)^2    \big) y^2
 - \alpha_3 (\alpha_0 - \beta_1 x)^6 = 0,
\end{split}
\end{equation}
such that
\begin{equation}  \label{eq: Case-7 equation01}
\alpha_0 - \beta_1 y < \beta_1 x < \alpha_0
\end{equation}
and
\begin{equation}  \label{eq: Case-7 equation02}
\beta_3  \tilde{\mathcal{B}}_{11} y^2 - \alpha_3 (\alpha_0 - \beta_1 x)^2 > 0.
\end{equation}
\end{subequations}
Moreover, the element values can be expressed as
\begin{equation}  \label{eq: Case-7 element values}
\begin{split}
c_1 = \frac{\tilde{\mathcal{B}}_{11} y^2}{(\alpha_0 - \beta_1 x)^2},  ~~
c_2 = \frac{\alpha_3 \tilde{\mathcal{B}}_{11} y^2}{\beta_3 \tilde{\mathcal{B}}_{11} y^2  - \alpha_3 (\alpha_0 - \beta_1 x)^2}, ~~ k_1 = x,  \\
 k_2 = y, ~~ k_3 = \frac{y(\alpha_0 - \beta_1 x)}{\beta_1 (x + y) - \alpha_0}, ~~
b_1 = \frac{\alpha_3 (\alpha_0 - \beta_1 x)^2}{\tilde{\mathcal{B}}_{11} (\beta_1 (x + y) - \alpha_0)}.
\end{split}
\end{equation}
\end{lemma}
\begin{proof}
The derivation process is similar to that of Lemma~\ref{lemma: Case-6}, and the details of the proof are presented in the
supplementary material \cite[Section~II.1]{Wang_sup}.
\end{proof}

\begin{lemma}  \label{lemma: Case-8}
Any admittance $Y(s)$ in \eqref{eq: specific bicubic admittance} satisfying Assumption~\ref{assumption: 01} can be realized by the circuit configuration in  Fig.~\ref{fig: Case-8}, if and only if there exits a positive root $x > 0$ for the equation
\begin{subequations}
\begin{equation} \label{eq: Case-8 equation}
\begin{split}
 \Delta_\beta x^2 - 4\beta_3 (2\beta_1 \mathcal{B}_{23} -\beta_2(\mathcal{B}_{22} - 2 \mathcal{B}_{13})) x  - 4\beta_3 (\beta_1 \mathcal{B}_{23}^2 - \beta_2 \mathcal{B}_{23} (\mathcal{B}_{22} - 2 \mathcal{B}_{13}) + \beta_2 \mathcal{B}_{12} \mathcal{B}_{33}) = 0,
\end{split}
\end{equation}
such that
\begin{equation} \label{eq: Case-8 condition01}
\max\left\{ -2 \mathcal{B}_{23}, -\frac{\beta_3 \mathcal{B}_{23} + \beta_2 \mathcal{B}_{33}}{\beta_3} \right\} < x < \alpha_2\beta_2 - \mathcal{B}_{23}
\end{equation}
and
\begin{equation}  \label{eq: Case-8 condition02}
(\beta_2^2 - 2\beta_1\beta_3) x - 2\beta_3(\beta_1 \mathcal{B}_{23} - \beta_2 (\mathcal{B}_{22} - \mathcal{B}_{13})) < 0.
\end{equation}
\end{subequations}
Moreover, the element values can be expressed as
\begin{equation}  \label{eq: Case-8 element values}
\begin{split}
c_1  = \frac{\alpha_2\beta_2 - \mathcal{B}_{23} - x}{\beta_2^2}, ~~
c_2  = \frac{\beta_3 x + \beta_3 \mathcal{B}_{23} + \beta_2 \mathcal{B}_{33}}{\beta_2^2\beta_3},   ~~
b_1  = \frac{\beta_3 x + \beta_3 \mathcal{B}_{23} + \beta_2 \mathcal{B}_{33}}{\beta_2^3},   \\
k_1  = \frac{x}{2\beta_2\beta_3},  ~~
k_2  = \frac{x + 2 \mathcal{B}_{23}}{2\beta_2\beta_3}, ~~
k_3  = \frac{-(\beta_2^2 - 2\beta_1\beta_3) x + 2\beta_3(\beta_1 \mathcal{B}_{23} - \beta_2 (\mathcal{B}_{22} - \mathcal{B}_{13}))}{2\beta_2^3\beta_3}.
\end{split}
\end{equation}
\end{lemma}
\begin{proof}
The derivation process is similar to that of Lemma~\ref{lemma: Case-6}, and the details of the proof are presented in the
supplementary material \cite[Section~II.2]{Wang_sup}.
\end{proof}

 Combining Lemmas~\ref{lemma: Spring-Biquadratic-Network}--\ref{lemma: Case-8} can yield the main theorem of this subsection as follows.

\begin{theorem}  \label{theorem: main theorem 01}
Any  admittance $Y(s)$ in \eqref{eq: specific bicubic admittance} satisfying Assumption~\ref{assumption: 01} is realizable by a one-port series-parallel damper-spring-inerter circuit containing at most six elements as in  Fig.~\ref{fig: Spring-N2}, if and only if
$Y(s)$ satisfies one of the conditions in Lemmas~\ref{lemma: Spring-Biquadratic-Network} and
\ref{lemma: Case-6}--\ref{lemma: Case-8}. Moreover, if Condition~1 in Lemma~\ref{lemma: Spring-Biquadratic-Network} holds, then
the Foster preamble can be utilized to realize $Y(s)$ as the required circuit;  if one of Conditions~2--5 in Lemma~\ref{lemma: Spring-Biquadratic-Network} or
one of the conditions in
Lemmas~\ref{lemma: Case-6}--\ref{lemma: Case-8} holds, then  $Y(s)$ can be realized by one of the circuit configurations in
Figs.~\ref{fig: Spring-Biquadratic-Network} and \ref{fig: Spring-Biquadratic-Network-partial-removal}.
\end{theorem}
\begin{proof}
Suppose that the conditions in Lemma~\ref{lemma: Spring-Biquadratic-Network} do not hold. Then, it follows from Lemma~\ref{lemma: Spring-Other-Network} that the circuit configurations that can realize all the other possible admittances $Y(s)$ in this theorem are shown in
Fig.~\ref{fig: Spring-Biquadratic-Network-partial-removal}. By Lemmas~\ref{lemma: Case-6}--\ref{lemma: Case-8}, one can finally prove this theorem.
\end{proof}

\subsection{Realization as Other Series-Parallel Structures}

This subsection will further investigate the cases when the realization circuits do not belong to the structure in Fig.~\ref{fig: Spring-N2}, which means that the conditions of
Theorem~\ref{theorem: main theorem 01} do not hold.
The main result of this subsection is shown in Theorem~\ref{theorem: main theorem 02}, where the necessary and sufficient condition for the realizability is the union of the conditions in Lemmas~\ref{lemma: 2-1}--\ref{lemma: 5-2}, and the admittance can be realized  by one of the circuit configurations in Figs.~\ref{fig: classes 2 and 3} and \ref{fig: classes 4 and 5}.

\begin{lemma}  \label{lemma: Other Series-Parallel Structures}
For any admittance $Y(s)$ in \eqref{eq: specific bicubic admittance} that satisfies Assumption~\ref{assumption: 01} and does not satisfies the conditions of Theorem~\ref{theorem: main theorem 01},  $Y(s)$ is realizable by a one-port series-parallel damper-spring-inerter circuit containing no more than six elements, if and only if
$Y(s)$ can be realized by one of the circuit configurations in  Figs.~\ref{fig: classes 2 and 3} and \ref{fig: classes 4 and 5}.
\end{lemma}
\begin{proof}
The details of the proof can be referred to Appendix~\ref{appendix: C}.
   In the proof,
a series of realization constraints on the types of elements and constraints on the structures are presented by Lemmas~\ref{lemma: graph constraint} and \ref{lemma: lossless subnetwork N1} when the conditions of Theorem~\ref{theorem: main theorem 01} do not hold, which are further applied to derive the configurations in  Figs.~\ref{fig: classes 2 and 3} and \ref{fig: classes 4 and 5}.
\end{proof}
\begin{figure}[thpb]
      \centering
      \subfigure[]{
      \includegraphics[scale=1.0]{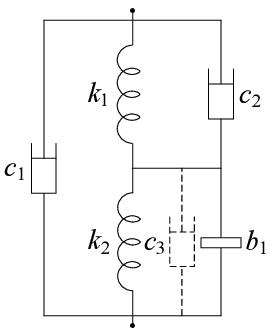}
      \label{fig: 2-1}}  \hspace{-0.5cm}
      \subfigure[]{
      \includegraphics[scale=1.0]{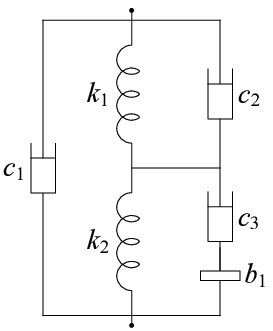}
      \label{fig: 2-2}}  \hspace{-0.5cm}
      \subfigure[]{
      \includegraphics[scale=1.0]{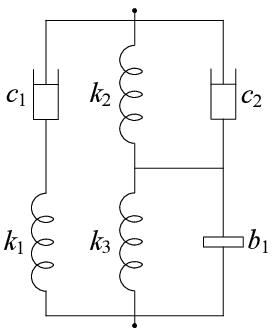}
      \label{fig: 3-1}}   \hspace{-0.5cm}
      \subfigure[]{
      \includegraphics[scale=1.0]{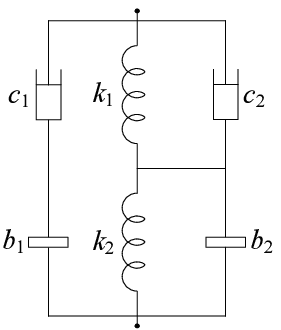}
      \label{fig: 3-2}}
      \caption{The one-port series-parallel circuit configurations containing at most six elements, which can realize the admittance $Y(s)$ in \eqref{eq: specific bicubic admittance} that satisfies Assumption~\ref{assumption: 01} and does not satisfy the conditions of Theorem~\ref{theorem: main theorem 01}. Here, the element value of the damper $c_3$ in (a) satisfies $c_3 \geq 0$ (this damper is open-circuited when $c_3 = 0$), and all the other element values  in (a)--(d) are positive and finite.}
      \label{fig: classes 2 and 3}
\end{figure}

\begin{figure}[thpb]
      \centering
      \subfigure[]{
      \includegraphics[scale=1.0]{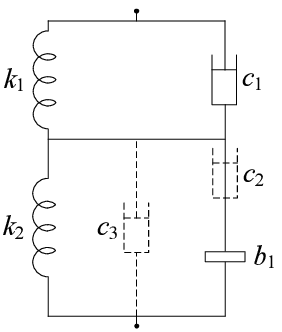}
      \label{fig: 4-1}}  \hspace{-0.5cm}
      \subfigure[]{
      \includegraphics[scale=1.0]{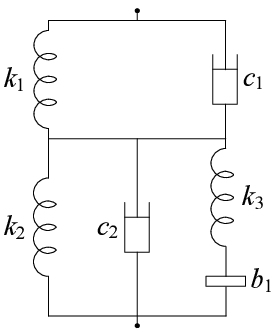}
      \label{fig: 4-2}}  \hspace{-0.5cm}
      \subfigure[]{
      \includegraphics[scale=1.0]{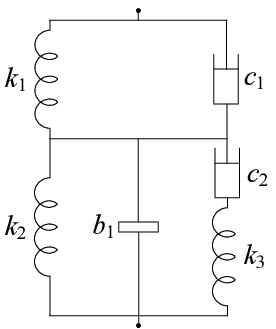}
      \label{fig: 4-3}}   \hspace{-0.5cm}
      \subfigure[]{
      \includegraphics[scale=1.0]{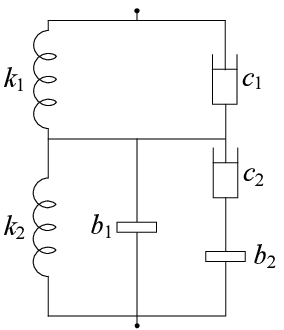}
      \label{fig: 4-4}} \\
      \subfigure[]{
      \includegraphics[scale=1.0]{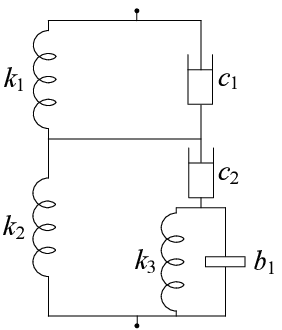}
      \label{fig: 4-5}}   \hspace{-0.5cm}
      \subfigure[]{
      \includegraphics[scale=1.0]{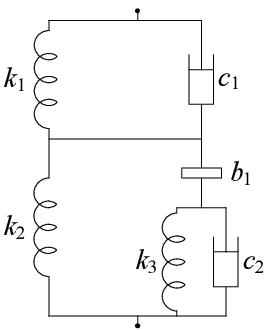}
      \label{fig: 4-6}}    \hspace{-0.5cm}
      \subfigure[]{
      \includegraphics[scale=1.0]{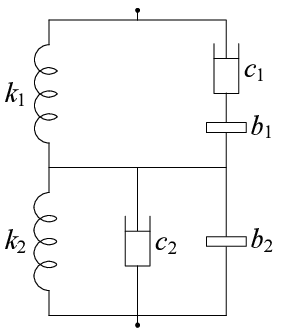}
      \label{fig: 5-1}}    \hspace{-0.5cm}
      \subfigure[]{
      \includegraphics[scale=1.0]{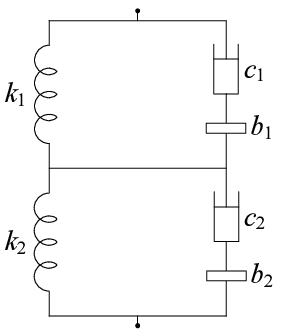}
      \label{fig: 5-2}}
      \caption{The one-port series-parallel circuit configurations containing at most six elements,  which can realize the admittance $Y(s)$ in \eqref{eq: specific bicubic admittance} that satisfies Assumption~\ref{assumption: 01} and does not satisfy the conditions of Theorem~\ref{theorem: main theorem 01}. Here, the element values of the dampers $c_2$ and $c_3$ in (a) satisfy $c_2^{-1} \geq 0$ and $c_3 \geq 0$, where the damper $c_2$ is short-circuited when $c_2^{-1} = 0$, the damper $c_3$ is open-circuited when $c_3 = 0$, and $c_2^{-1} = 0$ and $c_3 = 0$ cannot simultaneously hold;
      all the other element values  in (a)--(h) are positive and finite.}
      \label{fig: classes 4 and 5}
\end{figure}

Then, the following Lemmas~\ref{lemma: 2-1}--\ref{lemma: 5-2}  present the realizability conditions and element value expressions of the circuit configurations in Figs.~\ref{fig: classes 2 and 3} and \ref{fig: classes 4 and 5}.

\begin{lemma}  \label{lemma: 2-1}
Any admittance $Y(s)$ in \eqref{eq: specific bicubic admittance} satisfying Assumption~\ref{assumption: 01}   is realizable by the circuit configuration in
Fig.~\ref{fig: 2-1} ($c_3$ can be zero),   if and only if there exits a positive root $x > 0$ for the equation
\begin{subequations}
\begin{equation} \label{eq: 2-1 equation}
\begin{split}
\beta_3^2 \Delta_\beta x^2 - 2 \beta_3 (\beta_2 \mathcal{M}_{23} -
2 \beta_3 (\alpha_2\beta_1 + \alpha_0\beta_3)) x + (\mathcal{B}_{23}^2
- 4 \mathcal{B}_{13} \mathcal{B}_{33}) = 0,
\end{split}
\end{equation}
such that
\begin{equation} \label{eq: 2-1 condition01}
x \geq \frac{\mathcal{B}_{23}}{\beta_2 \beta_3},
\end{equation}
\begin{equation}  \label{eq: 2-1 condition02}
x > \max \left\{ \frac{\alpha_0}{\beta_1}, - \frac{\mathcal{B}_{23}}{\beta_2 \beta_3} \right\},
\end{equation}
and
\begin{equation}  \label{eq: 2-1 condition03}
\beta_2\beta_3 x^2 - \mathcal{M}_{23} x + 2\alpha_0\alpha_3 \leq 0.
\end{equation}
\end{subequations}
Moreover, the element values can be expressed as
\begin{equation}  \label{eq: 2-1 element values}
\begin{split}
c_1 = \frac{\beta_2\beta_3 x^2 - \mathcal{M}_{23} x + 2\alpha_0\alpha_3}{-2 \beta_3 (\beta_1 x - \alpha_0)},~~
c_2 = \frac{x (\beta_2\beta_3 x + \mathcal{B}_{23})}{2 \beta_3 (\beta_1 x - \alpha_0)},  \\
c_3 = \frac{x (\beta_2\beta_3 x - \mathcal{B}_{23})}{2 \beta_3 (\beta_1 x - \alpha_0)}, ~~
k_1 = x, ~~
 k_2 = \frac{\alpha_0 x}{\beta_1 x - \alpha_0}, ~~
b_1 = \frac{\beta_3 x^2}{\beta_1 x - \alpha_0}.
\end{split}
\end{equation}
\end{lemma}
\begin{proof}
The derivation process is similar to that of Lemma~\ref{lemma: Case-6}, and the details of the proof are presented in the
supplementary material \cite[Section~III.1]{Wang_sup}.
\end{proof}

\begin{lemma}  \label{lemma: 2-2}
Any admittance $Y(s)$ in \eqref{eq: specific bicubic admittance} satisfying Assumption~\ref{assumption: 01}   is realizable by the circuit configuration in
Fig.~\ref{fig: 2-2}, if and only if there exits   positive roots $x > 0$ and $y > 0$ for the two equations
\begin{subequations}
\begin{equation} \label{eq: 2-2 equation01}
\begin{split}
\beta_1 (\beta_1 x - 2 \alpha_0)(\beta_1 x - \alpha_0) y^2 - \beta_2 x (\beta_1 x - \alpha_0)^2 y
+ x^3 (\beta_1^2 \beta_3 x + \beta_2 \tilde{\mathcal{B}}_{11} - \alpha_2 \beta_1^2) = 0
\end{split}
\end{equation}
and
\begin{equation} \label{eq: 2-2 equation02}
\begin{split}
\beta_1  (\beta_1 x - \alpha_0)^2 y^3 - \beta_1 \beta_2 x^2  (\beta_1 x - \alpha_0) y^2
+ \beta_3 x^2 (\beta_1 x + \alpha_0)(\beta_1 x - \alpha_0) y - x^4 (\beta_1 \mathcal{B}_{23} -\beta_2 \mathcal{B}_{13}) = 0,
\end{split}
\end{equation}
such that
\begin{equation} \label{eq: 2-2 condition01}
0 < (\beta_1 x - \alpha_0)y < \beta_2 x^2
\end{equation}
and
\begin{equation}  \label{eq: 2-2 condition02}
\begin{split}
0 <  (\beta_1 x - \alpha_0)^2 y^2 - & \beta_2  x^2  (\beta_1 x - \alpha_0)y + \beta_1\beta_3 x^4 < \frac{\alpha_3 \beta_1 x^4}{y}.
\end{split}
\end{equation}
\end{subequations}
Moreover, the element values can be expressed as
\begin{equation}  \label{eq: 2-2 element values}
\begin{split}
c_1 =  \frac{\alpha_3 \beta_1 x^4 - y ((\beta_1 x - \alpha_0)^2 y^2  - \beta_2  x^2  (\beta_1 x - \alpha_0) y + \beta_1\beta_3 x^4)}{\beta_1 \beta_3 x^4},  \\
 c_2 = y,  ~~ c_3 = \frac{(\beta_1 x - \alpha_0)^2 y^2  - \beta_2  x^2 (\beta_1 x - \alpha_0) y  + \beta_1\beta_3 x^4}{(\beta_1 x - \alpha_0) (\beta_2 x^2 - (\beta_1 x - \alpha_0) y)}, \\
k_1 = x, ~~
k_2 = \frac{\alpha_0 x}{\beta_1 x - \alpha_0},  ~~
b_1 = \frac{(\beta_1 x - \alpha_0)^2 y^2  - \beta_2  x^2  (\beta_1 x - \alpha_0) y + \beta_1\beta_3 x^4}{\beta_1 x^2 (\beta_1 x - \alpha_0)}.
\end{split}
\end{equation}
\end{lemma}
\begin{proof}
The derivation process is similar to that of Lemma~\ref{lemma: Case-6}, and the details of the proof are presented in the
supplementary material \cite[Section~III.2]{Wang_sup}.
\end{proof}

\begin{lemma}  \label{lemma: 3-1}
Any admittance $Y(s)$ in \eqref{eq: specific bicubic admittance} satisfying Assumption~\ref{assumption: 01}   is realizable by the circuit configuration in
Fig.~\ref{fig: 3-1}, if and only if
\begin{subequations}
\begin{equation} \label{eq: 3-1 condition00}
\beta_1 \tilde{\mathcal{B}}_{13} + 4 \beta_3 \mathcal{B}_{12} \geq 0,
\end{equation}
and there exits a positive root $x > 0$ for the equation
\begin{equation} \label{eq: 3-1 equation}
\begin{split}
(2 \beta_2^2 + \beta_1 \beta_3) x^2  - \beta_2 (3 \beta_2 \mathcal{B}_{23}  + \alpha_2(\beta_2^2 + 2\beta_1\beta_3)) x
+ \beta_2^2 (\mathcal{B}_{23}^2 + \alpha_2 \beta_2 \mathcal{B}_{23} + \alpha_2^2 \beta_1 \beta_3) = 0,
\end{split}
\end{equation}
such that
\begin{equation} \label{eq: 3-1 condition01}
\mathcal{B}_{23} < x < \min \{ \tilde{\mathcal{B}}_{13}, \alpha_2 \beta_2 \}
\end{equation}
and
\begin{equation}  \label{eq: 3-1 condition02}
x = \frac{1}{2}\left( \tilde{\mathcal{B}}_{13} \pm \sqrt{\tilde{\mathcal{B}}_{13}^2 + 4
\mathcal{B}_{12} \tilde{\mathcal{B}}_{33}} \right).
\end{equation}
\end{subequations}
Moreover, the element values can be expressed as
\begin{equation}  \label{eq: 3-1 element values}
\begin{split}
c_1 = \frac{x - \mathcal{B}_{23}}{\beta_1 \beta_3}, ~~ c_2 = \frac{b_1 k_1^2 + c_1^2 k_2 + c_1^2 k_3}{c_1 k_1},  \\
k_1 = \frac{\alpha_2 \beta_2 - x}{\beta_2 \beta_3}, ~~
k_2 = \frac{x}{\beta_2 \beta_3}, ~~ k_3 = \frac{\tilde{\mathcal{B}}_{13} - x}{\beta_2 \beta_3},
~~ b_1 = \frac{\alpha_3}{\beta_2}.
\end{split}
\end{equation}
\end{lemma}
\begin{proof}
The details of the proof are presented in Appendix~\ref{appendix: D}.
\end{proof}

\begin{lemma}  \label{lemma: 3-2}
Any admittance $Y(s)$ in \eqref{eq: specific bicubic admittance} satisfying Assumption~\ref{assumption: 01}   is realizable by the circuit configuration in
Fig.~\ref{fig: 3-2}, if and only if
\begin{subequations}
\begin{equation}  \label{eq: 3-2 condition00}
\Delta_y \geq 0,
\end{equation}
\begin{equation}  \label{eq: 3-2 condition0-0}
\alpha_2 \tilde{\mathcal{B}}_{11} + \alpha_1 \mathcal{B}_{13} > 0,
\end{equation}
and there exits a positive root $x > 0$ for the equation
\begin{equation} \label{eq: 3-2 equation}
\begin{split}
 \beta_3 \tilde{\mathcal{B}}_{11} &(\alpha_2 \beta_2 + \alpha_1 \beta_3) x^2   - \tilde{\mathcal{B}}_{23} \tilde{\mathcal{B}}_{11} (\alpha_2 \tilde{\mathcal{B}}_{11}
+ \alpha_1 \mathcal{B}_{13}) x + \alpha_1^2\beta_3 (\alpha_2 \tilde{\mathcal{B}}_{11} + \alpha_1 \mathcal{B}_{13})^2 = 0,
\end{split}
\end{equation}
such that
\begin{equation} \label{eq: 3-2 condition01}
x = \frac{\alpha_2 \tilde{\mathcal{B}}_{11} + \alpha_1 \mathcal{B}_{13}}{2 \beta_1 \beta_3} \cdot \left( \beta_2 \pm \sqrt{\Delta_y} \right).
\end{equation}
\end{subequations}
Moreover, the element values can be expressed as
\begin{equation}  \label{eq: 3-2 element values}
\begin{split}
c_1 = \frac{x}{\tilde{\mathcal{B}}_{11}}, ~~
c_2 = \frac{b_1^2 k_1 + b_1^2 k_2 + b_2 c_1^2}{b_1 c_1},  ~~
k_1 = \frac{\alpha_0 \alpha_1}{\tilde{\mathcal{B}}_{11}}, \\
k_2 = \frac{\alpha_1}{\beta_2}, ~~
b_1 = \frac{\alpha_2 \tilde{\mathcal{B}}_{11}  + \alpha_1 \mathcal{B}_{13}}{\beta_1 \tilde{\mathcal{B}}_{11}}, ~~
b_2 = \frac{\alpha_1^2 \beta_3}{\beta_2 \tilde{\mathcal{B}}_{11}}.
\end{split}
\end{equation}
\end{lemma}
\begin{proof}
The derivation process is similar to that of Lemma~\ref{lemma: 3-1}, and the details of the proof are presented in the
supplementary material \cite[Section~III.3]{Wang_sup}.
\end{proof}

\begin{lemma} \label{lemma: 4-1}
Any admittance $Y(s)$ in \eqref{eq: specific bicubic admittance} satisfying Assumption~\ref{assumption: 01}   is realizable by the circuit configuration in
Fig.~\ref{fig: 4-1} ($c_2$ or $c_3$ can be infinite),   if and only if at least one of the following two conditions holds:
\begin{enumerate}
  \item[1.] there exists a negative root $y < 0$ for the equation
\begin{equation}  \label{eq: 4-1 condition 01 equation 01}
\alpha_3 y^3 + \alpha_2 y^2 + \alpha_1 y + \alpha_0 = 0,
\end{equation}
such that
$\Gamma_1 := \beta_3 y^2 + \beta_2 y + \beta_1$,
$\Gamma_2 :=  -2\alpha_3 y^3 - \alpha_2 y^2 + \alpha_0$,
$\Gamma_3 :=  2 \tilde{\mathcal{B}}_{13} y^2 + \tilde{\mathcal{M}}_{12} y - \mathcal{B}_{12}$,
$\Gamma_4 :=  - \tilde{\mathcal{B}}_{33} y^3 + \mathcal{B}_{33} y^2 + \tilde{\mathcal{B}}_{13} y - \mathcal{B}_{12}$, and
$\Gamma_5 := \beta_1 \beta_3 y^2 (\alpha_3 y + \alpha_2)^2 + (\tilde{\mathcal{B}}_{13} y - \mathcal{B}_{13})^2 - \beta_2 y (\alpha_3 y + \alpha_2) (\tilde{\mathcal{B}}_{13} y + \mathcal{B}_{13}) + \mathcal{B}_{12} \tilde{\mathcal{B}}_{23} y$, and
$\Gamma_6 := -y (2 \alpha_3 y + \alpha_2)(\tilde{\mathcal{B}}_{13} y^2 + \alpha_2 \beta_1 y - \mathcal{B}_{12}) - \alpha_0 (\tilde{\mathcal{B}}_{13} y
- \mathcal{B}_{13})$  have the same sign, where $\Gamma_k$ is nonzero for $k = 1, 2, ..., 5$  and $\Gamma_6$ can be zero;
  \item[2.] $\mathcal{M}_{23}^2 + 8 \mathcal{B}_{13} \tilde{\mathcal{B}}_{23} \geq 0$, and
  there exists a positive root $x > 0$ for the equation
\begin{subequations}
\begin{equation}  \label{eq: 4-1 condition 02 equation 01}
\beta_3^3 x^3 - \alpha_2 \beta_3^2 x^2 + \alpha_1 \tilde{\mathcal{B}}_{33} x - \alpha_0 \alpha_3^2 = 0,
\end{equation}
such that
\begin{equation}  \label{eq: 4-1 condition 02 equation 02}
x > \max \left\{ \frac{\alpha_0}{\beta_1}, \frac{\alpha_0 \alpha_3}{\alpha_1 \beta_3}  \right\}
\end{equation}
and
\begin{equation}  \label{eq: 4-1 condition 02 equation 03}
x = \frac{\mathcal{M}_{23} \pm \sqrt{\mathcal{M}_{23}^2 + 8 \mathcal{B}_{13} \tilde{\mathcal{B}}_{23}}}{2 \beta_2 \beta_3}.
\end{equation}
\end{subequations}
\end{enumerate}
Moreover, if Condition~1 holds, then the element values can be expressed as
\begin{equation}  \label{eq: 4-1 element values}
\begin{split}
c_1 = -\frac{\Gamma_2}{y \Gamma_1}, ~~  c_2 = - \frac{\Gamma_2^2 \Gamma_5}{y \Gamma_3^2 \Gamma_4}, ~~  c_3 = - \frac{\Gamma_2 \Gamma_6}{y \Gamma_3^2}, ~~
k_1 = \frac{\Gamma_2}{\Gamma_1}, ~~  k_2 = -\frac{\alpha_0 \Gamma_2}{y \Gamma_3}, ~~  b_1 = -\frac{\Gamma_2^2 \Gamma_5}{y \Gamma_3^3};
\end{split}
\end{equation}
if  Condition~2 holds, then the element values can be expressed as
\begin{equation}  \label{eq: 4-1 element values 02}
\begin{split}
c_1 =  \frac{\alpha_3}{\beta_3}, ~~  c_2 = \infty, ~~ c_3 = \frac{\alpha_1 \beta_3 x - \alpha_0 \alpha_3}{\beta_3 (\beta_1 x - \alpha_0)}, ~~
k_1 = x, ~~ k_2 = \frac{\alpha_0 x}{\beta_1 x - \alpha_0}, ~~
b_1 = \frac{\beta_3 x^2}{\beta_1 x - \alpha_0}.
\end{split}
\end{equation}
\end{lemma}
\begin{proof}
The derivation process is similar to that of Lemma~\ref{lemma: Case-6}, and the details of the proof are presented in the
supplementary material \cite[Section~IV.1]{Wang_sup}.
\end{proof}

\begin{lemma} \label{lemma: 4-2}
Any admittance $Y(s)$ in \eqref{eq: specific bicubic admittance} satisfying Assumption~\ref{assumption: 01}   is realizable by the circuit configuration in
Fig.~\ref{fig: 4-2}, if and only if
\begin{subequations}
\begin{equation}  \label{eq: 4-2 condition 01}
\Delta_\alpha > 0,
\end{equation}
\begin{equation}  \label{eq: 4-2 condition 02}
\mathcal{B}_{23} < \min \left\{  \frac{\alpha_3 \mathcal{B}_{12}}{\alpha_1},
\frac{\beta_3 \mathcal{B}_{12}}{\beta_1}  \right\},
\end{equation}
and
\begin{equation}  \label{eq: 4-2 condition 03}
\mathcal{B}_{23}^3 + \alpha_2\beta_2 \mathcal{B}_{23}^2 + \alpha_1 \beta_2 \tilde{\mathcal{B}}_{23} \mathcal{B}_{23} - \tilde{\mathcal{B}}_{23}^2 \mathcal{B}_{12} = 0.
\end{equation}
\end{subequations}
Moreover, the element values can be expressed as
\begin{equation}   \label{eq: 4-2 element values}
\begin{split}
c_1  = \frac{- \alpha_3 (\alpha_1 \mathcal{B}_{23} - \alpha_3 \mathcal{B}_{12})}{\mathcal{B}_{23}^2}, ~~
c_2  = \frac{c_1(b_1k_1^2k_2 + b_1k_1^2k_3 + c_1^2k_2k_3)}{k_1(b_1k_1^2+c_1^2k_3)}, ~~
k_1  = \frac{\alpha_1 \mathcal{B}_{23} - \alpha_3 \mathcal{B}_{12}}{\beta_2 \mathcal{B}_{23}}, \\
k_2  = \frac{\alpha_0(\alpha_1 \mathcal{B}_{23} - \alpha_3 \mathcal{B}_{12})}{\alpha_1 (\beta_1 \mathcal{B}_{23} - \beta_3 \mathcal{B}_{12})},    ~~
k_3  =  \frac{\Delta_\alpha (\alpha_1 \mathcal{B}_{23} - \alpha_3 \mathcal{B}_{12})}{\alpha_1\alpha_3 (\beta_1 \mathcal{B}_{23} - \beta_3 \mathcal{B}_{12})}, ~~
b_1  =   \frac{\Delta_\alpha (\alpha_1 \mathcal{B}_{23}
- \alpha_3 \mathcal{B}_{12})}{\alpha_1^2 (\beta_1 \mathcal{B}_{23}
- \beta_3 \mathcal{B}_{12})}.
\end{split}
\end{equation}
\end{lemma}
\begin{proof}
The derivation process is similar to that of Lemma~\ref{lemma: 3-1}, and the details of the proof are presented in the
supplementary material \cite[Section~IV.2]{Wang_sup}.
\end{proof}

\begin{lemma} \label{lemma: 4-3}
Any admittance $Y(s)$ in \eqref{eq: specific bicubic admittance} satisfying Assumption~\ref{assumption: 01}   is realizable by the circuit configuration in
Fig.~\ref{fig: 4-3}, if and only if
\begin{subequations}
\begin{equation}   \label{eq: 4-3 condition00}
\Delta_\alpha > 0,
\end{equation}
and
there exists a positive root $x > 0$ for the equation
\begin{equation} \label{eq: 4-3 equation}
\begin{split}
\beta_3^2 (\alpha_3 \tilde{\mathcal{B}}_{12}      - \alpha_2^2  \beta_2) x^2
+ \alpha_2 \beta_3^2 \Delta_{\alpha} x   + \alpha_0 \alpha_3^2 (\mathcal{B}_{23} - \alpha_2 \beta_2) = 0,
\end{split}
\end{equation}
such that
\begin{equation} \label{eq: 4-3 condition01}
x > \frac{\alpha_0}{\beta_1}
\end{equation}
and
\begin{equation} \label{eq: 4-3 condition02}
\beta_3^3 x^3 - \alpha_2 \beta_3^2 x^2 + \alpha_1 \tilde{\mathcal{B}}_{33} x - \alpha_0 \alpha_3^2 = 0.
\end{equation}
\end{subequations}
Moreover, the element values can be expressed as
\begin{equation}  \label{eq: 4-3 element values}
\begin{split}
c_1 = \frac{\alpha_3}{\beta_3}, ~~
c_2 = \frac{c_1 k_3 (b_1 k_1^2 + c_1^2 k_2)}{k_1 (b_1 k_1^2 + c_1^2 k_2 + c_1^2 k_3)}, ~~
k_1 = x, \\
k_2 = \frac{\alpha_0 x}{\beta_1 x - \alpha_0}, ~~
k_3 = \frac{\Delta_\alpha x}{\alpha_3 (\beta_1 x - \alpha_0)}, ~~
b_1 = \frac{\alpha_2 x}{\beta_1 x - \alpha_0}.
\end{split}
\end{equation}
\end{lemma}
\begin{proof}
The derivation process is similar to that of Lemma~\ref{lemma: 3-1}, and the details of the proof are presented in the
supplementary material \cite[Section~IV.3]{Wang_sup}.
\end{proof}

\begin{lemma} \label{lemma: 4-4}
Any admittance $Y(s)$ in \eqref{eq: specific bicubic admittance} satisfying Assumption~\ref{assumption: 01}   is realizable by the circuit configuration in
Fig.~\ref{fig: 4-4}, if and only if
\begin{subequations}
\begin{equation}  \label{eq: 4-4 condition00}
\Delta_\alpha > 0,
\end{equation}
and
there exists a positive root $x > 0$ for the equation
\begin{equation} \label{eq: 4-4 equation}
\begin{split}
\beta_3^2 (\alpha_2 \tilde{\mathcal{B}}_{11} +   \alpha_1 \mathcal{B}_{13}) x^2
 - \beta_3 \mathcal{B}_{13} \Delta_\alpha x  + \alpha_0 \alpha_3 (\alpha_1 \mathcal{B}_{23}
+ \alpha_3 \mathcal{B}_{12}) = 0,
\end{split}
\end{equation}
such that
\begin{equation} \label{eq: 4-4 condition01}
x > \frac{\alpha_0}{\beta_1}
\end{equation}
and
\begin{equation} \label{eq: 4-4 condition02}
\beta_3^3 x^3 - \alpha_2 \beta_3^2 x^2 + \alpha_1 \tilde{\mathcal{B}}_{33} x - \alpha_0 \alpha_3^2 = 0.
\end{equation}
\end{subequations}
Moreover, the element values can be expressed as
\begin{equation}  \label{eq: 4-4 element values}
\begin{split}
c_1 = \frac{\alpha_3}{\beta_3}, ~~
c_2 = \frac{b_2 k_1 (b_1 k_1^2 + c_1^2 k_2)}{c_1 (b_1 k_1^2 + b_2 k_1^2 + c_1^2 k_2)}, ~~
k_1 = x,  \\
k_2 = \frac{\alpha_0 x}{\beta_1 x - \alpha_0}, ~~
b_1 = \frac{\alpha_0 \alpha_3 x}{\alpha_1 (\beta_1 x - \alpha_0)}, ~~
b_2 = \frac{\Delta_\alpha x}{\alpha_1 (\beta_1 x - \alpha_0)}.
\end{split}
\end{equation}
\end{lemma}
\begin{proof}
The derivation process is similar to that of Lemma~\ref{lemma: 3-1}, and the details of the proof are presented in the
supplementary material \cite[Section~IV.4]{Wang_sup}.
\end{proof}

\begin{lemma} \label{lemma: 4-5}
Any admittance $Y(s)$ in \eqref{eq: specific bicubic admittance} satisfying Assumption~\ref{assumption: 01}   is realizable by the circuit configuration in
Fig.~\ref{fig: 4-5}, if and only if
\begin{subequations}
\begin{equation}  \label{eq: 4-5 condition01}
\tilde{\mathcal{B}}_{12} > 0,
\end{equation}
\begin{equation}   \label{eq: 4-5 condition02}
\beta_2 (\beta_1 \Delta_\alpha + \alpha_2 \mathcal{B}_{12}) - \tilde{\mathcal{B}}_{12}^2 > 0,
\end{equation}
and
\begin{equation}   \label{eq: 4-5 condition03}
\begin{split}
(\mathcal{B}_{23}+\alpha_2\beta_2)(\beta_1 \Delta_\alpha + \alpha_2 \mathcal{B}_{12})^2   - \alpha_2 \tilde{\mathcal{B}}_{12}^2 (\beta_1 \Delta_\alpha + \alpha_2 \mathcal{B}_{12})  + \alpha_0\alpha_3 \tilde{\mathcal{B}}_{12}^3 = 0.
\end{split}
\end{equation}
\end{subequations}
Moreover, the element values can be expressed as
\begin{equation}   \label{eq: 4-5 element values}
\begin{split}
c_1 = \frac{\Delta_\alpha (\beta_1 \Delta_\alpha + \alpha_2 \mathcal{B}_{12})^2}{\tilde{\mathcal{B}}_{12} (\alpha_0\alpha_2 \tilde{\mathcal{B}}_{12}^2 + (\beta_1 \Delta_\alpha + \alpha_2 \mathcal{B}_{12})^2)}, ~~
c_2 = \frac{c_1k_2(b_1 k_1^2 + c_1^2 k_3)}{k_1 (b_1 k_1^2 + c_1^2 k_2 + c_1^2 k_3)}, \\
k_1 = \frac{\alpha_0\Delta_\alpha (\beta_1 \Delta_\alpha + \alpha_2 \mathcal{B}_{12})}{\alpha_0\alpha_2 \tilde{\mathcal{B}}_{12}^2 +
(\beta_1 \Delta_\alpha + \alpha_2 \mathcal{B}_{12})^2}, ~~
k_2 = \frac{\Delta_\alpha (\beta_1 \Delta_\alpha + \alpha_2 \mathcal{B}_{12})}{\alpha_2(\beta_2 (\beta_1 \Delta_\alpha + \alpha_2 \mathcal{B}_{12}) - \tilde{\mathcal{B}}_{12}^2)},   \\
k_3 = \frac{\alpha_0\alpha_3 (\beta_1 \Delta_\alpha + \alpha_2 \mathcal{B}_{12})}{\alpha_2(\beta_2 (\beta_1 \Delta_\alpha + \alpha_2 \mathcal{B}_{12}) - \tilde{\mathcal{B}}_{12}^2)}, ~~
b_1 = \frac{\alpha_3 (\beta_1 \Delta_\alpha + \alpha_2 \mathcal{B}_{12})}{\beta_2 (\beta_1 \Delta_\alpha + \alpha_2 \mathcal{B}_{12}) - \tilde{\mathcal{B}}_{12}^2}.
\end{split}
\end{equation}
\end{lemma}
\begin{proof}
The derivation process is similar to that of Lemma~\ref{lemma: 3-1}, and the details of the proof are presented in the
supplementary material \cite[Section~IV.5]{Wang_sup}.
\end{proof}

\begin{lemma} \label{lemma: 4-6}
Any admittance $Y(s)$ in \eqref{eq: specific bicubic admittance} satisfying Assumption~\ref{assumption: 01}   is realizable by the circuit configuration in
Fig.~\ref{fig: 4-6}, if and only if
\begin{subequations}
\begin{equation}  \label{eq: 4-6 condition0-1}
\Delta_\alpha > 0,
\end{equation}
\begin{equation}  \label{eq: 4-6 condition00}
\tilde{\mathcal{B}}_{11} > 0,
\end{equation}
and there exists a positive root $x > 0$ for the equation
\begin{equation}  \label{eq: 4-6 equation}
\begin{split}
\tilde{\mathcal{B}}_{11} (\beta_1 &\Delta_\alpha + \alpha_1 \mathcal{B}_{13}) x^2  - \alpha_0 \tilde{\mathcal{B}}_{11} \Delta_\alpha x + \alpha_0^3 \alpha_1 \alpha_3 = 0,
\end{split}
\end{equation}
such that
\begin{equation} \label{eq: 4-6 condition01}
\max \left\{ \frac{\alpha_0}{\beta_1}, \frac{\alpha_0^2 \alpha_3}{\alpha_2 \tilde{\mathcal{B}}_{11}} \right\}
< x < \frac{\alpha_0 \alpha_1}{\tilde{\mathcal{B}}_{11}}
\end{equation}
and
\begin{equation} \label{eq: 4-6 condition02}
\alpha_3 y^3 + \alpha_2 y^2 + \alpha_1 y + \alpha_0 = 0,
\end{equation}
where
\begin{equation} \label{eq: 4-6 condition03}
y = - \sqrt{\frac{\alpha_0 (\tilde{\mathcal{B}}_{11} x - \alpha_0\alpha_1)}{\alpha_0^2 \alpha_3 - \alpha_2 \tilde{\mathcal{B}}_{11} x}}.
\end{equation}
\end{subequations}
Moreover, the element values can be expressed as
\begin{equation}  \label{eq: 4-6 element values}
\begin{split}
c_1 = -\frac{x}{y}, ~~ c_2 = \frac{c_1 (b_1 k_1^2 k_2 + b_1 k_1^2 k_3 + c_1^2 k_2 k_3)}{k_1 (b_1 k_1^2 + c_1^2 k_2)}, ~~
k_1 = x, \\
k_2 = \frac{\alpha_0 x}{\beta_1 x - \alpha_0}, ~~
k_3 = \frac{\alpha_0^2 \alpha_3 x}{\Delta_\alpha (\beta_1 x - \alpha_0)}, ~~
b_1 = \frac{\alpha_0 \alpha_3 x}{\alpha_1 (\beta_1 x - \alpha_0)}.
\end{split}
\end{equation}
\end{lemma}
\begin{proof}
The derivation process is similar to that of Lemma~\ref{lemma: 3-1}, and the details of the proof are presented in the
supplementary material \cite[Section~IV.6]{Wang_sup}.
\end{proof}

\begin{lemma} \label{lemma: 5-1}
Any admittance $Y(s)$ in \eqref{eq: specific bicubic admittance} satisfying Assumption~\ref{assumption: 01}   is realizable by the circuit configuration in
Fig.~\ref{fig: 5-1}, if and only if
\begin{subequations}
\begin{equation}   \label{eq: 5-1 condition00}
\tilde{\mathcal{B}}_{11} > 0,
\end{equation}
and there are positive roots $x > 0$ and $y > 0$ for the equations
\begin{equation}  \label{eq: 5-1 equation00}
\begin{split}
\alpha_3 (\beta_1^2 \beta_3 x^2  + (2 \beta_1 \mathcal{B}_{13} &- \beta_2 \tilde{\mathcal{B}}_{11}) x   - \alpha_0 \mathcal{B}_{13} + \alpha_1 \tilde{\mathcal{B}}_{11} ) y  \\
&- (\beta_1 x - \alpha_0)  (\beta_1 \beta_3^2 x^2 - \beta_3 \tilde{\mathcal{M}}_{12} x
   - \alpha_2 \mathcal{B}_{13} + \alpha_3 \tilde{\mathcal{B}}_{11}     )
    = 0
\end{split}
\end{equation}
and
\begin{equation}  \label{eq: 5-1 equation01}
\begin{split}
(\beta_2 x - \alpha_1) \Psi_2 y^2 - (\beta_1 x - \alpha_0) \Psi_1 y + \alpha_0 (\beta_1 x - \alpha_0)^2
\Psi_0 = 0,
\end{split}
\end{equation}
such that
\begin{equation} \label{eq: 5-1 condition03}
\beta_1  x - \alpha_0 > 0,
\end{equation}
\begin{equation}  \label{eq: 5-1 condition04}
(\beta_2 x - \alpha_1)y + \beta_1 x - \alpha_0 > 0,
\end{equation}
and
\begin{equation} \label{eq: 5-1 equation02}
\begin{split}
(\beta_2 x - \alpha_1) \Upsilon_2 y^2 + (\beta_2 x - \alpha_1) (\beta_1 x - \alpha_0) \Upsilon_1 y + (\beta_1 x - \alpha_0)^2 \Upsilon_0 = 0.
\end{split}
\end{equation}
\end{subequations}
Moreover, the element values can be expressed as
\begin{equation}  \label{eq: 5-1 element values}
\begin{split}
c_1 = \frac{\alpha_3}{\beta_3}, ~~ c_2 = \frac{\tilde{\mathcal{B}}_{11} x^2}{(\beta_1 x - \alpha_0)^2}, ~~   k_1 = x, ~~
k_2 = \frac{\alpha_0 x}{\beta_1 x - \alpha_0}, \\
b_1 = \frac{\alpha_3 ((\beta_2 x - \alpha_1)y + \beta_1 x - \alpha_0)}{\beta_3 y (\beta_1  x - \alpha_0)}, ~~
b_2 = \frac{\beta_3 x^2}{(\beta_2 x - \alpha_1)y + \beta_1 x - \alpha_0}.
\end{split}
\end{equation}
Here, $\Psi_2  := \beta_3 (\alpha_0 \beta_2^2 - \alpha_1 \beta_1 \beta_2 + \alpha_2 \beta_1^2) x^3 - (\alpha_0 \alpha_1 \beta_2 \beta_3 + 2 \alpha_0 \alpha_2 \beta_1 \beta_3 + \alpha_0 \alpha_3 \beta_1 \beta_2 - \alpha_1^2 \beta_1 \beta_3) x^2 + \alpha_0 (\alpha_0 \alpha_2 \beta_3 + \alpha_0 \alpha_3 \beta_2 + \alpha_1 \alpha_3 \beta_1) x + \alpha_0^2 \alpha_1 \alpha_3$,
$\Psi_1  := \beta_1^2 \beta_3^2 x^4 - \beta_3 (2 \alpha_0 \beta_1 \beta_3 + 2 \alpha_0 \beta_2^2 - \alpha_1 \beta_1 \beta_2 + \alpha_2 \beta_1^2) x^3 + (\alpha_0^2 \beta_3^2 + 3 \alpha_0 \alpha_1 \beta_2 \beta_3 + 2 \alpha_0 \alpha_2 \beta_1 \beta_3 + 2 \alpha_0 \alpha_3 \beta_1 \beta_2 - \alpha_1^2 \beta_1 \beta_3) x^2 - \alpha_0 (\alpha_0 \alpha_2 \beta_3 + 2 \alpha_0 \alpha_3 \beta_2 + \alpha_1^2 \beta_3 + 2 \alpha_1 \alpha_3 \beta_1) x + 2 \alpha_0^2 \alpha_1 \alpha_3$,
$\Psi_0  := \beta_2 \beta_3 x^2 - \mathcal{M}_{23} x + \alpha_0 \alpha_3$,
$\Upsilon_2 := \beta_1^2 \beta_3^2 x^4 + (\alpha_0 \beta_2^2 \beta_3  -2 \alpha_0 \beta_1 \beta_3^2  - \alpha_1 \beta_1 \beta_2 \beta_3 - \alpha_3 \beta_1^2 \beta_2) x^3 + (\alpha_0^2 \beta_3^2 - \alpha_0 \alpha_1 \beta_2 \beta_3 + 2 \alpha_0 \alpha_3 \beta_1 \beta_2 + \alpha_1^2 \beta_1 \beta_3 + \alpha_1 \alpha_3 \beta_1^2) x^2 - \alpha_0 \alpha_3 (\alpha_0 \beta_2 + 2 \alpha_1 \beta_1) x + \alpha_0^2 \alpha_1 \alpha_3$,
$\Upsilon_1 := \beta_1 \beta_2 \beta_3 x^3 -
(2 \alpha_1 \beta_1 \beta_3 + 2 \alpha_3 \beta_1^2 -\alpha_0 \beta_2 \beta_3) x^2 + 4 \alpha_0 \alpha_3 \beta_1 x - 2 \alpha_0^2 \alpha_3$, and
$\Upsilon_0 := \beta_1 \beta_2 \beta_3 x^3 - \beta_1 \mathcal{M}_{23} x^2 + 2 \alpha_0 \alpha_3 \beta_1 x - \alpha_0^2 \alpha_3$.
\end{lemma}
\begin{proof}
The details of the proof are presented in  \cite[Section~IV.7]{Wang_sup}.
\end{proof}

\begin{lemma} \label{lemma: 5-2}
Any admittance $Y(s)$ in \eqref{eq: specific bicubic admittance} satisfying Assumption~\ref{assumption: 01}   is realizable by the circuit configuration in
Fig.~\ref{fig: 5-2}, if and only if
\begin{subequations}
\begin{equation}   \label{eq: 5-2 condition00}
\tilde{\mathcal{B}}_{11} = 0,
\end{equation}
and there are positive roots $x > 0$, $y > 0$, and $z > 0$ for the equations
\begin{equation}  \label{eq: 5-2 equation01}
\alpha_0 x \Theta_0 z^2 -  y \Theta_1 z + \alpha_3^2 x^2 y^2 (\beta_1 y - \alpha_0) = 0,
\end{equation}
\begin{equation}  \label{eq: 5-2 equation02}
\alpha_0 \Theta_0 z^2 + x \Theta_2  z + \alpha_0 \alpha_3 x y^2 (\beta_3 x - \alpha_3) = 0,
\end{equation}
and
\begin{equation}  \label{eq: 5-2 equation03}
x (\beta_1 y - \alpha_0) \Theta_0 z^2 - y^2  \Theta_3 z + \alpha_3^2 x^2 y^2 (\beta_1 y - \alpha_0)  = 0,
\end{equation}
such that
\begin{equation} \label{eq: 5-2 condition03}
\beta_3 x - \alpha_3 > 0,
\end{equation}
\begin{equation}  \label{eq: 5-2 condition04}
\beta_1 y - \alpha_0 > 0,
\end{equation}
and
\begin{equation} \label{eq: 5-2 condition05}
\Theta_0 > 0.
\end{equation}
\end{subequations}
Moreover, the element values can be expressed as
\begin{equation}  \label{eq: 5-2 element values}
\begin{split}
c_1 = x, ~~ c_2 = \frac{\alpha_3 x}{\beta_3 x - \alpha_3}, ~~
k_1 = y, ~~
 k_2 = \frac{\alpha_0 y}{\beta_1 y - \alpha_0}, ~~
b_1 = z, ~~ b_2 = \frac{\alpha_3^2 x y^2}{z \Theta_0}.
\end{split}
\end{equation}
Here, $\Theta_0 := -  \tilde{\mathcal{B}}_{12} x y + \alpha_3 y (\beta_1 y - 2 \alpha_0) + \alpha_0\alpha_2 x$,
$\Theta_1 := \alpha_0 (\alpha_2 x - \alpha_3 y)^2 - \tilde{\mathcal{B}}_{12} x y (\alpha_2 x - \alpha_3 y)  + \alpha_1 \alpha_3 x^2(\beta_1 y - \alpha_0)$, $\Theta_2 := x y (\beta_1 \Delta_\alpha  + \alpha_1 \mathcal{B}_{13}) - \alpha_0 \Delta_\alpha x  - \alpha_1 \alpha_3 y (\beta_1 y - 2 \alpha_0)$,
and $\Theta_3 := x^2 y
(\beta_1 \mathcal{B}_{33} - \beta_3 \mathcal{B}_{13})  - \alpha_0  \mathcal{B}_{33} x^2 + 2  \alpha_3 \mathcal{B}_{13} x y + \alpha_3 \tilde{\mathcal{B}}_{13} y^2$.
\end{lemma}
\begin{proof}
The details of the proof are presented in  \cite[Section~IV.8]{Wang_sup}.
\end{proof}


\begin{theorem}    \label{theorem: main theorem 02}
For any admittance $Y(s)$ in \eqref{eq: specific bicubic admittance} that satisfies Assumption~\ref{assumption: 01} and does not satisfy the conditions of Theorem~\ref{theorem: main theorem 01}, $Y(s)$   is realizable by a one-port series-parallel damper-spring-inerter circuit containing at most six elements, if and only if
one of the conditions in Lemmas~\ref{lemma: 2-1}--\ref{lemma: 5-2} holds. Moreover,
the realizability conditions correspond to the circuit configurations in Figs.~\ref{fig: classes 2 and 3} and \ref{fig: classes 4 and 5}.
\end{theorem}
\begin{proof}
By Lemma~\ref{lemma: Other Series-Parallel Structures},  the circuit configurations
that can realize all the possible cases of $Y(s)$ in this theorem are shown in
Figs.~\ref{fig: classes 2 and 3} and \ref{fig: classes 4 and 5}.
Since the realizability conditions of these circuit configurations are presented in Lemmas~\ref{lemma: 2-1}--\ref{lemma: 5-2}, this theorem can be proved.
\end{proof}

\subsection{Summary}

Combining the above discussions, the final result of this paper is summarized.

\begin{theorem}    \label{theorem: main theorem 03}
Any admittance $Y(s)$ in \eqref{eq: specific bicubic admittance} satisfying Assumption~\ref{assumption: 01}
is realizable as a one-port series-parallel damper-spring-inerter circuit containing at most six elements, if and only if
$Y(s)$ satisfies one of the conditions in
Lemmas~\ref{lemma: Spring-Biquadratic-Network}, \ref{lemma: Case-6}--\ref{lemma: Case-8}, and \ref{lemma: 2-1}--\ref{lemma: 5-2}. Moreover, if any of the conditions in Lemma~\ref{lemma: Spring-Biquadratic-Network} holds, then
$Y(s)$ is realizable by the required circuit
after extracting the parallel spring $k_1 = \alpha_0/\beta_1$ to remove the pole at $s = 0$;
If one of the conditions in Lemmas~\ref{lemma: Case-6}--\ref{lemma: Case-8}  and \ref{lemma: 2-1}--\ref{lemma: 5-2} holds, then
$Y(s)$ is realizable by one of the circuit configurations in Figs.~\ref{fig: Spring-Biquadratic-Network-partial-removal}--\ref{fig: classes 4 and 5} (summarized in Table~\ref{table: summary}).
\end{theorem}
\begin{proof}
Any one-port series-parallel damper-spring-inerter circuit
 realizing $Y(s)$ in this theorem
is either the structure in Fig.~\ref{fig: Spring-N2} or any of other possible configurations.
Combining Theorems~\ref{theorem: main theorem 01} and \ref{theorem: main theorem 02} can imply this theorem.
\end{proof}

\begin{table}[!t]
\renewcommand{\arraystretch}{1.1}
\caption{The summary of the realization results of the circuit configurations in Figs.~\ref{fig: Spring-Biquadratic-Network-partial-removal}--\ref{fig: classes 4 and 5}.}
\label{table: summary}
\centering
\begin{tabular}{c   c}
\hline
Configurations & Realizability results    \\
\hline
Fig.~\ref{fig: Spring-Biquadratic-Network-partial-removal}    & Conditions and element value expressions in Lemmas~\ref{lemma: Case-6}--\ref{lemma: Case-8}     \\
Fig.~\ref{fig: classes 2 and 3}      &   Conditions and element value expressions in   Lemmas~\ref{lemma: 2-1}--\ref{lemma: 3-2}    \\
Fig.~\ref{fig: classes 4 and 5}     &  Conditions and element value expressions in  Lemmas~\ref{lemma: 4-1}--\ref{lemma: 5-2}   \\
\hline
\end{tabular}
\end{table}

\begin{remark}
It is noted that the necessary and sufficient condition   in Theorem~\ref{theorem: main theorem 03} is the
union    of the   conditions    in
Lemmas~\ref{lemma: Spring-Biquadratic-Network}, \ref{lemma: Case-6}--\ref{lemma: Case-8}, and \ref{lemma: 2-1}--\ref{lemma: 5-2}, which can be described by the set $\mathcal{S} = \{ (\alpha_3, ..., \alpha_0, \beta_3, ..., \beta_1) ~|~\alpha_i, \beta_j > 0 \text{ satisfy Assumption~\ref{assumption: 01} and the condition in Theorem~\ref{theorem: main theorem 03}}  \}$.
Then,   $\mathcal{S}$ is a
semi-algebraic subset  of the 7-dimensional Euclidean space, whose dimension is equal to the dimension of the positive-real set of $Y(s)$ in Lemma~\ref{lemma: positive-realness}. In contrast, the dimension of the realizability set for the five-element  realization results in \cite{WJ19} is one less than that of the positive-real set.  The five-element series-parallel realization results in \cite{WJ19} can be included by the results of this paper as special cases.
\end{remark}

\section{Examples of Circuit Synthesis and Passive Controller Optimizations}
\label{sec: examples}

This section will present several examples to illustrate the circuit synthesis results of this paper,
including the passive controller optimizations for a train suspension system.
Example~\ref{example 4-1} is an ideal numerical example, Example~\ref{eq: example optimal admittances} presents the mechanical circuit synthesis results for the optimal positive-real admittances in  \cite{WLLSC09}, and Example~\ref{example: 04} gives the mechanical circuit synthesis results based on the optimization results for the passive controller design of a side-view train suspension system.


\begin{example}    \label{example 4-1}
Consider an admittance $Y(s)$ in \eqref{eq: specific bicubic admittance} with
$\alpha_3 = 6$, $\alpha_2 = 13$, $\alpha_1 = 17$, $\alpha_0 = 10$, $\beta_3 = 7$, $\beta_2 = 13$, and $\beta_1 = 15$. Since
one can check that none of the conditions in Lemma~\ref{lemma: Spring-Biquadratic-Network} holds,
$Y(s)$ cannot be realized by a one-port six-element series-parallel
circuit  by completely removing the pole at $s = 0$. By  the Bott-Duffin circuit synthesis procedure, $Y(s)$ can be realized by a one-port ten-element series-parallel damper-spring-inerter circuit.
Then, it is
calculated that a negative root of the equation \eqref{eq: 4-1 condition 01 equation 01} is $y = -1$, such that $\Gamma_1 = 9$, $\Gamma_2 = 9$, $\Gamma_3 = 45$, $\Gamma_4 = 9$, $\Gamma_5 = 1125$, and $\Gamma_6 = 225$ have the same sign. Therefore, Condition~1 of Lemma~\ref{lemma: 4-1} holds. By the element value expressions in \eqref{eq: 4-1 element values}, $Y(s)$ can be realized by the one-port six-element series-parallel circuit in Fig.~\ref{fig: 4-1} with $c_1 = 1$~Ns/m, $c_2 = 5$~Ns/m, $c_3 = 1$~Ns/m, $k_1 = 1$~N/m, $k_2 = 2$~N/m, and $b_1 = 1$~kg, which saves four elements compared with the circuit realization by the Bott-Duffin procedure.
\end{example}

\begin{example} \label{eq: example optimal admittances}
Consider the one-wheel train suspension model in \cite[Fig.~1]{WLLSC09}.
By choosing the  static stiffness settings of $Q_1(s)$ and $Q_2(s)$ as $k_s = 1.41 \times 10^5$~N/m and $k_b = 1.26 \times 10^6$~N/m, it is shown in   \cite{WLLSC09} that the
positive-real
admittances $Q_1(s)$ and $Q_2(s)$ in the form of \eqref{eq: specific bicubic admittance} that minimizes the ride comfort index $J_1$ can be determined by the BMI optimization method as $Q_1(s) = (3212.9 s^3 + 2.407 \times 10^5 s^2 + 1.082 \times 10^6 s + 7.42 \times 10^6)/(s^3 + 7.674 s^2 + 52.627  s)$
and
$Q_2(s) =  (13754 s^3 + 1.272 \times 10^6 s^2 + 8.294 \times 10^7 s + 9.644  \times 10^7)/(s^3 + 65.763  s^2 + 76.541  s)$.
One can verify that $Q_1(s)$   satisfies Condition~1 of Lemma~\ref{lemma: Spring-Biquadratic-Network}. By making use of the Foster preamble, $Q_1(s)$ is realizable by a one-port six-element series-parallel circuit in Fig.~\ref{fig: foster-preamble-example03} with $c_1 = 0.26455$~Ns/m, $c_2 = 3212.635$~Ns/m, $c_3 = 13465.328$~Ns/m,
$k_1 = k_s = 1.41 \times 10^5$~N/m, $k_2 = 80490.163$~N/m, and $b_1 = 1894.339$~kg.
It can be checked that $Q_2(s)$ does not satisfy any of the conditions in Lemma~\ref{lemma: Spring-Biquadratic-Network}, which means that $Q_2(s)$ cannot be realized by any six-element series-parallel circuit by completely removing the pole at $s = 0$.
It can be checked that $Q_2(s)$   satisfies any of the
conditions in Lemmas~\ref{lemma: Case-6}--\ref{lemma: Case-8}, which means that the admittance
$Q_2(s)$ is realizable by one of the one-port six-element series-parallel circuit configurations in Fig.~\ref{fig: Spring-Biquadratic-Network-partial-removal}.
 For instance, by the element value expressions in \eqref{eq: Case-6 element values} with $x = 9.012 \times 10^5 > 0$, $Q_2(s)$ is realizable as  the configuration in Fig.~\ref{fig: Case-6} (also shown in Fig.~\ref{fig: Case-6-example03}) with $c_1 = 13754$~Ns/m,
$c_2 = 50.471$~Ns/m, $k_1 = 1.257 \times 10^6$~N/m, $k_2 = 11828.383$~N/m, $k_3 = 4238.539$~N/m,
and $b_1 = 209.912$~kg. In comparison, by utilizing the Bott-Duffin circuit synthesis procedure, $Q_2(s)$ is realizable as a ten-element series-parallel circuit as shown in Fig.~\ref{fig: Bott-Duffin} with
$c_1 = 16435.553$~Ns/m, $c_2 = 1047.275$~Ns/m, $c_3 = 84299.893$~Ns/m,
$k_1 = k_b = 1.26 \times 10^6$~N/m,  $k_2 = 1136.435$~N/m, $k_3 = 4.698 \times 10^5$~N/m, $k_4 = 3.819 \times 10^6$~N/m, $c_1 = 16435.553$~Ns/m, $c_2 = 1047.275$~Ns/m, and $c_3 = 84299.893$~Ns/m. Therefore, the results of this paper can save four elements for the physical realization of $Q_2(s)$.
\end{example}

\begin{figure}[thpb]
      \centering
      \subfigure[]{
      \includegraphics[scale=0.8]{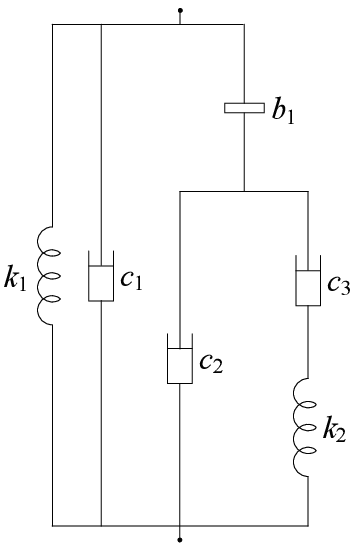}
      \label{fig: foster-preamble-example03}}
      \subfigure[]{
      \includegraphics[scale=0.8]{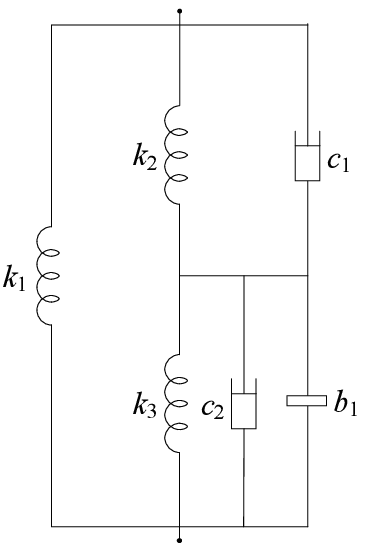}
      \label{fig: Case-6-example03}}
       \subfigure[]{
      \includegraphics[scale=0.8]{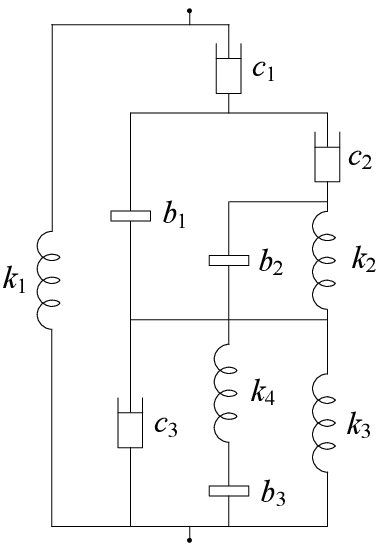}
      \label{fig: Bott-Duffin}}
      \caption{(a) The one-port six-element series-parallel circuit realizing the admittance $Q_1(s)$ in
      Example~\ref{eq: example optimal admittances}; (b) the one-port six-element series-parallel circuit realizing the admittance
      $Q_2(s)$ in Example~\ref{eq: example optimal admittances}, which is the configuration in
      Figure~\ref{fig: Case-6}; (c)  the one-port ten-element series-parallel circuit realizing the admittance $Q_2(s)$ by making use of the
      Bott-Duffin circuit synthesis procedure.}
      \label{fig: Examples-Configurations}
\end{figure}

The following train suspension control system  is  a specific case of Appendix~\ref{sec: mechanical control}, where the number of positive-real admittances satisfies $m = 2$, and the two admittances $Q_1(s) = Q_2(s)$   as in \eqref{eq: specific bicubic admittance} can be expressed in  \eqref{eq: Qi}
with the McMillan degree  being three.

Consider a side-view train suspension model  as shown in Fig.~\ref{fig: side-view-train-model} (see  \cite{JMGS12}), where $m_s$, $I_s$, $z_s$, and $\theta_s$ denote the mass, pitch inertia, vertical displacement, and pitch angle of train body;
$m_{b1}$, $I_{b1}$, $z_{s1}$, and $\theta_{b1}$  denote the mass, pitch inertia, vertical displacement, and pitch angle of the front bogie; $m_{b2}$, $I_{b2}$, $z_{s2}$, and $\theta_{b2}$  denote the mass, pitch inertia, vertical displacement, and pitch angle of the rear bogie;
$m_w$ and $k_w$ denote the wheel-set mass and the stimulate vertical stiffness between the rail and the wheel;
$z_{w1}$, $z_{w2}$, $z_{w3}$, and $z_{w4}$ denote the vertical displacements of four wheelsets;
$z_{r1}$, $z_{r2}$, $z_{r3}$, and $z_{r4}$ denote the rail track displacements;
$z_{sf}$ and $z_{sr}$ denote the displacements of the front and rear part of the train body;
$V$ denotes the train speed.    Here, each of   four primary suspensions is the  parallel connection of a damper $c_p$ and a spring $k_p$;
the admittances of the two secondary suspension struts are denoted as $Q_1(s)$ and $Q_2(s)$; $F_1$ and $F_2$ are forces provided by $Q_1(s)$ and $Q_2(s)$; $l_s$ and $l_b$ denote the semi-longitudinal spacings of the secondary suspensions and wheelsets.

\begin{figure}[thpb]
      \centering
      \includegraphics[scale=1.2]{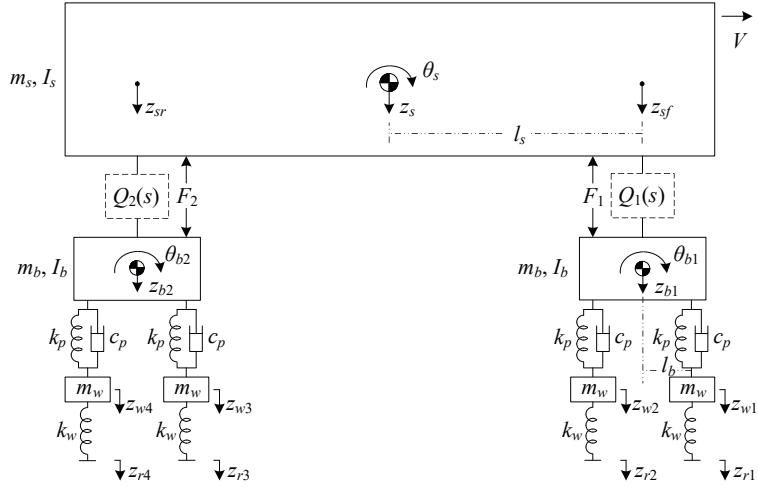}
      \caption{A side-view train suspension model, where the positive-real transfer functions $Q_1(s)$ and $Q_2(s)$ to be optimized are the admittances of  damper-spring-inerter circuits (secondary suspension parts) \cite{JMGS12}.  }
       \label{fig: side-view-train-model}
\end{figure}

By Newton's Second Law, one can formulate the  motion equations as
\begin{equation}  \label{eq: side-view suspension motion equation}
M_g \ddot{z}_g + C_g \dot{z}_g + K_g z_g = E_g u + K_r z_r,
\end{equation}
where $z_g  = \left[
           z_s, \theta_s, z_{b1}, \theta_{b1}, z_{b2}, \theta_{b2}, z_{w1}, z_{w2}, z_{w3}, z_{w4}
       \right]^\mathrm{T}$,
$u = \left[ F_1, F_2 \right]^\mathrm{T}$, $z_r = \left[ z_{r1}, z_{r2}, z_{r3}, z_{r4} \right]^\mathrm{T}$,
and the matrices $M_g \in \mathbb{R}^{10 \times 10}$, $C_g \in \mathbb{R}^{10 \times 10}$, $K_g \in \mathbb{R}^{10 \times 10}$, $E_g \in \mathbb{R}^{10 \times 2}$, and $K_r \in \mathbb{R}^{10 \times 4}$ are shown in  \cite[Section~V]{Wang_sup}.
By letting $x_m = \left[z_{g}^\mathrm{T}, \dot{z}_{g}^\mathrm{T} \right]^\mathrm{T}$, equation
\eqref{eq: side-view suspension motion equation} can be equivalent to
\begin{equation}  \label{eq: motion state-space model}
\dot{x}_m = A_m x_m + B_m u + B_{mw} z_r,
\end{equation}
where
\begin{equation*}
\begin{split}
A_m = \left[
    \begin{array}{cc}
                \mathbf{0} & I   \\
                -M_g^{-1} K_g & -M_g^{-1} C_g    \\
              \end{array}
            \right], ~
B_m = \left[
        \begin{array}{c}
          \mathbf{0} \\
          M_g^{-1} E_g \\
        \end{array}
      \right], ~
      B_{mw} = \left[
        \begin{array}{c}
          \mathbf{0} \\
          M_g^{-1} K_r \\
        \end{array}
      \right].
\end{split}
\end{equation*}
Let
$w = z_{r1}$.
Then, it is clear that
$z_{r2}(t) = w(t - \tau_2)$, $z_{r3}(t) = w(t - \tau_3)$, and
$z_{r4}(t) = w(t - \tau_4)$, where
$\tau_2 = 2 l_b/ V$, $\tau_3 = 2 l_s/V$, and $\tau_4 = 2 (l_b + l_s)/V$.
Then, by making use of the Pad\'{e} approximation method, the
minimal state-space realization of the transfer function from $w$ to $z_r$ can be obtained as
\begin{equation} \label{eq: wheel model}
\begin{split}
\dot{x}_r = A_r x_r + B_r w, ~~
z_r = C_r x_r + D_r w,
\end{split}
\end{equation}
where the dimension of $x_r \in \mathbb{R}^{n_r}$ is based on the order of the Pad\'{e} approximation.
By \eqref{eq: side-view suspension motion equation}--\eqref{eq: wheel model}, the state-space equations are obtained in the form of \eqref{eq: state-space equation}, that is,
\begin{equation}    \label{eq: state-space equation 02}
\begin{split}
\dot{x}  &= A x + B  u + B_w w,  \\
y  &= C  x,    ~~ z = C_z x,
\end{split}
\end{equation}
where
$x = \left[ x_m^\mathrm{T}, x_r^\mathrm{T} \right]^\mathrm{T}$,  $z = \left[ \dot{z}_{s}, \dot{z}_{sf}, \dot{z}_{sr}  \right]^\mathrm{T}$, $y = \left[ z_{sf} - z_{b1}, z_{sr} - z_{b2}, \dot{z}_{sf} - \dot{z}_{b1}, \dot{z}_{sr} - \dot{z}_{b2} \right]^\mathrm{T}$, and
\begin{equation*}
\begin{split}
A = \left[
      \begin{array}{cc}
        A_m & B_{mw} C_r \\
        \mathbf{0} &  A_r \\
      \end{array}
    \right], ~~
B = \left[
      \begin{array}{c}
        B_m   \\
        \mathbf{0}   \\
      \end{array}
    \right],   ~~
B_w =  \left[
      \begin{array}{c}
        B_{mw} D_r   \\
        B_r   \\
      \end{array}
    \right], \\
C = \left[
      \begin{array}{ccc}
        C_{11} & \mathbf{0}  & \mathbf{0}_{1\times n_r}  \\
        \mathbf{0}  & C_{11} & \mathbf{0}_{1\times n_r}   \\
      \end{array}
    \right],  ~~
C_z = \left[
      \begin{array}{cccc}
        \mathbf{0}_{1\times 10}   &  1 &  0 &  \mathbf{0}_{1\times (n_r+8)}  \\
        \mathbf{0}_{1\times 10}   &  1 & l_s &  \mathbf{0}_{1\times (n_r+8)}  \\
        \mathbf{0}_{1\times 10}   &  1 & -l_s &  \mathbf{0}_{1\times (n_r+8)}  \\
      \end{array}
    \right],
\end{split}
\end{equation*}
with
\[
C_{11} = \left[
           \begin{array}{ccccccc}
             1 &  l_s & -1 & 0 & 0 &  0 & \mathbf{0}_{1\times 4} \\
             1 & -l_s &  0 & 0 & -1 & 0 & \mathbf{0}_{1\times 4} \\
           \end{array}
         \right].
\]
Furthermore, let the admittances of two secondary suspension circuits satisfy
$Q_1(s) = Q_2(s)$ and be as in \eqref{eq: specific bicubic admittance},
that is,
\[
Q_1(s) = Q_2(s) = \frac{\alpha_3 s^3 + \alpha_2 s^2 + \alpha_1 s + \alpha_0}{\beta_3 s^3 + \beta_2 s^2 + \beta_1 s},
\]
where  the coefficients satisfy Assumption~\ref{assumption: 01}. Then, by \eqref{eq: controller K(s)},
the passive controller $K(s)$   can be obtained as
\begin{equation*}
K(s) = \left[
          \begin{array}{cccc}
               \alpha_{0}/\beta_{1} &     &  Q_{1,1}(s) &        \\
                 &    \alpha_{0}/\beta_{1} &     & Q_{2,1}(s)
          \end{array}
        \right],
\end{equation*}
where
\[
Q_{i,1}(s) = Q_{i}(s) - \frac{\alpha_{0}}{\beta_{1} s}
= \frac{\alpha_2' s^2 + \alpha_1' s + \alpha_0'}{\beta_3 s^2 + \beta_2 s + \beta_1}
\]
with $\alpha_2' = \alpha_3$ and $\alpha_j' = (\alpha_{j+1} \beta_1 - \alpha_0 \beta_{j+2})/\beta_1$
for $i = 1, 2$ and $j = 0, 1$.
Furthermore, by \eqref{eq: Realization Ak},
one can formulate a minimal state-space realization of $K(s)$ as in \eqref{eq: state-space equation controller}, that is,
\begin{equation}  \label{eq: state-space K}
\dot{x}_k = A_k x_k + B_k y,  ~~
u = C_k x_k + D_k y,
\end{equation}
where $\{A_k, B_k, C_k, D_k \}$ satisfy
\begin{equation*}
\begin{split}
A_k = \left[
        \begin{array}{cccc}
          0 & 1 & 0 & 0 \\
          -\displaystyle\frac{\beta_1}{\beta_3} & -\displaystyle\frac{\beta_2}{\beta_3} & 0 & 0 \\
          0 & 0 & 0 & 1 \\
          0 & 0 & -\displaystyle\frac{\beta_1}{\beta_3} & -\displaystyle\frac{\beta_2}{\beta_3} \\
        \end{array}
      \right],   ~
B_k = \left[
        \begin{array}{cccc}
          0 & 0 & 0 & 0 \\
          0 & 0 & 1 & 0 \\
          0 & 0 & 0 & 0 \\
          0 & 0 & 0 & 1 \\
        \end{array}
  \right],   \\
C_k = \left[
        \begin{array}{cccc}
          \gamma_0 & \gamma_1 & 0 & 0 \\
          0 & 0 & \gamma_0 & \gamma_1 \\
        \end{array}
      \right],  ~
 D_k = \left[
        \begin{array}{cccc}
          \displaystyle\frac{\alpha_0}{\beta_1} & 0 & \displaystyle\frac{\alpha_3}{\beta_3} & 0 \\
          0 & \displaystyle\frac{\alpha_0}{\beta_1} & 0 & \displaystyle\frac{\alpha_3}{\beta_3} \\
        \end{array}
      \right]
\end{split}
\end{equation*}
with $\gamma_0 = (-\alpha_0 \beta_2 \beta_3 + \alpha_1 \beta_1 \beta_3 - \alpha_3 \beta_1^2)/(\beta_1 \beta_3^2)$ and $\gamma_1 = (-\alpha_0 \beta_3^2 + \alpha_2 \beta_1 \beta_3 - \alpha_3 \beta_1 \beta_2)/(\beta_1 \beta_3^2)$.
Finally, by \eqref{eq: state-space equation 02} and \eqref{eq: state-space K},
one  obtains
the closed-loop state-space equations in the form of \eqref{eq: closed-loop state-space equation}, that is,
\begin{equation} \label{eq: closed-loop system}
\dot{x}_{cl}  = A_{cl} x_{cl} + B_{cl} w, ~~
z  = C_{cl} x_{cl},
\end{equation}
where $x_{cl} = [x^\mathrm{T}, x_k^\mathrm{T}]^\mathrm{T}$, and $A_{cl}$, $B_{cl}$, and $C_{cl}$ satisfy
\eqref{eq: Acl Bcl}, that is,
\begin{equation*}
\begin{split}
A_{cl} = \left[
           \begin{array}{cc}
             A + B D_k C & B C_k \\
             B_k C & A_k \\
           \end{array}
         \right], ~~
B_{cl} = \left[
           \begin{array}{c}
             B_w   \\
             \mathbf{0} \\
           \end{array}
         \right],   ~~
C_{cl} = \left[
           \begin{array}{cc}
             C_z & \mathbf{0}
           \end{array}
         \right].
\end{split}
\end{equation*}

Let $S_{\dot{z}}(j\omega)$ and $S_{\dot{w}}(j\omega)$ denote the power spectral density (PSD) function of $\dot{z}$ and $\dot{w}$, respectively.
By  \cite{JMGS12}, assume that $\dot{w} = \dot{z}_{r1}$ is a white noise and the corresponding PSD function  satisfies $S_{\dot{w}}(j\omega) = 4 \pi^2 \kappa V$, where $\kappa$ denotes the vertical track roughness factor.
The ride comfort performance measure index $J_{1}$ can be expressed in terms of the root-mean-square (RMS) of $\dot{z} = \left[ \ddot{z}_{s}, \ddot{z}_{sf}, \ddot{z}_{sr}  \right]^\mathrm{T}$, which by   \cite{ZDG95} can be equivalent to the $\mathcal{H}_2$ norm of the transfer matrix from $w$ to $z$, that is,
\begin{equation*}
\begin{split}
J_{1} &= \sqrt{\lim_{T \rightarrow \infty} \frac{1}{2T} \int_{-T}^T \| \dot{z}(t) \|^2 dt}  = \sqrt{\frac{1}{2\pi} \int_{-\infty}^{\infty} \text{Trace}(S_{\dot{z}}(j\omega)) d \omega  }  \\
&= 2 \pi \sqrt{\frac{\kappa V}{2\pi} \int_{-\infty}^{\infty} \text{Trace}( H^{\mathrm{T}}_{\dot{w} \rightarrow \dot{z}}(-j\omega) H_{\dot{w} \rightarrow \dot{z}}(j\omega)) d \omega  } \\
&= 2 \pi \sqrt{\kappa V} \| H_{w\rightarrow z} \|_2,
\end{split}
\end{equation*}
where  $H_{\dot{w} \rightarrow \dot{z}}$ denotes the transfer matrix  from $\dot{w}$ to $\dot{z}$ and
$H_{w\rightarrow z}$ denotes the  transfer matrix  from  $w$ to $z$ for the closed-loop system in \eqref{eq: closed-loop system}, satisfying $H_{\dot{w} \rightarrow \dot{z}} = H_{w\rightarrow z} = C_{cl} (s I - A_{cl})^{-1} B_{cl}$.

By referring to  \cite{JMGS12,CSW17},  set the parameters as $m_s = 38000$~kg, $m_b = 2500$~kg, $m_w = 1117.9$~kg,
$I_s = 2.31 \times 10^6$~kgm$^2$, $I_b = 1500$~kgm$^2$, $l_s = 9.5$~m, $l_b = 1.25$~m, $k_p = 4.935 \times 10^6$~N/m, $k_w = 1 \times 10^6$~N/m,  $c_p = 5.074 \times 10^4$~Ns/m, $V = 55$~m/s, and
$\kappa = 2.5 \times 10^{-7}$~m$^3$/cycle. Furthermore, utilize the third-order Pad\'{e} approximation to obtain \eqref{eq: wheel model}.
By  following Procedure~\ref{procedure: 01} and
utilizing the optimization solver \emph{patternsearch} in MATLAB, one can obtain the optimal values of $J_1$ (solid line in Fig.~\ref{fig: performances}) and the corresponding optimal positive-real admittances $Q_1(s) = Q_2(s)$, where the value of static stiffness $k_s$ is fixed and ranges from $0.5 \times 10^6$~N/m to $10 \times 10^6$~N/m. In the optimization, the objective function can be expressed as  $J_1^2/(4 \pi^2 \kappa V ) = \text{Trace}(C_{cl} P C_{cl}^\mathrm{T})$, where
$P$ is the positive definite matrix solved from
the Lyapunov equation
$A_{cl} P + P A_{cl}^\mathrm{T} + B_{cl} B_{cl}^\mathrm{T} = 0$,
$A_{cl}$ is constrained to be stable, and
$Q_1(s)=Q_2(s)$  is constrained to be a positive-real admittance in \eqref{eq: specific bicubic admittance}, whose coefficients satisfy the condition in
Lemma~\ref{lemma: positive-realness}.
In comparison, the optimal performances corresponding to the case when each of the admittances  $Q_1(s) = Q_2(s)$ is realizable as the parallel circuit of one spring and one damper are also presented (dot-dashed line in Fig.~\ref{fig: performances}). As shown in Fig.~\ref{fig: performances_comparison}, the ride comfort performance $J_1$ can be significantly improved by using the positive-real admittance as in \eqref{eq: specific bicubic admittance}, which can also show that introducing inerters can certainly improve system performances.

\begin{figure}[thpb]
      \centering
      \subfigure[]{
      \includegraphics[scale=0.5]{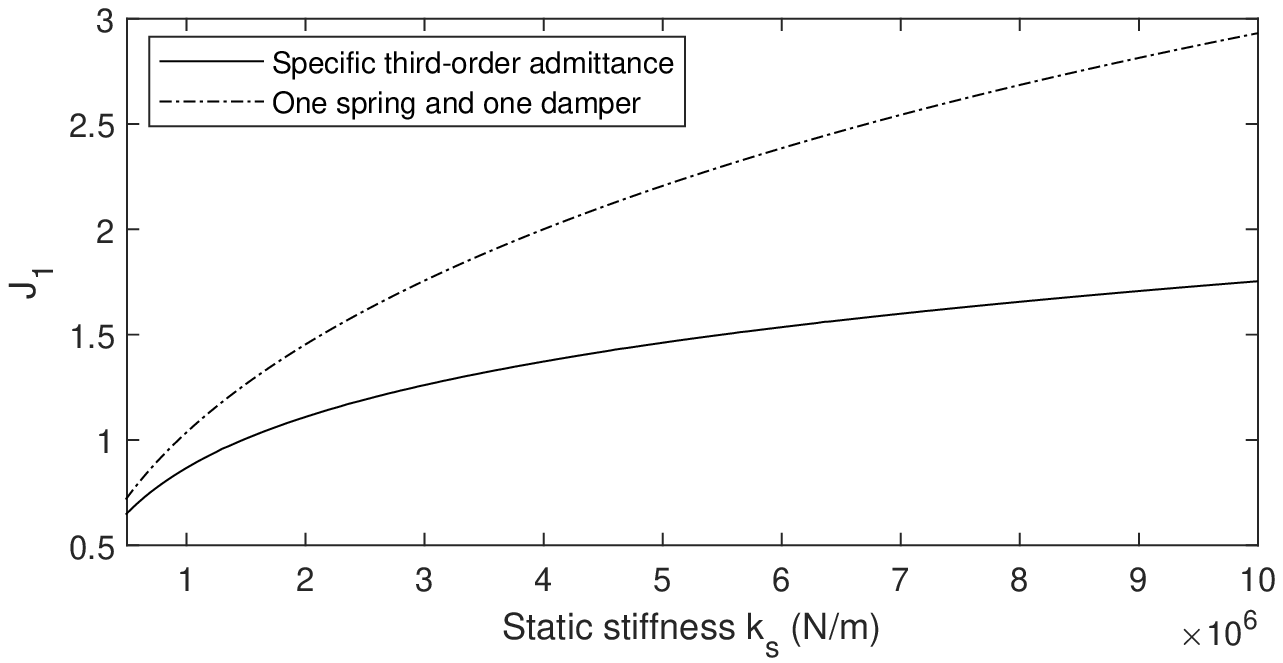}
      \label{fig: performances}}
      \subfigure[]{
      \includegraphics[scale=0.5]{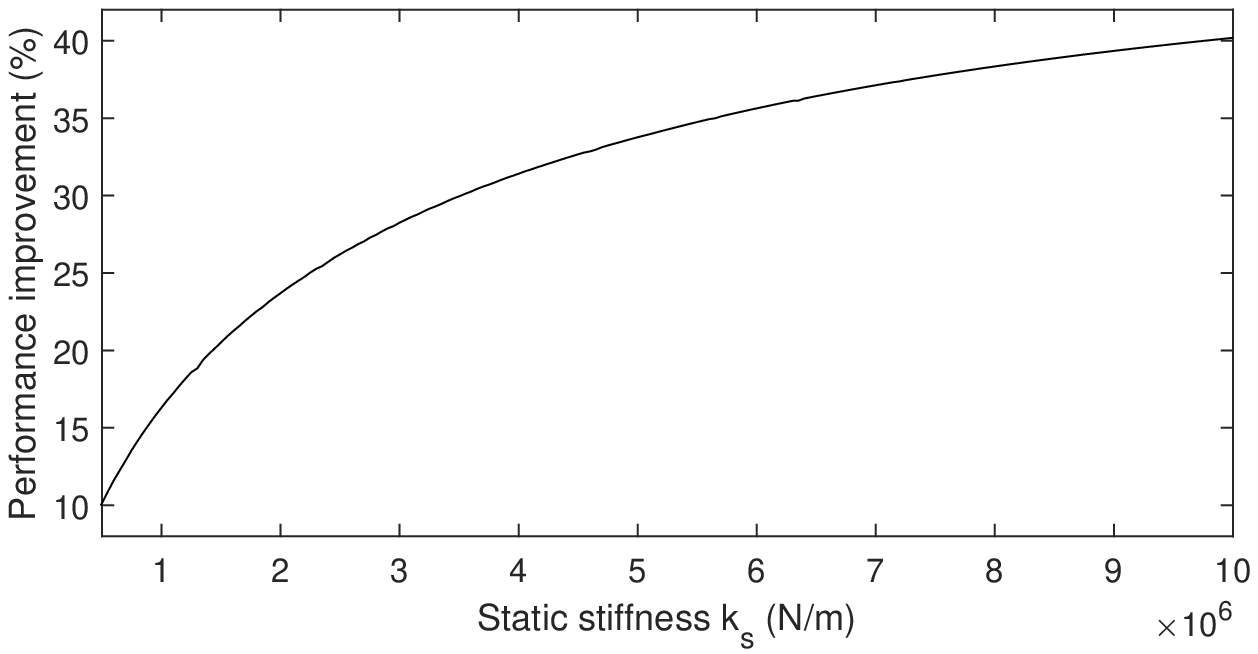}
      \label{fig: performances_comparison}}
      \caption{(a) The optimal ride comfort performances $J_1$ for the case when each of the admittances $Q_1(s)=Q_2(s)$ is in the form of
      \eqref{eq: specific bicubic admittance} (solid line) and the case when each of the admittances  $Q_1(s) = Q_2(s)$ is realizable by the two-element parallel circuit of one spring and one damper  (dot-dashed line); (b) the optimal performance improvement in (a),  which is
      $(J_1^{(2)} - J_1^{(1)})/J_1^{(2)}  \times 100 \%$, where $J_1^{(1)}$ and $J_1^{(2)}$ are optimal performance values corresponding to the above two cases, respectively. Here, the static stiffness $k_s$ ranges from $0.5 \times 10^6$~N/m to $10 \times 10^6$~N/m.}
      \label{fig: J1}
\end{figure}

\begin{example}  \label{example: 04}
When $k_s = 4 \times 10^6$~N/m, the optimal performance satisfies $J_1 = 1.3722$ and the corresponding positive-real admittance $Q_1(s) = Q_2(s)$ is in the form of
\eqref{eq: specific bicubic admittance} where $\alpha_3 = 3.905 \times 10^7$,
$\alpha_2 = 1.647 \times 10^8$, $\alpha_1 = 2.93 \times 10^9$, $\alpha_0 = k_s \beta_1 = 4 \times 10^6$, $\beta_3 = 41.181$,
$\beta_2 = 732.533$, and $\beta_1 = 1$. It is verified that $Q_1(s) = Q_2(s)$  does not satisfy any of the conditions in Lemma~\ref{lemma: Spring-Biquadratic-Network}, which means that $Q_1(s) = Q_2(s)$ cannot be realized by a one-port six-element series-parallel circuit by completely removing the pole at $s = 0$.
It can be checked that $Q_1(s) = Q_2(s)$   satisfies any of the
conditions in Lemmas~\ref{lemma: Case-6}--\ref{lemma: Case-8}, which means that
$Q_1(s) = Q_2(s)$ can be realized by one of the  six-element circuit configurations in Fig.~\ref{fig: Spring-Biquadratic-Network-partial-removal}.
 For instance, by the element value expressions in \eqref{eq: Case-8 element values} with
 $x = 2.413 \times 10^{11}$, $Q_1(s)$ and $Q_2(s)$ can be realized as the circuit in Fig.~\ref{fig: Case-8}
with $c_1 = 1.177$~Ns/m, $c_2 = 9.484 \times 10^5$~Ns/m, $k_1 = 3.9997 \times 10^6$~N/m,
$k_2 = 809.474$~N/m, $k_3 = 485.188$~N/m, and $b_1 = 53314.9$~kg. In comparison,  by utilizing the Bott-Duffin circuit synthesis procedure, $Q_1(s) = Q_2(s)$ is realizable by a one-port ten-element series-parallel circuit as shown in Fig.~\ref{fig: Bott-Duffin-02} with
$c_1 = 8.701 \times 10^6$~Ns/m, $c_2 = 1.355 \times 10^5$~Ns/m, $c_3 = 1.064 \times 10^6$~Ns/m,
$k_1 = k_s = 4 \times 10^6$~N/m,  $k_2 = 182.143$~N/m, $k_3 = 4.206 \times 10^6$~N/m,
$k_4 = 7.583 \times 10^6$~N/m, $c_1 = 8.701 \times 10^6$~Ns/m,
$c_2 = 1.355 \times 10^5$~Ns/m, and $c_3 = 1.064 \times 10^6$~Ns/m. Therefore, the results of this paper can save four elements for each mechanical circuit realization.
\end{example}

\begin{remark}
The circuit synthesis results of this paper can guarantee that the six-element series-parallel damper-spring-inerter circuit realizations in Examples~\ref{example 4-1}--\ref{example: 04} contain the minimal number of elements. In most cases, if a given admittance satisfies one of
the realizability conditions in this paper,
any of the corresponding  circuit realizations cannot be equivalent to another circuit containing fewer elements by other circuit synthesis approaches.
\end{remark}

\begin{figure}[thpb]
      \centering
      \subfigure[]{
      \includegraphics[scale=0.8]{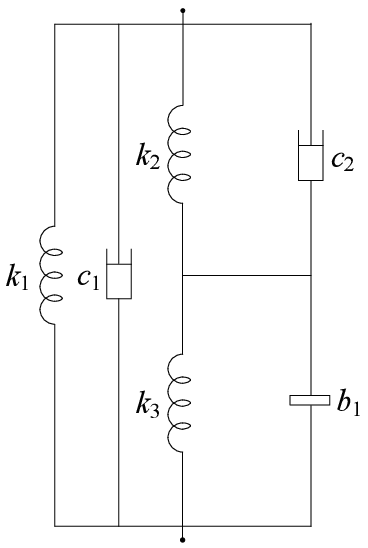}
      \label{fig: Case-8-example04}}
       \subfigure[]{
      \includegraphics[scale=0.8]{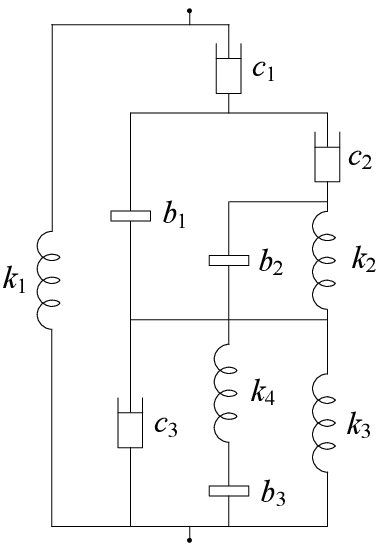}
      \label{fig: Bott-Duffin-02}}
      \caption{(a) The one-port six-element series-parallel circuit realizing the optimal positive-real admittances $Q_1(s)$ and $Q_2(s)$ in
      Example~\ref{example: 04}, which is the configuration in
      Figure~\ref{fig: Case-8};  (b)  the one-port ten-element series-parallel circuit realizing $Q_1(s)$ and $Q_2(s)$ in
      Example~\ref{example: 04} by making use of   the Bott-Duffin circuit synthesis procedure.}
      \label{fig: Examples-04-Configurations}
\end{figure}

\section{Conclusion}

This paper has solved the passive circuit synthesis problem for a bicubic (third-order) admittance with a simple pole at the origin
to be realizable by a one-port series-parallel damper-spring-inerter circuit consisting of at most six elements.
Necessary and sufficient conditions have been derived for such a specific bicubic admittance to be realizable as this class of passive circuits,  where the conditions are related to the function coefficients and the roots of certain algebraic equations.
Moreover, a group of circuit configurations that can realize the admittances satisfying the conditions have been
presented, where the element value expressions are explicitly given.
Compared with the Bott-Duffin circuit synthesis procedure, much fewer elements are needed to achieve the circuit realizations by using the results of this paper.
Finally, numerical examples and the control system design for a train suspension system have been presented.
The results derived in this paper can be applied to design and to realize the passive
controllers in many inerter-based control systems, and can contribute to the development of passive circuit synthesis and many other related fields.

\vspace{0.2cm}

\begin{appendices}

\section{Inerter-Based Mechanical Control Using Passive Controllers}  \label{sec: mechanical control}

\setcounter{procedure}{0}
\renewcommand{\theprocedure}{A.\arabic{procedure}}

\setcounter{equation}{0}
\renewcommand{\theequation}{A.\arabic{equation}}

\setcounter{remark}{0}
\renewcommand{\theremark}{A.\arabic{remark}}

In this appendix, the  design procedure of a general class of inerter-based control systems will be formulated.  There are $m$ positive-real admittances with a pole at $s=0$, which  constitute
the passive controller to be determined such that the closed-loop system is stable and the system performance is optimized.

Consider the augmented  model of a linear time-invariant vibration system $G(s)$ to be controlled, such as a suspension system, wind turbine vibration system, building vibration system, etc.,  whose state-space equations are
\begin{equation}    \label{eq: state-space equation}
\begin{split}
\dot{x}  &= A x + B  u + B_w w,  \\
y  &= C  x,    ~~ z = C_z x,
\end{split}
\end{equation}
where $x$ denotes the state vector, $u$ denotes the input vector consisting of forces provided by passive mechanical circuits, $y$ denotes the measured output for control, $z$ is the controlled output related to system performances,  and  $w$ denotes the noise vector.

Suppose that there are $m$ one-port   spring-damper-inerter circuits, and the admittance of each circuit is positive-real and contains a pole at $s = 0$, which is expressed as
\begin{equation}  \label{eq: Qi}
Q_i(s) = \frac{\alpha_{i, n_i} s^{n_i} + \alpha_{i, n_i - 1} s^{n_i - 1} + \cdots + \alpha_{i, 1} s +  \alpha_{i, 0}}{\beta_{i, n_i} s^{n_i} + \beta_{i, n_i - 1} s^{n_i - 1} + \cdots + \beta_{i, 1}s}
= \frac{\alpha_{i, 0}}{\beta_{i, 1}s} + Q_{i,1}(s),
\end{equation}
for $i = 1, 2, \ldots, m$, where
\begin{equation}  \label{eq: Qi1}
Q_{i,1}(s) = \frac{\alpha_{i, n_i - 1}' s^{n_i - 1} + \alpha_{i, n_i - 2}' s^{n_i - 2} + \cdots + \alpha_{i, 0}'}{\beta_{i, n_i} s^{n_i - 1} + \beta_{i, n_i - 1} s^{n_i - 2} + \cdots + \beta_{i, 1}}
\end{equation}
is also a positive-real function
with
\begin{equation*}
\alpha_{i, n_i-1}' = \alpha_{i, n_i}, ~~~ \alpha_{i, j}' = \frac{\alpha_{i, j+1} \beta_{i, 1} - \alpha_{i, 0}\beta_{i, j+2}}{\beta_{i,1}},
\end{equation*}
for $j = 0, 1, \ldots, n_i - 2$.

Then, one aims to design an inerter-based control system whose diagram is as shown in   Fig.~\ref{fig: Control-Diagram}, where the above $m$ admittances constitute the passive controller $K(s)$.

\begin{figure}[thpb]
      \centering
      \includegraphics[scale=0.9]{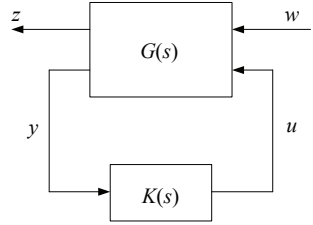}
      \caption{Control synthesis diagram, where $G(s)$ is the transfer (function) matrix of the vibration system \eqref{eq: state-space equation} to be controlled, $K(s)$ is the passive controller, $y$ is the measured output, $u$ is the control input, $w$ is the noise signal, and $z$ is the controlled output.}
      \label{fig: Control-Diagram}
\end{figure}

Let the measured output $y = \left[
                                y_1^\mathrm{T} , y_2^\mathrm{T}
                             \right]^\mathrm{T} \in \mathbb{R}^{2m}$ consist of the relative displacements $y_1 \in \mathbb{R}^{m}$ and the relative velocities $y_2 \in \mathbb{R}^{m}$
of two terminals for $m$  circuits. Referring to \cite{CHW15}, the $m$ passive circuits $Q_i(s)$ as in \eqref{eq: Qi} for $i = 1, 2, \ldots, m$ can constitute the passive controller $K(s)$ whose input is $y$ and output is the force $u$ as in \eqref{eq: state-space equation}, that is, $u = K(s) y$. Here, the  controller $K(s)$ can be expressed as
\begin{equation}  \label{eq: controller K(s)}
K(s) = \left[
          \begin{array}{ccc;{2pt/2pt}ccc}
             \displaystyle\frac{\alpha_{1, 0}}{\beta_{1, 1}} &   &  &  Q_{1,1} &   &       \\
                  &  \ddots &  &  & \ddots &   \\
                 &   & \displaystyle\frac{\alpha_{m, 0}}{\beta_{m, 1}} &  &   & Q_{m,1}
          \end{array}
        \right],
\end{equation}
with $Q_{i,1}(s)$ for $i = 1, 2, \ldots, m$ being expressed as in \eqref{eq: Qi1}. Assume that $Q_{i,1}(s)$ for $i = 1, 2, \ldots, m$ as in \eqref{eq: Qi1} does not contain any common factor. Based on the results in \cite{AV06},
a minimal state-space realization $\{A_{q_{i}}, B_{q_{i}}, C_{q_{i}}, D_{q_{i}}\}$ of the positive-real function $Q_{i,1}(s)$ for $i = 1, 2, \ldots, m$ can be obtained as
\begin{equation*}
\begin{split}
&A_{q_{i}}  = \left[\begin{array}{cccc}
0 & 1 &   \cdots & 0   \\
\vdots & \vdots   & \ddots & \vdots    \\
0 & 0 &   \cdots & 1     \\
-\displaystyle\frac{\beta_{i,1}}{\beta_{i,n_i}} &  -\displaystyle\frac{\beta_{i,2}}{\beta_{i,n_i}}   & \cdots & -\displaystyle\frac{\beta_{i,n_i-1}}{\beta_{i,n_i}}
\end{array}\right],
B_{q_{i}}   = \left[\begin{array}{c}  0 \\  \vdots \\ 0 \\ 1 \end{array}\right],   \\
&~~C_{q_{i}}  = \left[\begin{array}{cccc} \gamma_{i,0} &  \gamma_{i,1}  & \cdots &   \gamma_{i,n_i-2} \end{array}\right], ~~
D_{q_{i}} = \left[\begin{array}{c}  \displaystyle\frac{\alpha_{i,n_i}}{\beta_{i,n_i}} \end{array}\right],
\end{split}
\end{equation*}
where
\begin{equation*}
\gamma_{i,k}  = \frac{\alpha_{i,k}'\beta_{i,n_i} - \alpha_{i,n_i-1}'\beta_{i,k+1}}{\beta_{i,n_i}^2}  =  \frac{(\alpha_{i,k+1}\beta_{i,1} - \alpha_{i,0}\beta_{i,k+2}) \beta_{i,n_i} - \alpha_{i,n_i}\beta_{i,1}\beta_{i,k+1}}{\beta_{i,1}\beta_{i,n_i}^2},
\end{equation*}
for $k = 0, 1, \ldots, n_i-2$.
Furthermore, a minimal state-space realization $\{ A_k, B_k, C_k, D_k \}$ of the controller $K(s)$ as in \eqref{eq: controller K(s)}
satisfies
\begin{equation}  \label{eq: state-space equation controller}
\begin{split}
\dot{x}_k = A_k x_k + B_k y,  ~~
u = C_k x_k + D_k y,
\end{split}
\end{equation}
where the dimension of $x_k$ is equal to the  McMillan degree  of $K(s)$, and $A_k$, $B_k$, $C_k$, and $D_k$ are expressed as
\begin{equation}   \label{eq: Realization Ak}
\begin{split}
A_k  = \left[
  \begin{array}{ccc}
    A_{q_{1}} &   &  \\
     & \ddots & \\
     &   &  A_{q_{m}}
  \end{array}
\right],     ~~
 B_k  = \left[
  \begin{array}{ccc;{2pt/2pt}ccc}
    \mathbf{0}    & &   &  B_{q_{1}} & &    \\
       & \ddots  &  &     & \ddots &         \\
      &   &  \mathbf{0} &   &  &  B_{q_{m}}
  \end{array}
\right],       \\
C_k   = \left[
  \begin{array}{ccc}
    C_{q_{1}} &   &  \\
      & \ddots & \\
      &   &  C_{q_{m}}
  \end{array}
\right],     ~~
D_k  =  \left[
  \begin{array}{ccc;{2pt/2pt}ccc}
    \displaystyle\frac{\alpha_{1, 0}}{\beta_{1, 1}}  &   &   &  D_{q_{1}} & &   \\
       & \ddots  &  &     & \ddots &         \\
      &   &  \displaystyle\frac{\alpha_{m, 0}}{\beta_{m, 1}}    & &  &  D_{q_{m}}
  \end{array}
\right].
\end{split}
\end{equation}

Combining \eqref{eq: state-space equation} and \eqref{eq: state-space equation controller}, the closed-loop state-space equation can be obtained as
\begin{equation}   \label{eq: closed-loop state-space equation}
\begin{split}
\dot{x}_{cl} = A_{cl} x_{cl} + B_{cl} w,  ~~
z = C_{cl} x_{cl},
\end{split}
\end{equation}
where $x_{cl} = [x^\mathrm{T}, x_k^\mathrm{T}]^\mathrm{T}$, and
\begin{equation}  \label{eq: Acl Bcl}
A_{cl} = \left[
           \begin{array}{cc}
             A + B D_k C & B C_k \\
             B_k C & A_k \\
           \end{array}
         \right], ~~
B_{cl} = \left[
           \begin{array}{c}
             B_w   \\
             \mathbf{0} \\
           \end{array}
         \right],   ~~
C_{cl} = \left[
           \begin{array}{cc}
             C_z & \mathbf{0}
           \end{array}
         \right].
\end{equation}
Let $H_{w\rightarrow z} (s) = C_{cl} (sI - A_{cl})^{-1} B_{cl}$ denote the transfer (function) matrix  from $w$ to $z$.
Then, suppose that a system performance  is proportional to the $\mathcal{H}_2$ norm of  $H_{w\rightarrow z} (s)$, that is,
\begin{equation}   \label{eq: system performances}
J_{\mathrm{passive}} \propto \| H_{w\rightarrow z} \|_2.
\end{equation}
It is implied from \cite[Lemma~4.6]{ZDG95} that $\| H_{w\rightarrow z} \|_2^2 = \text{Trace}(C_{cl} P C_{cl}^\mathrm{T})$, where the positive definite matrix $P > 0$ is the unique solution of the   Lyapunov equation
\begin{equation}  \label{eq: Lyapunov equation}
A_{cl} P + P A_{cl}^\mathrm{T} + B_{cl} B_{cl}^\mathrm{T} = 0.
\end{equation}

The optimization design procedure for $J_{\mathrm{passive}}$ can be summarized as follows.
\begin{procedure}  \label{procedure: 01}
Consider a vibration system whose state-space equation satisfies \eqref{eq: state-space equation}. Then, the steps of designing a passive controller $K(s)$ to minimize the system performance $J_{\mathrm{passive}} \propto \| H_{w\rightarrow z} \|_2$ are   as follows.
\begin{enumerate}
  \item[1.] Choosing the McMillan degrees of   admittances $Q_i(s)$   in \eqref{eq: Qi} for $i = 0, 1, \ldots, m$, $K(s)$ can be formulated as in \eqref{eq: controller K(s)}. Determine the positive-real conditions and further choose the constraint conditions of $Q_i(s)$ such that $Q_i(s)$ is realizable as a specific class of passive spring-damper-inerter circuits, where the coefficients of $Q_i(s)$ are optimization variables.
  \item[2.] Formulate  a minimal state-space realization $\{ A_k, B_k, C_k, D_k \}$ of $K(s)$ by \eqref{eq: Realization Ak}.
  \item[3.] Formulate  $A_{cl}$, $B_{cl}$, and $C_{cl}$ by \eqref{eq: Acl Bcl}.
  \item[4.] Solve the following optimization problem to determine the optimal $J_{\mathrm{passive}}$ and the positive-real admittances $Q_i(s)$ for $i = 1, 2, \ldots, m$:
\begin{equation*}
\begin{split}
& \min_{\alpha_{i,j}, \beta_{i,k}}~  \| H_{w\rightarrow z} \|_2^2 =  \text{Trace}(C_{cl} P C_{cl}^\mathrm{T})   \\
& ~\text{s.t.}~~A_{cl}~\text{is stable}   \\
& ~~~~~~  P>0~\text{is the solution of   \eqref{eq: Lyapunov equation}}  \\
& ~~~~~~   Q_i(s)~\text{is a class of positive-real functions in Step~1}.
\end{split}
\end{equation*}
  \item[5.] By utilizing the results of passive circuit synthesis, realize the positive-real functions $Q_i(s)$ for $i = 1, 2, \ldots, m$ corresponding to the optimal performance as the admittances of the required spring-damper-inerter circuits.
\end{enumerate}
\end{procedure}

\begin{remark}
By properly modifying the objective function,
Procedure~\ref{procedure: 01} can be similarly applied to the control system design when the system performances are in other forms, such as  the $\mathcal{H}_{\infty}$ norm of transfer functions, the weighting sum of multiple performances, etc.
\end{remark}

As shown in Procedure~\ref{procedure: 01},
the circuit synthesis results can be utilized as the further optimization constraints in Step~4 and can be utilized to physically realize the positive-real admittances as passive mechanical circuits in Step~5.
Therefore, it is both theoretically and practically significant to solve the   realization problems of positive-real admittances in the form of \eqref{eq: Qi} as passive mechanical circuits containing the least number of elements, where this paper   investigates the  low-complexity passive circuit  synthesis problem when the McMillan degree of the admittance in \eqref{eq: Qi} is three.

\section{Proof of Lemma~5}   \label{appendix: A}

\setcounter{lemma}{0}
\renewcommand{\thelemma}{B.\arabic{lemma}}

\setcounter{equation}{0}
\renewcommand{\theequation}{B.\arabic{equation}}

To prove Lemma~\ref{lemma: Spring-Other-Network},
the following lemmas are presented.

\begin{lemma}   \label{lemma: Spring-Other-Network lemma a}
For  any admittance $Y(s)$ in \eqref{eq: specific bicubic admittance} that satisfies Assumption~\ref{assumption: 01} and does not satisfy the conditions of
Lemma~\ref{lemma: Spring-Biquadratic-Network}, if $Y(s)$ can be
realized by the one-port series-parallel damper-spring-inerter circuit  containing at most six elements as in Fig.~\ref{fig: Spring-N2}, which satisfies Assumption~\ref{assumption: 02}, then the graph of the subcircuit $N_2$ must have $k$-$\mathcal{P}(a, a')$
and cannot have any of $b$-$\mathcal{P}(a,a')$, $k$-$\mathcal{C}(a,a')$, or $b$-$\mathcal{C}(a,a')$.
\end{lemma}
\begin{proof}
The assumption that the conditions of Lemma~\ref{lemma: Spring-Biquadratic-Network} do not hold implies that the admittance (or impedance) of circuit $N_2$ is not in the biquadratic form, which is obtained by the partial removal of the pole of $Y(s)$ at $s = 0$. Therefore, it is implied that $k_1 \in (0, \alpha_0/\beta_1)$. Then, the admittance of $N_2$ is also in the form of
\eqref{eq: specific bicubic admittance} with all the coefficients being positive, which has a pole at $s = 0$, does not have any pole at $s = \infty$, and does not have any zero at $s = 0$ or $s = \infty$. By Lemma~\ref{lemma: graph constraint}, one can prove this lemma.
\end{proof}

\begin{lemma}   \label{lemma: Spring-Other-Network lemma b}
For any admittance $Y(s)$ in \eqref{eq: specific bicubic admittance} that satisfies Assumption~\ref{assumption: 01} and does not satisfy the conditions of
Lemma~\ref{lemma: Spring-Biquadratic-Network}, if $Y(s)$ can be realized by a  one-port series-parallel damper-spring-inerter circuit containing at most six elements as in Fig.~\ref{fig: Spring-N2}, then the subcircuit $N_2$ cannot contain  at most two types of elements.
\end{lemma}
\begin{proof}
As shown in the proof of Lemma~\ref{lemma: Spring-Other-Network lemma a}, the admittance $Y_2(s)$ of $N_2$ is also in the form of \eqref{eq: specific bicubic admittance} with positive coefficients. Therefore, it is clear that $N_2$ contains at least two types of elements. By Lemma~\ref{lemma: lossless subnetwork N1},   $N_2$ cannot be a spring-inerter circuit. By Lemma~\ref{lemma: Spring-Other-Network lemma a}, the requirement of $k$-$\mathcal{P}(a, a')$ implies that $N_2$ cannot be a damper-inerter circuit.

Assume that $N_2$ is a damper-spring circuit. Then, $Y_2(s)$ can be   in the form of
$Y_2(s) = H(s+z_1)(s+z_2)(s+z_3)/(s(s+p_1)(s+p_2))$ where $H > 0$ and $0 < z_1 < p_1 < z_2 < p_2 < z_3$ \cite[Chapter~6]{Van60}. By the \emph{second Foster form} \cite[Chapter~6]{Van60}, $Z_3(s) = 1/(Y_2(s) - k_2/s) = 1/(H_1/(s + p_1) + H_2/(s + p_2) + H_3)$ is a positive-real biquadratic impedance that can be realized by a damper-spring circuit, where $k_2$ is the residue of $Y_2(s)$ at $ s= 0$. Furthermore, by  \cite[Lemma 7]{JS11},
$Z_3(s)$ is regular. Therefore, Condition~1 of Lemma~\ref{lemma: Spring-Biquadratic-Network} holds, which contradicts the assumption.
\end{proof}

\begin{lemma}   \label{lemma: not realizable}
Any admittance  $Y(s)$ in \eqref{eq: specific bicubic admittance} satisfying Assumption~\ref{assumption: 01} cannot be realized by the one-port circuit configuration in Fig.~\ref{fig: figures in Appendix a}(a).
\end{lemma}
\begin{proof}
The admittance of the circuit in Fig.~\ref{fig: figures in Appendix a}(a) is calculated as $Y(s) = n(s)/d(s)$ where
$n(s) = b_1b_2c_1 s^4 + b_1b_2(k_1+k_3) s^3 + c_1(b_1k_1 + b_1k_3 + b_2k_1 + b_2k_2)s^2 + b_2(k_1k_2 + k_1k_3 + k_2k_3)s + c_1(k_1k_2 + k_1k_3 + k_2k_3)$ and $d(s) = b_1b_2 s^4 + c_1(b_1 + b_2) s^3 + b_2(k_2+k_3) s^2 + c_1(k_2+k_3) s$. Then, the resultant of $n(s)$ and $d(s)$ in $s$ is calculated as $R_0(n, d, s) = b_1b_2c_1^4(k_1k_2 + k_1k_3 + k_2k_3)(c_1^2(b_1k_3 - b_2k_2)^2 + b_1b_2^2k_2k_3^2)^2$, which can never be zero. Thus, the circuit in  Fig.~\ref{fig: figures in Appendix a}(a) cannot realize the admittance $Y(s)$ in this lemma, whose McMillan degree is three.
\end{proof}

\emph{Necessity.} To prove the necessity part, one will show that any circuit realizing the admittance $Y(s)$ in this lemma can be equivalent to one of the configurations in Fig.~\ref{fig: Spring-Biquadratic-Network-partial-removal}.
Since any one-port circuit whose graph is not connected or whose augmented graph is separable can be equivalent to another circuit satisfying Assumption~\ref{assumption: 02},
one only needs to discuss the circuits that satisfies Assumption~\ref{assumption: 02} to avoid repeated discussions.

When the circuit $N_2$ is  further the series connection of two circuits, by Lemma~\ref{lemma: Spring-Other-Network lemma a}, $Y(s)$ can be realized by the circuit structure in Fig.~\ref{fig: figures in Appendix a}(b) to guarantee the existence of $k$-$\mathcal{P}(a,a')$ and to avoid the existence of $k$-$\mathcal{C}(a,a')$, where  the springs $k_2$ and $k_3$ and two subcircuits $N_3$ and $N_4$ constitute $N_2$. Based on the symmetry, assume that $N_3$ contains one element and $N_4$ contains one or two elements. Then, by Lemmas~\ref{lemma: Spring-Other-Network lemma a}--\ref{lemma: not realizable},
the realization can always be equivalent to one of
the configurations in Figs.~\ref{fig: Case-6} and \ref{fig: Case-7}, where Fig.~\ref{fig: Case-6} can represent a five-element configuration when $c_2 = 0$ (open-circuited).

When the circuit $N_2$ is further the parallel connection of two circuits, to avoid repeated discussions, one assumes that $N_2$ cannot be  equivalent to the parallel circuit of a spring and a subcircuit due to the parallel spring $k_1$. Therefore, the subcircuit of $N_2$ providing the $k$-$\mathcal{P}(a,a')$ must contain at least two springs.
Furthermore, based on the equivalence in Fig.~\ref{fig: Network_Equivalence} and by Lemma~\ref{lemma: Spring-Other-Network lemma a}, the realization of $Y(s)$ can be a circuit whose structure is in Fig.~\ref{fig: figures in Appendix a}(c), where the damper $c_1$, springs $k_2$ and $k_3$, and two subcircuits $N_3$ and $N_4$ constitute $N_2$, and each of $N_3$ and $N_4$ contains only one element. Furthermore, by Lemmas~\ref{lemma: lossless subnetwork N1},
\ref{lemma: Spring-Other-Network lemma a}, and \ref{lemma: Spring-Other-Network lemma b}, it is implied that only the circuit in Fig.~\ref{fig: Case-8} is possible.

Therefore, the proof of the necessity part has been completed.

\emph{Sufficiency.} The sufficiency part of this lemma obviously holds.

\begin{figure}[thpb]
      \centering
      \subfigure[]{
      \includegraphics[scale=1.0]{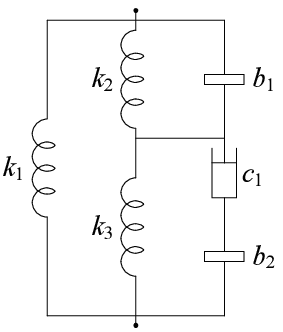}
      \label{fig: Fourth-Degree-Network-01}}
      \subfigure[]{
      \includegraphics[scale=1.0]{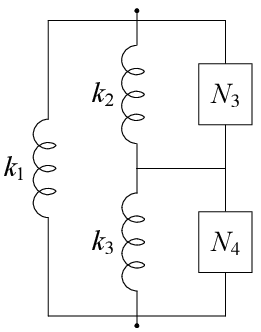}
      \label{fig: Spring-N3-N4}}
      \subfigure[]{
      \includegraphics[scale=1.0]{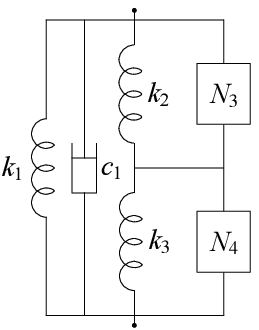}
      \label{fig: Spring-damper-N3-N4}}
      \caption{(a) The circuit configuration that cannot realize the admittance $Y(s)$ in Lemma~\ref{lemma: not realizable}; (b) the realization structure of $Y(s)$ discussed in the proof of Lemma~\ref{lemma: Spring-Other-Network}, where $k_2$, $k_3$, $N_3$, and $N_4$ constitute the circuit $N_2$ as in Fig.~1, $N_3$ contains one element, and $N_4$ contains one or two elements; (c) the realization structure of $Y(s)$ discussed in the proof of Lemma~\ref{lemma: Spring-Other-Network}, where $c_1$, $k_2$, $k_3$, $N_3$, and $N_4$ constitute the circuit $N_2$ as in  Fig.~1, and each of $N_3$ and $N_4$ contains only one element. }
      \label{fig: figures in Appendix a}
\end{figure}

\begin{figure}[thpb]
      \centering
      \subfigure[]{
      \includegraphics[scale=1.0]{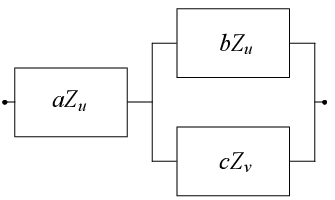}
      \label{fig: Equivalent-a}}
      \subfigure[]{
      \includegraphics[scale=1.0]{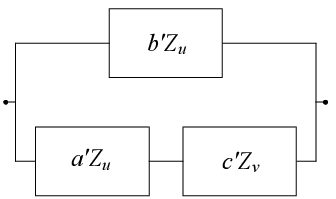}
      \label{fig: Equivalent-b}}
      \caption{Two equivalent one-port series-parallel passive circuit structures, where $a' = a(a+b)/b$, $b' = a + b$,   $c' = c(a+b)^2/b^2$, and $Z_u$ and $Z_v$ are positive-real impedances (See  \cite[Fig.~6]{JS11}). }
      \label{fig: Network_Equivalence}
\end{figure}

\section{Proof of Lemma~6}   \label{appendix: B}

\setcounter{equation}{0}
\renewcommand{\theequation}{C.\arabic{equation}}

\emph{Necessity.} The admittance of the circuit configuration in Fig.~\ref{fig: Case-6} is computed as $Y(s) = n(s)/d(s)$, where
$n(s)  = b_1c_1 s^3 + (b_1k_1 + b_1k_2 + c_1c_2) s^2 + (c_1k_1 + c_1k_3 + c_2k_1 + c_2k_2) s + k_1k_2 + k_1k_3 + k_2k_3$ and
$d(s)  = b_1 s^3 + (c_1 + c_2) s^2 + (k_2 + k_3) s$.
Since $Y(s)$ in  \eqref{eq: specific bicubic admittance} satisfying Assumption~\ref{assumption: 01} is realizable by the circuit in Fig.~\ref{fig: Case-6},
there exists $k>0$ such that $n(s) = k \alpha(s)$ and $d(s) = k \beta(s)$. Then, it follows that
\begin{subequations}
\begin{align}
b_1 c_1 &= k \alpha_3,  \label{eq: Case-6 eqn1}  \\
b_1k_1 + b_1k_2 + c_1c_2 &= k \alpha_2,  \label{eq: Case-6 eqn2}  \\
c_1k_1 + c_1k_3 + c_2k_1 + c_2k_2 &= k \alpha_1, \label{eq: Case-6 eqn3}  \\
k_1k_2 + k_1k_3 + k_2k_3 &= k \alpha_0, \label{eq: Case-6 eqn4}  \\
b_1 &= k \beta_3,   \label{eq: Case-6 eqn5}  \\
c_1 + c_2 &= k \beta_2,  \label{eq: Case-6 eqn6}  \\
k_2 + k_3 &= k \beta_1,  \label{eq: Case-6 eqn7}
\end{align}
\end{subequations}
where $k > 0$. Letting
\begin{equation} \label{eq: Case-6 lambda}
x = \frac{\alpha_3^2}{k \beta_3} > 0,
\end{equation}
the expression of $b_1$ can be obtained from \eqref{eq: Case-6 eqn5} as in
\eqref{eq: Case-6 element values}.
Then, it follows from \eqref{eq: Case-6 eqn1} and \eqref{eq: Case-6 eqn5} that
$c_1$ can be expressed as in \eqref{eq: Case-6 element values}, which together with
\eqref{eq: Case-6 eqn6} and \eqref{eq: Case-6 lambda} implies that $c_2$ can be expressed as in
\eqref{eq: Case-6 element values}.
Since $c_2 \geq 0$, one implies that $x$ satisfies \eqref{eq: Case-6 condition01}. It follows from \eqref{eq: Case-6 eqn7} and \eqref{eq: Case-6 lambda} that
\begin{equation}  \label{eq: Case-6 k3 k2}
k_3 = \frac{\alpha_3^2 \beta_1}{\beta_3 x} - k_2.
\end{equation}
Then, substituting
\eqref{eq: Case-6 eqn5},  \eqref{eq: Case-6 k3 k2}, and the expressions of $c_1$, $c_2$, and $b_1$ shown in \eqref{eq: Case-6 element values} into \eqref{eq: Case-6 eqn2} and \eqref{eq: Case-6 eqn3} yields the expressions of $k_1$ and $k_2$   as in \eqref{eq: Case-6 element values}, which together with \eqref{eq: Case-6 k3 k2} implies that $k_3$ can be expressed as in \eqref{eq: Case-6 element values}. By the assumption that $k_1 > 0$, $k_2 > 0$, and $k_3 > 0$, it is implied that $x$ satisfies \eqref{eq: Case-6 condition02} and \eqref{eq: Case-6 condition03}. Furthermore, substituting  \eqref{eq: Case-6 lambda} and the expressions of $k_1$, $k_2$, and $k_3$ as in \eqref{eq: Case-6 element values}
 into \eqref{eq: Case-6 eqn4}, one implies that $x$ is a positive root of the equation
\eqref{eq: Case-6 equation}. Now, the proof of the necessity part has been completed.

\emph{Sufficiency.} Let $c_1$, $c_2$, $k_1$, $k_2$, $k_3$, and $b_1$ satisfy the expressions in \eqref{eq: Case-6 element values}, where $x$ is a positive root of equation
\eqref{eq: Case-6 equation} such that \eqref{eq: Case-6 condition01}--\eqref{eq: Case-6 condition03}. Then, it can be verified that $c_2 \geq 0$ and other element values are  positive and finite.
Since \eqref{eq: Case-6 equation} holds, it can be verified that \eqref{eq: Case-6 eqn1}--\eqref{eq: Case-6 eqn7} hold with $k$ satisfying $k = \alpha_3^2/(\beta_3 x)$.  Therefore,   $Y(s)$ can be realized by the circuit in Fig.~\ref{fig: Case-6}.

\section{Proof of Lemma~9}   \label{appendix: C}

\setcounter{lemma}{0}
\renewcommand{\thelemma}{D.\arabic{lemma}}

\setcounter{equation}{0}
\renewcommand{\theequation}{D.\arabic{equation}}

To prove Lemma~\ref{lemma: Other Series-Parallel Structures}, the following lemmas are presented, which will be utilized for the proof.

\begin{lemma} \label{lemma: Other Series-Parallel Structures lemma 01}
For any admittance $Y(s)$ in \eqref{eq: specific bicubic admittance} that satisfies Assumption~\ref{assumption: 01} and does not satisfy the conditions of Theorem~\ref{theorem: main theorem 01}, if   $Y(s)$ is realizable as a one-port series-parallel damper-spring-inerter circuit containing at most six elements, then the circuit cannot contain no more than two types of elements.
\end{lemma}
\begin{proof}
The assumption that the conditions of Theorem~\ref{theorem: main theorem 01}
 do not hold implies that the series-parallel realizations of $Y(s)$ containing no more than six elements cannot be as  in Fig.~\ref{fig: Spring-N2}. Then, this lemma can be proved similar to  Lemma~\ref{lemma: Spring-Other-Network lemma b}.
\end{proof}

\begin{lemma} \label{lemma: Other Series-Parallel Structures lemma 02}
For any admittance $Y(s)$ in \eqref{eq: specific bicubic admittance} that satisfies Assumption~\ref{assumption: 01} and does not satisfy the conditions of Theorem~\ref{theorem: main theorem 01}, if  $Y(s)$ is realizable as a one-port series-parallel damper-spring-inerter circuit containing no more than six elements, which satisfies Assumption~\ref{assumption: 02} and is the parallel connection of two subcircuits $N_1$ and $N_2$, then the subcircuit  $N_2$ constituting $k$-$\mathcal{P}(a,a')$ must contain at least four elements and cannot contain   no more than two types of elements.
\end{lemma}
\begin{proof}
Since   the conditions of Theorem~\ref{theorem: main theorem 01}
do not hold, $Y(s)$ cannot be realizable by any circuit as in Fig.~\ref{fig: Spring-N2}, which contains at most six elements.
Therefore, the subnetwork  $N_2$ constituting $k$-$\mathcal{P}(a,a')$ must contain at least two springs, to avoid the single parallel spring in Fig.~\ref{fig: Spring-N2}. Moreover, together with the equivalence in Fig.~\ref{fig: Network_Equivalence}, there must be at least two other elements. Therefore, the total number of elements is at least four.
Assume that $N_2$ is a damper-spring circuit. Then, similar to the proof of Lemma~\ref{lemma: Spring-Other-Network lemma b}, this contradicts the assumption that $Y(s)$ cannot be realizable as in Fig.~\ref{fig: Spring-N2} containing no more than six elements.
Together with Lemma~\ref{lemma: lossless subnetwork N1}, one can prove that $N_2$ cannot contain   no more than two types of elements.
\end{proof}

\begin{lemma}  \label{lemma: Other Series-Parallel Structures lemma 03}
Any admittance $Y(s)$ in \eqref{eq: specific bicubic admittance} satisfying Assumption~\ref{assumption: 01} cannot be realized by any of the circuit configurations in Fig.~\ref{fig: non-realizable-2}.
\end{lemma}
\begin{proof}
The method of the proof is similar to  Lemma~\ref{lemma: not realizable}.
\end{proof}

\begin{figure}[thpb]
      \centering
      \subfigure[]{
      \includegraphics[scale=1.0]{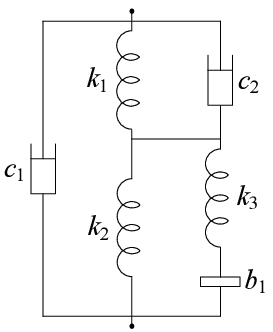}
      \label{fig: non-realizable-2-3}}
      \subfigure[]{
      \includegraphics[scale=1.0]{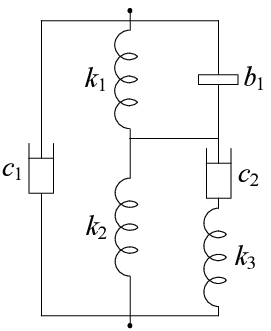}
      \label{fig: non-realizable-2-4}}
      \subfigure[]{
      \includegraphics[scale=1.0]{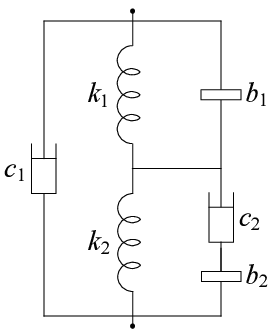}
      \label{fig: non-realizable-2-5}}
      \caption{The circuit configurations that cannot realize the admittance $Y(s)$ in Lemma~\ref{lemma: Other Series-Parallel Structures lemma 03}.}
      \label{fig: non-realizable-2}
\end{figure}

\emph{Necessity.}
Similar to the proof of Lemma~\ref{lemma: Spring-Other-Network},  one will show that any circuit realizing the admittance $Y(s)$ in this lemma
can be equivalent to one of the configurations in Figs.~\ref{fig: classes 2 and 3}--\ref{fig: classes 4 and 5}. To avoid repeated discussions,
one only needs to discuss the circuits that satisfies Assumption~\ref{assumption: 02}.

First, one will investigate the case when the circuit realizing $Y(s)$ is the parallel connection of two subcircuits $N_1$ and $N_2$.
By Lemma~\ref{lemma: graph constraint}, at least one of the two subcircuits must constitute $k$-$\mathcal{P}(a,a')$, which is assumed to be $N_2$.
Then, together with Lemmas~\ref{lemma: graph constraint}, \ref{lemma: lossless subnetwork N1}, and \ref{lemma: Other Series-Parallel Structures lemma 02}, one can imply that the realization can be equivalent to one of the structures in Fig.~\ref{fig: essential-parallel structures}, where $N_1$ contains only one element  and $N_2$ contains one or two elements for the structure in Fig.~\ref{fig: essential-parallel structures}(a), and each of $N_1$ and $N_2$ contains only one element for each structure in Figs.~\ref{fig: essential-parallel structures}(b) and \ref{fig: essential-parallel structures}(c).
Together with the realization constraints in Lemma~\ref{lemma: graph constraint} and the configurations that cannot realize $Y(s)$ as stated in Lemma~\ref{lemma: Other Series-Parallel Structures lemma 03}, $Y(s)$ is realizable by one of the circuit configurations in Fig.~\ref{fig: classes 2 and 3} by the method of enumeration.

Then, it turns to discuss the case when the realization of $Y(s)$ is the series  connection of two subcircuits $N_1$ and $N_2$. Here, one can assume that the number of elements in $N_1$ is no larger than the number of elements in $N_2$
without loss of generality. Thus, $N_1$ can only contain one, two, or three elements. By Lemma~\ref{lemma: graph constraint}, $N_1$ cannot contain only one element, to simultaneously guarantee $k$-$\mathcal{P}(a,a')$ and avoid $k$-$\mathcal{C}(a,a')$. When $N_1$ contain two elements, it is implied that $N_1$ can only be the parallel circuit of a damper and a spring,
since $N_1$ cannot be a spring-inerter circuit by Lemma~\ref{lemma: lossless subnetwork N1} and $N_1$ must contain
at least one spring to form $k$-$\mathcal{P}(a,a')$.
Furthermore, one can prove that $N_2$ can only be the parallel connection of two subcircuits.
To guarantee $k$-$\mathcal{P}(a,a')$ and together with the equivalence in Fig.~\ref{fig: Network_Equivalence}, $N_2$ can always be equivalent to the parallel connection of a spring and a subcircuit $N_3$, where $N_3$ contains two or three elements. The structure is shown in Fig.~\ref{fig: essential-parallel structures}(d).
By the method of enumeration and together with Lemmas~\ref{lemma: lossless subnetwork N1} and \ref{lemma: Other Series-Parallel Structures lemma 01},
$Y(s)$ is realizable by one of the circuit configurations in Figs.~\ref{fig: 4-1}--\ref{fig: 4-6}, where the circuit configuration in Fig.~\ref{fig: 4-1} can represent a five-element series-parallel circuit configuration when $c_2^{-1} = 0$ or $c_3 = 0$ ($c_2^{-1} = 0$ and $c_3 = 0$ cannot simultaneously hold).
Similarly, when $N_1$ contain three elements, one can show that any realization of $Y(s)$ can always be equivalent to the structure in
Fig.~\ref{fig: essential-parallel structures}(e), where both $N_3$ and $N_4$ contain two elements.
Furthermore, by the method of enumeration and together with
Lemmas~\ref{lemma: lossless subnetwork N1} and \ref{lemma: Other Series-Parallel Structures lemma 01}, $Y(s)$ is realizable by one of the circuit configurations in Figs.~\ref{fig: 5-1} and \ref{fig: 5-2}.

Therefore, the necessity part of the proof has been completed.

\emph{Sufficiency.} The sufficiency part of this lemma obviously holds.

\begin{figure}[thpb]
      \centering
      \subfigure[]{
      \includegraphics[scale=1.0]{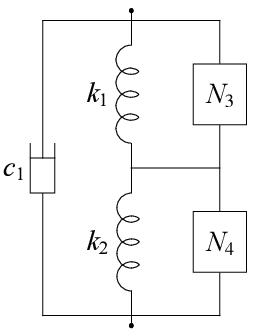}
      \label{fig: damper-N3-N4}}
      \subfigure[]{
      \includegraphics[scale=1.0]{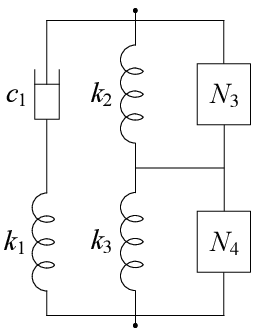}
      \label{fig: damper_spring-N3-N4}}
      \subfigure[]{
      \includegraphics[scale=1.0]{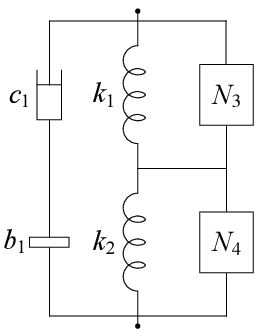}
      \label{fig: damper_inerter-N3-N4}}
      \subfigure[]{
      \includegraphics[scale=1.0]{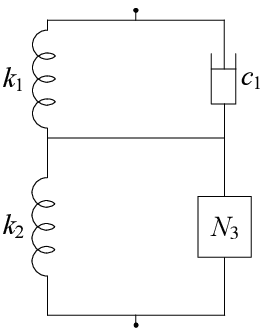}
      \label{fig: spring-damper_spring-N3}}
      \subfigure[]{
      \includegraphics[scale=1.0]{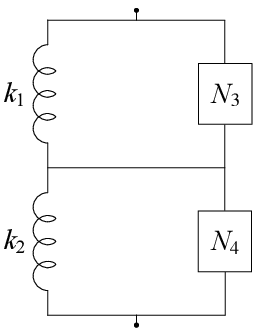}
      \label{fig: spring-N3_spring-N4}}
      \caption{The structures of the circuits realizing $Y(s)$ discussed in the proof of
      Lemma~\ref{lemma: Other Series-Parallel Structures lemma 02}. For the structure in (a),  $N_3$ contains only one element,  and $N_4$ contains two or three elements; for the structures in (b) and (c), each of $N_3$ and $N_4$ contains one element; for the structure in (d), $N_3$ contains two or three elements; for the structure  in  (e), each of $N_3$ and $N_4$ contains two elements.  }
      \label{fig: essential-parallel structures}
\end{figure}

\section{Proof of Lemma~12}   \label{appendix: D}

\setcounter{equation}{0}
\renewcommand{\theequation}{E.\arabic{equation}}

\emph{Necessity.} The admittance of the circuit configuration in Fig.~\ref{fig: 3-1} is computed as $Y(s) = n(s)/d(s)$, where
$n(s) = b_1 c_1 c_2 s^4 + b_1 (c_1 k_1 + c_1 k_2 + c_2 k_1) s^3 + (b_1 k_1 k_2 + c_1 c_2 k_1 + c_1 c_2 k_3) s^2 + (c_1 k_1 k_2 + c_1 k_2 k_3 + c_1 k_1 k_3 + c_2 k_1 k_3) s + k_1 k_2 k_3$ and $d(s) = b_1 c_1 s^4 + (b_1 k_1 + c_1 c_2) s^3 + (c_1 k_2 + c_1 k_3 + c_2 k_1) s^2 + k_1 (k_3 + k_3) s$.
If the given admittance $Y(s)$ of this lemma is realizable by the circuit in Fig.~\ref{fig: 3-1}, then the resultant of $n(s)$ and $d(s)$ in $s$ calculated as $R_0 (n, d, s) = - b_1 c_1^2 k_1^3 k_2 k_3 (b_1 k_2^2 + c_2^2 k_3)^2 (b_1 k_1^2 + c_1^2 k_2 + c_1^2 k_3 - c_1 c_2 k_1)^2$ is zero. Therefore, one obtains that $c_2$ satisfies the expression in
\eqref{eq: 3-1 element values}, which further implies that the admittance of the configuration in Fig.~\ref{fig: 3-1} becomes
$Y(s) = \tilde{n}(s) / \tilde{d}(s)$, where $\tilde{n}(s) = b_1 (b_1 k_1^2 + k_2 c_1^2 + k_3 c_1^2) s^3 + b_1 c_1 k_1 (k_1 + k_2) s^2 + (k_3 (b_1 k_1^2 + c_1^2 k_2 + c_1^2 k_3) + c_1^2 k_1 (k_2 + k_3)) s + c_1 k_1 k_2 k_3$ and $\tilde{d}(s) = b_1 c_1 k_1 s^3 + (b_1 k_1^2 + c_1^2 k_2 + c_1^2 k_3) s^2 + c_1 k_1 (k_2 + k_3) s$. Therefore, there exists $k > 0$ such that $\tilde{n}(s) = k \alpha(s)$ and $\tilde{d}(s) = k \beta(s)$.
Then, it follows that
\begin{subequations}
\begin{align}
b_1 (b_1 k_1^2 + k_2 c_1^2 + k_3 c_1^2) &= k \alpha_3,  \label{eq: 3-1 eqn1}  \\
b_1 c_1 k_1 (k_1 + k_2) &= k \alpha_2,   \label{eq: 3-1 eqn2}  \\
k_3 (b_1 k_1^2 + c_1^2 k_2 + c_1^2 k_3) + c_1^2 k_1 (k_2 + k_3) &= k \alpha_1, \label{eq: 3-1 eqn3}  \\
c_1 k_1 k_2 k_3  &=  k \alpha_0, \label{eq: 3-1 eqn4}  \\
b_1 c_1 k_1 &=  k \beta_3,   \label{eq: 3-1 eqn5}  \\
b_1 k_1^2 + c_1^2 k_2 + c_1^2 k_3    &=  k \beta_2,  \label{eq: 3-1 eqn6}  \\
c_1 k_1 (k_2 + k_3) &= k \beta_1,  \label{eq: 3-1 eqn7}
\end{align}
\end{subequations}
where $k > 0$. From \eqref{eq: 3-1 eqn1} and \eqref{eq: 3-1 eqn6}, it is implied that the value of $b_1$ can be expressed as in \eqref{eq: 3-1 element values}. Let $x$ satisfies
\begin{equation} \label{eq: 3-1 lambda k1}
x = \beta_2 (\alpha_2 - \beta_3 k_1).
\end{equation}
Then, the expression of $k_1$ can be obtained as in \eqref{eq: 3-1 element values}. Together with \eqref{eq: 3-1 lambda k1}, it is implied from \eqref{eq: 3-1 eqn2} and \eqref{eq: 3-1 eqn5} that $k_2$ can be expressed as in \eqref{eq: 3-1 element values}, which implies that $x > 0$.
By \eqref{eq: 3-1 eqn5}, \eqref{eq: 3-1 eqn7}, \eqref{eq: 3-1 lambda k1}, and the expression of $b_1$ as in \eqref{eq: 3-1 element values}, the expression of $k_3$ can be derived as in \eqref{eq: 3-1 element values}. It follows from \eqref{eq: 3-1 eqn4} and \eqref{eq: 3-1 eqn7} that
$k_2^{-1} + k_3^{-1} = \beta_1/\alpha_0$, which together with the expressions of $k_2$ and $k_3$ as in \eqref{eq: 3-1 element values} implies that
\begin{equation} \label{eq: 3-1 lambda equation}
x^2 - \alpha_3 \beta_1 x + \alpha_0 \alpha_3 \beta_2 \beta_3 = 0.
\end{equation}
Therefore, it is clear that \eqref{eq: 3-1 condition00} holds and the root of equation \eqref{eq: 3-1 lambda equation} can be solved as in \eqref{eq: 3-1 condition02}. Substituting \eqref{eq: 3-1 eqn6} and \eqref{eq: 3-1 eqn7} into \eqref{eq: 3-1 eqn3} implies that $k_3 \beta_2 + c_1 \beta_1 = \alpha_1$, which together with the expressions of $c_1$ and $k_3$ implies that the value of $c_1$ can be expressed as in \eqref{eq: 3-1 element values}. By the element values of $c_1$, $k_1$, $k_2$, and $k_3$, it follows from \eqref{eq: 3-1 eqn7} that
\begin{equation} \label{eq: 3-1 x}
k = \frac{\alpha_3 (\alpha_2 \beta_2 - x)(x - \mathcal{B}_{23})}{\beta_1\beta_2^2\beta_3^3}.
\end{equation}
Furthermore, by the expressions of $c_1$, $k_1$, $k_2$, $k_3$ and $b_1$, it follows from \eqref{eq: 3-1 eqn6} that \eqref{eq: 3-1 equation} holds. By the assumption that the element values expressed in \eqref{eq: 3-1 element values} are positive and finite, one implies that $x > 0$  satisfies \eqref{eq: 3-1 condition01}. The proof of the necessity part has been completed.

\emph{Sufficiency.}
Let the element values of $c_1$, $c_2$, $k_1$, $k_2$, $k_3$, and
$b_1$ satisfy \eqref{eq: 3-1 element values}, where
\eqref{eq: 3-1 condition00} holds,
and $x$ is a positive root of \eqref{eq: 3-1 equation}, such that \eqref{eq: 3-1 condition01} and \eqref{eq: 3-1 condition02} hold. Then, it can be verified that  the element values are positive and finite. Moreover, together with the expression of $c_2$ in \eqref{eq: 3-1 element values}, it can be calculated that the admittance of the configuration in Fig.~\ref{fig: 3-1} is $Y(s) = \tilde{n}(s) / \tilde{d}(s)$, where $\tilde{n}(s) = b_1 (b_1 k_1^2 + k_2 c_1^2 + k_3 c_1^2) s^3 + b_1 c_1 k_1 (k_1 + k_2) s^2 + (k_3 (b_1 k_1^2 + c_1^2 k_2 + c_1^2 k_3) + c_1^2 k_1 (k_2 + k_3)) s + c_1 k_1 k_2 k_3$ and $\tilde{d}(s) = b_1 c_1 k_1 s^3 + (b_1 k_1^2 + c_1^2 k_2 + c_1^2 k_3) s^2 + c_1 k_1 (k_2 + k_3) s$. Since $x$ satisfies \eqref{eq: 3-1 equation} and \eqref{eq: 3-1 condition02}, it can be verified that \eqref{eq: 3-1 eqn1}--\eqref{eq: 3-1 eqn7} hold. Therefore,  $Y(s)$ is realizable by the circuit in Fig.~\ref{fig: 3-1}.

\end{appendices}

\newpage

\setcounter{figure}{0}
\setcounter{equation}{0}
\setcounter{section}{0}

\numberwithin{equation}{section}
\numberwithin{figure}{section}
\numberwithin{lemma}{section}
\numberwithin{theorem}{section}
\numberwithin{definition}{section}

\begin{center}
\LARGE Supplementary Material to: Series-Parallel Mechanical Circuit Synthesis of a
Positive-Real Third-Order Admittance Using at Most
Six Passive Elements for Inerter-Based Control
\end{center}

\begin{center}
\large Kai~Wang,  ~Michael~Z.~Q.~Chen, ~ and ~ Fei Liu
\end{center}

\vspace{0.5cm}

 \section{Introduction}
This report presents some supplementary material to  the paper entitled   ``Series-parallel mechanical circuit synthesis of a positive-real third-order admittance using at most six passive elements for inerter-based control'' \cite{IJCTA_sub}, which are omitted from the original paper for brevity. It is assumed that
the numbering of lemmas, theorems, equations and figures in this report agrees with that in the original paper.

\section{Realizability Conditions of the Configurations in Fig.~3}
\label{appendix: Lemmas Cases-6 to 8}

\subsection{The Configuration in Fig.~3(b) (The Proof of Lemma~7)}

\emph{Necessity.} The admittance of the circuit configuration in Fig.~3(b) is calculated as $Y(s) = n(s)/d(s)$, where
$n(s) = b_1c_1c_2 s^3 + b_1(c_1k_1 + c_1k_3 + c_2k_1 + c_2k_2) s^2 +
(b_1k_1k_2 + b_1k_1k_3 + b_1k_2k_3 + c_1c_2k_1 + c_1c_2k_3) s
+ c_2 (k_1k_2 + k_2k_3 + k_1k_3)$
and $d(s) = b_1(c_1+c_2)s^3 + (b_1k_2 + b_1k_3 + c_1c_2)s^2 + c_2(k_2+k_3)s$.
Since $Y(s)$ is realizable by the circuit in Fig.~3(b),
there exists $k > 0$ such that $n(s) = k \alpha(s)$ and $d(s) = k \beta(s)$. Then,
it follows that
\begin{subequations}
\begin{align}
b_1c_1c_2 &= k \alpha_3,  \label{eq: Case-7 eqn1}  \\
b_1(c_1k_1 + c_1k_3 + c_2k_1 + c_2k_2) &= k \alpha_2,
\label{eq: Case-7 eqn2}  \\
b_1k_1k_2 + b_1k_1k_3 + b_1k_2k_3 + c_1c_2k_1 + c_1c_2k_3 &= k \alpha_1,
\label{eq: Case-7 eqn3}  \\
c_2 (k_1k_2 + k_2k_3 + k_1k_3)  &=  k \alpha_0,
\label{eq: Case-7 eqn4}  \\
b_1(c_1+c_2) &= k \beta_3, \label{eq: Case-7 eqn5}  \\
b_1k_2 + b_1k_3 + c_1c_2  &= k \beta_2, \label{eq: Case-7 eqn6}  \\
c_2(k_2+k_3)  &=  k \beta_1,   \label{eq: Case-7 eqn7}
\end{align}
\end{subequations}
where $k > 0$.
Then, it follows from \eqref{eq: Case-7 eqn1} and \eqref{eq: Case-7 eqn5} that
$1/c_1 + 1/c_2 = \beta_3/\alpha_3$, which  implies the expression of $c_2$ as in
\begin{equation}  \label{eq: Case-7 c2 c1}
c_2 = \frac{\alpha_3 c_1}{\beta_3 c_1 - \alpha_3}.
\end{equation}
Furthermore, by \eqref{eq: Case-7 eqn4} and \eqref{eq: Case-7 eqn7},  the expression of $k_3$ can be obtained as in
\begin{equation}   \label{eq: Case-7 k3 k1 k2}
k_3 = \frac{k_2 (\alpha_0 - \beta_1 k_1)}{\beta_1 (k_1 + k_2) - \alpha_0}.
\end{equation}
It follows from \eqref{eq: Case-7 eqn1} and \eqref{eq: Case-7 eqn7} that
$(k_2 + k_3)/(b_1c_1) = \beta_1/\alpha_3$, which implies the expression of $b_1$ as in
\begin{equation}   \label{eq: Case-7 k3 k1 k2}
b_1 = \frac{\alpha_3 k_2^2}{c_1 (\beta_1 (k_1 + k_2) - \alpha_0)},
\end{equation}
which implies that $\beta_1 (k_1 + k_2) - \alpha_0 > 0$. Substituting
\eqref{eq: Case-7 c2 c1}--\eqref{eq: Case-7 k3 k1 k2} into \eqref{eq: Case-7 eqn1}
yields
\begin{equation}  \label{eq: Case-7 x}
k = \frac{\alpha_3 c_1 k_2^2}{(\beta_3 c_1 - \alpha_3)(\beta_1(k_1 + k_2)-\alpha_0)}.
\end{equation}
Let
\begin{equation} \label{eq: lambda mu nu}
x = k_1 > 0, ~~~ y = k_2 > 0.
\end{equation}
By \eqref{eq: Case-7 c2 c1}--\eqref{eq: Case-7 k3 k1 k2} the element values of $c_2$, $k_3$, and $b_1$ can be further expressed as in (7), which together with \eqref{eq: Case-7 eqn3} and \eqref{eq: Case-7 eqn6} implies that $c_1$ can be expressed as in (7). Substituting the element values in (7) and $k$ expressed as in \eqref{eq: Case-7 x} into \eqref{eq: Case-7 eqn2} and \eqref{eq: Case-7 eqn3} implies (6a) and (6b), respectively. The assumption that the element values are positive and finite implies that (6c) and (6d)  hold.
Now, the necessity part is proved.

\emph{Sufficiency.}
Let  $c_1$, $c_2$, $k_1$, $k_2$, $k_3$, and $b_1$ satisfy
(7), where
$x > 0$ and $y > 0$ are positive roots of the equations in (6a) and (6b) such that (6c) and (6d) hold.   Then, it can be verified that the element values are positive and finite. Since (6a) and (6b) hold, it can be verified that \eqref{eq: Case-7 eqn1}--\eqref{eq: Case-7 eqn7} hold with $k$ satisfying \eqref{eq: Case-7 x}.  Therefore,   $Y(s)$ is realizable by the circuit in Fig.~3(b).

\subsection{The Configuration in Fig.~3(c) (The Proof of Lemma~8)}

\emph{Necessity.} The admittance of the circuit configuration in Fig.~3(c) is calculated as $Y(s) = n(s)/d(s)$ where
$n(s) = b_1(c_1+c_2)s^3 + (b_1k_1 + b_1k_2 + c_1c_2)s^2 +
(c_1k_2 + c_1k_3 + c_2k_1 + c_2k_3)s + k_1k_2 + k_2k_3 + k_1k_3$
and $d(s) = b_1 s^3 + c_2 s^2 + (k_2 + k_3) s$.  Since $Y(s)$ is realizable by the circuit in Fig.~3(c),
there exists $k > 0$ such that $n(s) = k \alpha(s)$ and $d(s) = k \beta(s)$. Then, it follows that
\begin{subequations}
\begin{align}
b_1(c_1+c_2) &= k \alpha_3,  \label{eq: Case-8 eqn1}  \\
b_1k_1 + b_1k_2 + c_1c_2 &= k \alpha_2,  \label{eq: Case-8 eqn2}  \\
c_1k_2 + c_1k_3 + c_2k_1 + c_2k_3 &= k \alpha_1, \label{eq: Case-8 eqn3}  \\
k_1k_2 + k_2k_3 + k_1k_3 &= k \alpha_0, \label{eq: Case-8 eqn4}  \\
b_1 &= k \beta_3,   \label{eq: Case-8 eqn5}  \\
c_2 &= k \beta_2,  \label{eq: Case-8 eqn6}  \\
k_2 + k_3 &= k \beta_1,  \label{eq: Case-8 eqn7}
\end{align}
\end{subequations}
where $k > 0$. Substituting \eqref{eq: Case-8 eqn5} and \eqref{eq: Case-8 eqn6} into \eqref{eq: Case-8 eqn1} implies
\begin{equation}  \label{eq: Case-8 c1 x}
c_1 = \frac{\alpha_3 - \beta_2\beta_3 k}{\beta_3},
\end{equation}
which together with \eqref{eq: Case-8 eqn2},
\eqref{eq: Case-8 eqn5}, and \eqref{eq: Case-8 eqn6} implies that
\begin{equation}    \label{eq: Case-8 k2 k1 x}
k_2 = \frac{\beta_2^2\beta_3 k - \beta_3^2 k_1 - \mathcal{B}_{33}}{\beta_3^2}.
\end{equation}
Substituting \eqref{eq: Case-8 eqn6}--\eqref{eq: Case-8 k2 k1 x} into
\eqref{eq: Case-8 eqn3} can imply
\begin{equation}     \label{eq: Case-8 k3 k1 x}
k_3 = \frac{\beta_1\beta_2\beta_3 k - \beta_2\beta_3 k_1 - \mathcal{B}_{23}}{\beta_2\beta_3}.
\end{equation}
By \eqref{eq: Case-8 k2 k1 x} and \eqref{eq: Case-8 k3 k1 x}, it follows from
\eqref{eq: Case-8 eqn7} that
\begin{equation}   \label{eq: Case-8 x k1}
k = \frac{2\beta_2 \beta_3^2 k_1 + \beta_3 \mathcal{B}_{23} + \beta_2 \mathcal{B}_{33}}{\beta_2^3 \beta_3}.
\end{equation}
Letting
\begin{equation}  \label{eq: Case-8 lambda k1}
x = 2 \beta_2\beta_3 k_1 > 0,
\end{equation}
it follows from \eqref{eq: Case-8 eqn5}, \eqref{eq: Case-8 eqn6}, and
\eqref{eq: Case-8 c1 x}--\eqref{eq: Case-8 x k1} that the values of $c_1$, $c_2$, $k_1$, $k_2$, $k_3$, and $b_1$ can be expressed as in
(9). The assumption that the element values are positive and finite implies that $x$ satisfies
(8b) and (8c). Finally, substituting the expressions of $k_1$, $k_2$, and $k_3$ as in (9) into
\eqref{eq: Case-8 eqn4}, it is implied that $x > 0$ is a positive root of the equation
(8a). Now, the proof of the necessity part has been completed.

\emph{Sufficiency.}
Let the values of $c_1$, $c_2$, $k_1$, $k_2$, $k_3$, and $b_1$ satisfy
(9), where $x$ is a positive root of equation (8a) such that
(8b) and (8c) hold. Then, it can be verified that all the element values are positive and finite.
Since (8a) holds, it can be verified that \eqref{eq: Case-8 eqn1}--\eqref{eq: Case-8 eqn7} hold with $k$ satisfying \eqref{eq: Case-8 x k1}.  Therefore,   $Y(s)$ is realizable by  the circuit in Fig.~3(c).

\section{Realizability Conditions of the Configurations in Fig.~4}
\label{appendix: Lemmas classes 2 and 3}

\subsection{The Configuration in Fig.~4(a) (The Proof of Lemma~10)}

\emph{Necessity.} The admittance of the circuit configuration in Fig.~4(a) is computed as $Y(s) = n(s)/d(s)$ where
$n(s) = b_1 (c_1 + c_2) s^3 + (b_1 k_1 + c_1c_2 + c_2c_3 + c_1c_3) s^2 + (c_1 k_1 + c_1 k_2 + c_2 k_2 + c_3 k_1) s + k_1k_2$
and $d(s) = b_1 s^3 + (c_2 + c_3) s^2 + (k_1 + k_2) s$.  Since $Y(s)$ can be realized by the circuit in Fig.~4(a),
there exists $k > 0$ such that $n(s) = k \alpha(s)$ and $d(s) = k \beta(s)$. Then,
it follows that
\begin{subequations}
\begin{align}
b_1 (c_1 + c_2) &= k \alpha_3,  \label{eq: 2-1 eqn1}  \\
b_1 k_1 + c_1c_2 + c_2c_3 + c_1c_3 &= k \alpha_2,  \label{eq: 2-1 eqn2}  \\
c_1 k_1 + c_1 k_2 + c_2 k_2 + c_3 k_1 &= k \alpha_1, \label{eq: 2-1 eqn3}  \\
k_1 k_2 &= k \alpha_0, \label{eq: 2-1 eqn4}  \\
b_1 &= k \beta_3,   \label{eq: 2-1 eqn5}  \\
c_2 + c_3 &= k \beta_2,  \label{eq: 2-1 eqn6}  \\
k_1 + k_2 &= k \beta_1,  \label{eq: 2-1 eqn7}
\end{align}
\end{subequations}
where $k > 0$. Then, it follows from \eqref{eq: 2-1 eqn1} and \eqref{eq: 2-1 eqn5} that
$c_1 + c_2 = \alpha_3/\beta_3$,
which together with \eqref{eq: 2-1 eqn6} implies that
\begin{equation} \label{eq: 2-1 c2 c3}
c_2 = \frac{\alpha_3 - \beta_3 c_1}{\beta_3}, ~~~ c_3 = \frac{\beta_2\beta_3 k + \beta_3 c_1 - \alpha_3}{\beta_3}.
\end{equation}
Let $x$ satisfies
\begin{equation} \label{eq: 2-1 lambda}
x  = k_1 > 0.
\end{equation}
From \eqref{eq: 2-1 eqn4} and \eqref{eq: 2-1 eqn7}, one obtains
\begin{equation} \label{eq: 2-1 x}
k = \frac{x^2}{\beta_1 x - \alpha_0},
\end{equation}
which implies that $\beta_1 x - \alpha_0 > 0$.
Substituting \eqref{eq: 2-1 eqn7} and \eqref{eq: 2-1 c2 c3}--\eqref{eq: 2-1 x} into \eqref{eq: 2-1 eqn3}, it is implied that $c_1$ can be expressed as in (11). Furthermore, substituting the expression of $k$ in \eqref{eq: 2-1 x} into \eqref{eq: 2-1 eqn5}, \eqref{eq: 2-1 eqn7}, and \eqref{eq: 2-1 c2 c3}, one can obtain the expressions of $b_1$, $k_2$, $c_2$, and $c_3$ as in (11). Together with the element value expressions in (11) and the expression of $k$ in \eqref{eq: 2-1 x}, it follows from \eqref{eq: 2-1 eqn2} that  (10a) holds. The assumption that $c_3 \geq 0$ and other element values are positive and finite implies that $x$ satisfies
(10b)--(10d). Now, the proof of the necessity part has been completed.

\emph{Sufficiency.}
Let  $c_1$, $c_2$, $c_3$, $k_1$, $k_2$, and $b_1$ satisfy
(11), where $x$ is a positive root of equation (10a) such that
(10b)--(10d).  Then, it can be verified that $c_3 \geq 0$ and other element values can be positive and finite.
Since (10a) holds, one can verify that \eqref{eq: 2-1 eqn1}--\eqref{eq: 2-1 eqn7} hold with $k$ satisfying \eqref{eq: 2-1 x}.  Therefore, $Y(s)$ is realizable by the circuit configuration in Fig.~4(a).

\subsection{The Configuration in Fig.~4(b) (The Proof of Lemma~11)}

\emph{Necessity.} The admittance of the circuit configuration in Fig.~4(b) is calculated as $Y(s) = n(s)/d(s)$ where
$n(s) = b_1 (c_1 c_2 + c_2 c_3 + c_1 c_3) s^3 + (c_1(b_1 k_1 + b_1 k_2 + c_2 c_3) + b_1 (c_2 k_2 + c_3 k_1)) s^2 + (b_1k_1k_2 + c_1c_3 k_1 + c_1c_3 k_2 + c_2c_3 k_2) s + c_3 k_1 k_2$ and $d(s) = b_1(c_2+c_3)s^3 + (b_1k_1 + b_1k_2 + c_2c_3) s^2 + c_3(k_1+k_2)s$.
Since the admittance $Y(s)$  is realizable by the circuit in Fig.~4(b), there exists $k > 0$ such that $n(s) = k \alpha(s)$ and $d(s) = k \beta(s)$.
Then, it follows that
\begin{subequations}
\begin{align}
b_1 (c_1 c_2 + c_2 c_3 + c_1 c_3) &= k \alpha_3,  \label{eq: 2-2 eqn1}  \\
c_1(b_1 k_1 + b_1 k_2 + c_2 c_3) + b_1 (c_2 k_2 + c_3 k_1) &= k \alpha_2,
\label{eq: 2-2 eqn2}  \\
b_1k_1k_2 + c_1c_3 k_1 + c_1c_3 k_2 + c_2c_3 k_2  &= k \alpha_1, \label{eq: 2-2 eqn3}  \\
c_3 k_1 k_2  &=  k \alpha_0, \label{eq: 2-2 eqn4}  \\
b_1(c_2+c_3) &=  k \beta_3,   \label{eq: 2-2 eqn5}  \\
b_1k_1 + b_1k_2 + c_2c_3    &=  k \beta_2,  \label{eq: 2-2 eqn6}  \\
c_3(k_1+k_2) &= k \beta_1,  \label{eq: 2-2 eqn7}
\end{align}
\end{subequations}
where $k > 0$. It follows from \eqref{eq: 2-2 eqn4} and \eqref{eq: 2-2 eqn7} that $1/k_1 + 1/k_2 = \beta_1/\alpha_0$, which implies that
\begin{equation}  \label{eq: 2-2 element values k2 k1}
k_2 = \frac{\alpha_0 k_1}{\beta_1 k_1 - \alpha_0}.
\end{equation}
The assumption that $k_1 > 0$ and $k_2 > 0$ implies that $\beta_1 k_1 - \alpha_0 > 0$.
Similarly, it follows from \eqref{eq: 2-2 eqn1} and \eqref{eq: 2-2 eqn5} that
\begin{equation}  \label{eq: 2-2 element values c1 c2 c3}
c_1 = \frac{\alpha_3 (c_2 + c_3) - \beta_3 c_2 c_3}{\beta_3 (c_2 + c_3)}.
\end{equation}
Substituting \eqref{eq: 2-2 element values k2 k1} into \eqref{eq: 2-2 eqn4} yields
\begin{equation}  \label{eq: 2-2 element values c3 k1}
c_3 = \frac{k(\beta_1 k_1 - \alpha_0)}{k_1^2},
\end{equation}
which together with \eqref{eq: 2-2 eqn5} implies that $b_1$ can be expressed as in
\begin{equation}  \label{eq: 2-2 element values b1 k1}
b_1 = \frac{k \beta_3 k_1^2}{(\beta_1 k_1 - \alpha_0)k + c_2 k_1^2}.
\end{equation}
Therefore, it is implied that  $(\beta_1 k_1 - \alpha_0)k + c_2 k_1^2 > 0$.
Let
\begin{equation}  \label{eq: 2-2 lambda nu}
x = k_1 > 0, ~~~ y = c_2 > 0.
\end{equation}
Then, by \eqref{eq: 2-2 eqn6}, \eqref{eq: 2-2 element values k2 k1}, and \eqref{eq: 2-2 element values c3 k1}--\eqref{eq: 2-2 lambda nu},
one obtains
\begin{equation}   \label{eq: 2-2 x}
k = \frac{x^2((\beta_1 x - \alpha_0)^2 y^2  - \beta_2  x^2 (\beta_1 x - \alpha_0) y  + \beta_1\beta_3 x^4)}{(\beta_1 x - \alpha_0)^2(\beta_2 x^2 - (\beta_1 x - \alpha_0) y)}.
\end{equation}
Since   $\beta_2 x^2 - (\beta_1 x - \alpha_0) y   < 0$ can imply
$(\beta_1 x - \alpha_0)^2 y^2  - \beta_2  x^2 (\beta_1 x - \alpha_0) y  + \beta_1\beta_3 x^4 > 0$, it follows from \eqref{eq: 2-2 x} and $k > 0$ that $\beta_2 x^2 - (\beta_1 x - \alpha_0) y   > 0$ and $(\beta_1 x - \alpha_0)^2 y^2  - \beta_2  x^2 (\beta_1 x - \alpha_0) y  + \beta_1\beta_3 x^4 > 0$.
Combining \eqref{eq: 2-2 element values k2 k1}--\eqref{eq: 2-2 x}, the element values can be expressed as in (13), which together with \eqref{eq: 2-2 eqn2} and \eqref{eq: 2-2 eqn3} implies that (12a) and (12b) hold. The assumption that the element values are positive and finite implies that (12c) and (12d) hold.
The proof of the  necessity part has been completed.

\emph{Sufficiency.}
Let the element values of $c_1$, $c_2$, $c_3$, $k_1$, $k_2$, and $b_1$
 satisfy (13), where $x > 0$ and $y > 0$ are positive roots of the equations
(12a) and (12b), such that
(12c) and (12d) hold.
  Then, it can be verified that the element values are positive and finite.
Since (12a) and (12b) hold, it can be verified that \eqref{eq: 2-2 eqn1}--\eqref{eq: 2-2 eqn7} hold with $k$ satisfying
\eqref{eq: 2-2 x}. Therefore,  $Y(s)$   is realizable by  the circuit configuration in Fig.~4(b).

\subsection{The Configuration in Fig.~4(d) (The Proof of Lemma~13)}

\emph{Necessity.} The admittance of the configuration in Fig.~4(d) is calculated as $Y(s) = \alpha(s)/\beta(s)$, where
$\alpha(s) = b_1 b_2 (c_1 + c_2) s^4 + (b_1b_2k_1 + b_1c_1c_2 + b_2c_1c_2) s^3 +
(b_1 c_1 k_1 + b_1 c_1 k_2 + b_1 c_2 k_2 + b_2 c_1 k_1) s^2 + k_2 (b_1 k_1 +
c_1 c_2) s + c_1 k_1 k_2$ and $\beta(s) = b_1b_2 s^4 + (b_1 c_2 + b_2 c_1) s^3 +
(b_1 k_1 + b_1 k_2 + c_1c_2) s^2 + c_1 (k_1 + k_2) s$.
If the given admittance $Y(s)$ of this lemma is realizable as in Fig.~4(d), then the resultant of $\alpha(s)$ and $\beta(s)$ in $s$ calculated as $R_0 (\alpha, \beta, s) = b_1b_2c_1^4k_1k_2(b_2k_1^2 + c_2^2k_2)^2 (b_1^2 k_1 + b_1^2 k_2 + c_1^2 b_2 - b_1 c_1 c_2)^2$ is zero. Therefore, one obtains that $c_2$ satisfies the expression in
(17), which further implies that the admittance of the configuration in Fig.~4(d) becomes
$Y(s) = \alpha'(s)/\beta'(s)$, where $\alpha'(s) = b_2 (b_1^2 k_1 + b_1^2 k_2 + b_1 c_1^2 + b_2 c_1^2) s^3 + b_1c_1(b_1k_1 + b_1 k_2 + b_2 k_1) s^2 + k_2 (b_1^2k_1 + b_1^2k_2 + b_2 c_1^2) s + b_1c_1k_1 k_2$ and $\beta'(s) = b_1b_2c_1 s^3 + (b_1^2 k_1 + b_1^2 k_2 + b_2 c_1^2) s^2 + b_1 c_1 (k_1 + k_2) s$. Then, it follows that
\begin{subequations}
\begin{align}
b_2 (b_1^2 k_1 + b_1^2 k_2 + b_1 c_1^2 + b_2 c_1^2)  &= k \alpha_3,  \label{eq: 3-2 eqn1}  \\
b_1c_1(b_1k_1 + b_1 k_2 + b_2 k_1) &= k \alpha_2,   \label{eq: 3-2 eqn2}  \\
k_2 (b_1^2k_1 + b_1^2k_2 + b_2 c_1^2) &= k \alpha_1,   \label{eq: 3-2 eqn3}  \\
b_1c_1k_1 k_2 &= k \alpha_0,     \label{eq: 3-2 eqn4}  \\
b_1b_2c_1 &= k \beta_3,  \label{eq: 3-2 eqn5}  \\
b_1^2 k_1 + b_1^2 k_2 + b_2 c_1^2 &= k \beta_2,   \label{eq: 3-2 eqn6}  \\
b_1 c_1 (k_1 + k_2)   &= k \beta_1,  \label{eq: 3-2 eqn7}
\end{align}
\end{subequations}
where $k > 0$.  Then, it follows from \eqref{eq: 3-2 eqn3} and \eqref{eq: 3-2 eqn6} that $k_2$ can be expressed as in (17). From \eqref{eq: 3-2 eqn4} and \eqref{eq: 3-2 eqn7}, one implies that
$k_1^{-1} + k_2^{-1} = \beta_1/\alpha_0$, which together with the expression of $k_2$ further implies the expression of $k_1$ as in (17). The assumption that the value of $k_1$ is positive and finite implies that $\tilde{\mathcal{B}}_{11} > 0$. Furthermore, it follows from \eqref{eq: 3-2 eqn4} and \eqref{eq: 3-2 eqn5} that $b_2 = \beta_3 k_1 k_2/\alpha_0$, which together with the expressions of $k_1$ and $k_2$ yields the expression of $b_2$ as in (17).  Similarly, from \eqref{eq: 3-2 eqn2} and \eqref{eq: 3-2 eqn5}, one obtains $b_1 k_1 + b_1 k_2 + b_2 k_1 = \alpha_2 b_2/\beta_3$.
Together with the expressions of $k_1$, $k_2$, and $b_2$, one implies that $b_1$ can be expressed as in (17), which together with the assumption that $b_1 > 0$ can imply that
(16b) holds.
Let
\begin{equation} \label{eq: 3-2 lambda c1}
x = c_1 \tilde{\mathcal{B}}_{11},
\end{equation}
which is positive since  $\tilde{\mathcal{B}}_{11} > 0$.
Then, it is implied from \eqref{eq: 3-2 lambda c1} that the value of $c_1$ can be expressed as in (17).
Substituting the expressions of $c_1$, $k_1$, $k_2$, $b_1$, and $b_2$ into \eqref{eq: 3-2 eqn1}
and \eqref{eq: 3-2 eqn6}
yields
(16c)  and $\beta_1 \beta_3 x^2 - \beta_2 (\alpha_2 \tilde{\mathcal{B}}_{11}
+ \alpha_1 \mathcal{B}_{13}) x +  (\alpha_2 \tilde{\mathcal{B}}_{11} + \alpha_1 \mathcal{B}_{13})^2= 0$. Then,
from the latter equation, it is implied that (16a) holds and $x$ satisfies
(16d).
From  \eqref{eq: 3-2 eqn5} and
the element values of $c_1$, $b_1$, and $b_2$ as in (17), the expression of $k$ can be obtained as
\begin{equation} \label{eq: 3-2 k}
k = \frac{\alpha_1^2 x(\alpha_2 \tilde{\mathcal{B}}_{11} + \alpha_1 \mathcal{B}_{13})}{\beta_1 \beta_2 \tilde{\mathcal{B}}_{11}^3}.
\end{equation}
Now, the necessity part is proved.

\emph{Sufficiency.} Let the element values of
$c_1$, $c_2$, $k_1$, $k_2$, $b_1$, and $b_2$
satisfy (17), where
(16a) and  (16b) hold, and
 $x$ is a positive root of
(16c), such that (16d) holds. Then, it can be verified that   the element values are positive and finite. Moreover, together with the expression of $c_2$ in (17), it can be calculated that the admittance of the configuration in Fig.~4(d) is $Y(s) = \alpha' (s) / \beta' (s)$, where $\alpha'(s) = b_2 (b_1^2 k_1 + b_1^2 k_2 + b_1 c_1^2 + b_2 c_1^2) s^3 + b_1c_1(b_1k_1 + b_1 k_2 + b_2 k_1) s^2 + k_2 (b_1^2k_1 + b_1^2k_2 + b_2 c_1^2) s + b_1c_1k_1 k_2$ and $\beta'(s) = b_1b_2c_1 s^3 + (b_1^2 k_1 + b_1^2 k_2 + b_2 c_1^2) s^2 + b_1 c_1 (k_1 + k_2) s$.
Since $x$ satisfies (16c) and (16d), it can be verified that \eqref{eq: 3-2 eqn1}--\eqref{eq: 3-2 eqn7} hold with $k$ satisfying \eqref{eq: 3-2 k}.
Therefore,   $Y(s)$ is realizable as the configuration in Fig.~4(d).

\section{Realizability Conditions of the Configurations in Fig.~5}
\label{appendix: Lemmas classes 4 and 5}

\subsection{The Configuration in Fig.~5(a) (The Proof of Lemma~14)}   \label{appendix: F}

\emph{Necessity.}
It can be proved that $c_2^{-1}  = 0$ and $c_3 = 0$ cannot simultaneously hold.
Assume that $c_2^{-1} \neq 0$ and $c_3 \geq 0$, which means that the value of $c_2$ is positive and finite. The admittance of the circuit configuration in Fig.~5(a) is computed as $Y(s) = n(s)/d(s)$, where $n(s) = b_1 c_1 (c_2 + c_3) s^3 + (b_1 c_1 k_2 + b_1 c_2 k_1 + b_1 c_3 k_1 + c_1 c_2 c_3) s^2 + (b_1 k_1 k_2 + c_1 c_2 k_2 + c_2 c_3 k_1) s + c_2 k_1 k_2$ and $d(s) = b_1 (c_1 + c_2 + c_3) s^3 + (b_1 k_1 + b_1 k_2 + c_1 c_2 + c_2 c_3) s^2 + c_2 (k_1 + k_2) s$. Since the admittance $Y(s)$ in this lemma can be realized by the circuit in Fig.~5(a), there exists $k > 0$ such that $n(s) = k \alpha(s)$ and $d(s) = k \beta(s)$.
Then, it follows that
\begin{subequations}
\begin{align}
b_1 c_1 (c_2 + c_3) &= k \alpha_3,  \label{eq: 4-1 eqn1}  \\
b_1 c_1 k_2 + b_1 c_2 k_1 + b_1 c_3 k_1 + c_1 c_2 c_3  &= k \alpha_2   \label{eq: 4-1 eqn2}  \\
b_1 k_1 k_2 + c_1 c_2 k_2 + c_2 c_3 k_1  &= k \alpha_1,   \label{eq: 4-1 eqn3}  \\
c_2 k_1 k_2 &= k \alpha_0,     \label{eq: 4-1 eqn4}  \\
b_1 (c_1 + c_2 + c_3)  &= k \beta_3,  \label{eq: 4-1 eqn5}  \\
b_1 k_1 + b_1 k_2 + c_1 c_2 + c_2 c_3 &= k \beta_2,   \label{eq: 4-1 eqn6}  \\
c_2 (k_1 + k_2)   &= k \beta_1,  \label{eq: 4-1 eqn7}
\end{align}
\end{subequations}
where $k > 0$.
Then, it follows from \eqref{eq: 4-1 eqn4} and \eqref{eq: 4-1 eqn7} that $1/k_1 + 1/k_2 = \beta_1/\alpha_0$, which implies that $k_2$ can be expressed as in
\begin{equation} \label{eq: 4-1 k2 k1}
k_2 = \frac{\alpha_0 k_1}{\beta_1 k_1 - \alpha_0},
\end{equation}
which implies that  $\beta_1 k_1 - \alpha_0 > 0$. Substituting \eqref{eq: 4-1 k2 k1} into \eqref{eq: 4-1 eqn4} implies that $c_2$ can be expressed as in
\begin{equation} \label{eq: 4-1 c2 k1}
c_2 = \frac{k(\beta_1 k_1 - \alpha_0)}{k_1^2}.
\end{equation}
Combining \eqref{eq: 4-1 eqn1} and \eqref{eq: 4-1 eqn5}, one obtains $1/c_1 + 1/(c_2 + c_3) = \beta_3/\alpha_3$, which together with \eqref{eq: 4-1 c2 k1} implies that
\begin{equation} \label{eq: 4-1 c3 k1}
c_3 = -\frac{(\beta_3 c_1 - \alpha_3)(\beta_1 k_1 - \alpha_0)k - c_1 k_1^2 \alpha_3}{k_1^2(\beta_3 c_1 - \alpha_3)}.
\end{equation}
Then, substituting \eqref{eq: 4-1 c2 k1} and \eqref{eq: 4-1 c3 k1} into \eqref{eq: 4-1 eqn1} implies that
\begin{equation}  \label{eq: 4-1 b1 c1}
b_1 = \frac{k(\beta_3 c_1 - \alpha_3)}{c_1^2}.
\end{equation}
The assumption that    $b_1 > 0$ and $c_3 \geq  0$ implies that  $\beta_3 c_1 - \alpha_3 > 0$ and $(\beta_3c_1 - \alpha_3)(\beta_1k_1 - \alpha_0) k - c_1 k_1^2\alpha_3 \leq 0$. Furthermore, let
\begin{equation} \label{eq: 4-1 mu}
y = - \frac{k_1}{c_1} < 0.
\end{equation}
Together with \eqref{eq: 4-1 k2 k1}--\eqref{eq: 4-1 mu}, it follows from \eqref{eq: 4-1 eqn2} and \eqref{eq: 4-1 eqn3} that
(18) holds. Similarly, together with (18) and \eqref{eq: 4-1 k2 k1}--\eqref{eq: 4-1 mu}, it follows from
\eqref{eq: 4-1 eqn2} and \eqref{eq: 4-1 eqn6} that $k_1$ can be expressed as in
(20). Together with \eqref{eq: 4-1 k2 k1}--\eqref{eq: 4-1 mu}, one can obtain from \eqref{eq: 4-1 eqn6} that $k$ satisfies
\begin{equation} \label{eq: 4-1 x}
k = \frac{k_1^4 (\beta_1 y^2 (\alpha_3 y + \beta_3 k_1)^2 + \beta_3 (\beta_1 k_1 - \alpha_0)^2 + \beta_2 y (\alpha_3 y + \beta_3 k_1) (\beta_1 k_1 - \alpha_0))}{-y (\beta_1 k_1 - \alpha_0)^3 (\alpha_3 y + \beta_3 k_1)}.
\end{equation}
Since $\beta_1 k_1 - \alpha_0 > 0$, $y < 0$, and $k > 0$, it is clear that $\beta_1 y^2 (\alpha_3 y + \beta_3 k_1)^2 + \beta_3 (\beta_1 k_1 - \alpha_0)^2 + \beta_2 y (\alpha_3 y + \beta_3 k_1) (\beta_1 k_1 - \alpha_0)$ and $\alpha_3 y + \beta_3 k_1$ are nonzero and have the same sign. Since it can be verified that $\alpha_3 y + \beta_3 k_1 < 0$ implies $\beta_1 y^2 (\alpha_3 y + \beta_3 k_1)^2 + \beta_3 (\beta_1 k_1 - \alpha_0)^2 + \beta_2 y (\alpha_3 y + \beta_3 k_1) (\beta_1 k_1 - \alpha_0) > 0$, one can indicate that $\alpha_3 y + \beta_3 k_1 > 0$ and $\beta_1 y^2 (\alpha_3 y + \beta_3 k_1)^2 + \beta_3 (\beta_1 k_1 - \alpha_0)^2 + \beta_2 y (\alpha_3 y + \beta_3 k_1) (\beta_1 k_1 - \alpha_0) > 0$, together with $\beta_1 k_1 - \alpha_0 > 0$,  $\beta_3 c_1 - \alpha_3 > 0$, $(\beta_3c_1 - \alpha_3)(\beta_1k_1 - \alpha_0) k - c_1 k_1^2\alpha_3 \leq 0$,
and
the expression of $k_1$ in (20) imply that $\Gamma_1$, $\Gamma_2$, $\Gamma_3$, $\Gamma_4$,   $\Gamma_5$, and $\Gamma_6$  have the same sign, where $\Gamma_k$ is nonzero for $k = 1, 2, ..., 5$ and $\Gamma_6$ can be zero.
Furthermore, by \eqref{eq: 4-1 x} and the expression of $k_1$ in (20), the element values in \eqref{eq: 4-1 k2 k1}--\eqref{eq: 4-1 mu} can be further expressed as in (20).

The realization condition for the case when $c_2^{-1} = 0$ and $c_3 > 0$ can be similarly derived as Condition~2, and the element value expressions are as in (21).

\emph{Sufficiency.}
Assume that Condition~1  holds. Let  $c_1$, $c_2$, $c_3$, $k_1$, $k_2$, and $b_1$
satisfy
(20), where $y$ is a negative  root of (18), and $\Gamma_1$, $\Gamma_2$, $\Gamma_3$, $\Gamma_4$,   $\Gamma_5$, and $\Gamma_6$  have the same sign, where $\Gamma_k$ is nonzero for $k = 1, 2, ..., 5$ and $\Gamma_6$ can be zero. Then, it can be verified that  $c_3 \geq 0$ and other element values are positive and finite. Since
(18) holds, it can be verified that
\eqref{eq: 4-1 eqn1}--\eqref{eq: 4-1 eqn7} hold with $k$ satisfying \eqref{eq: 4-1 x}. Therefore, the given admittance $Y(s)$  is realizable as the circuit in Fig.~5(a). Similarly, the case when Condition~2 holds can be also proved.

\subsection{The Configuration in Fig. 5(b) (Proof of Lemma~15)}   \label{appendix: G}

\emph{Necessity.}
The admittance of the circuit configuration in Fig.~5(b) is computed as $Y(s) = n(s)/d(s)$, where
$n(s) = b_1c_1c_2 s^4 + b_1 (c_1 k_2 + c_1 k_3 + c_2 k_1) s^3 + (b_1 k_1 k_2 + b_1 k_1 k_3 + c_1 c_2 k_3) s^2 + k_3 (c_1 k_2 + c_2 k_1) s + k_1 k_2 k_3$ and $d(s) = b_1 (c_1 + c_2) s^4 + b_1 (k_1 + k_2 + k_3) s^3 + k_3 (c_1 + c_2) s^2 + k_3 (k_1 + k_2) s$.
If the given admittance $Y(s)$ of this lemma is realizable by the circuit in Fig.~5(b), then the resultant of $n(s)$ and $d(s)$ in $s$ calculated as $R_0 (n, d, s) = - b_1^3 k_1 k_2 k_3^5 (k_1 (b_1k_1^2 + c_1^2k_3) c_2 - c_1 (b_1k_1^2k_2 + b_1k_1^2k_3 + c_1^2k_2k_3))^2$ is zero. Therefore, one obtains that $c_2$ satisfies the expression in
(23), which further implies that the admittance of the circuit is equivalent to
$Y(s) = \tilde{n}(s)/\tilde{d}(s)$, where $\tilde{n}(s) = b_1 c_1 (b_1 k_1^2 k_2 + b_1 k_1^2 k_3 + c_1^2 k_2 k_3) s^3 + b_1 k_1 (k_2 + k_3)  (b_1 k_1^2 + c_1^2 k_3)s^2 + c_1 k_3 (b_1 k_1^2 k_2 + b_1 k_1^2 k_3 + c_1^2 k_2 k_3) s + k_1 k_2 k_3 (b_1 k_1^2 + c_1^2 k_3)$ and $\tilde{d}(s) = b_1 (b_1 k_1^3 + b_1 k_1^2 k_2 + b_1 k_1^2 k_3 + c_1^2 k_1 k_3 + c_1^2 k_2 k_3) s^3 + b_1 c_1 k_1 k_3^2 s^2 + k_3 (k_1 + k_2) (b_1 k_1^2 + c_1^2 k_3)s$. Therefore, there exists $k > 0$ such that $\tilde{n}(s) = k \alpha(s)$ and $\tilde{d}(s) = k \beta(s)$. Then,
it follows that
\begin{subequations}
\begin{align}
b_1 c_1 (b_1 k_1^2 k_2 + b_1 k_1^2 k_3 + c_1^2 k_2 k_3) &= k \alpha_3,  \label{eq: 4-2 eqn1}  \\
b_1 k_1 (k_2 + k_3)  (b_1 k_1^2 + c_1^2 k_3)  &= k \alpha_2   \label{eq: 4-2 eqn2}  \\
c_1 k_3 (b_1 k_1^2 k_2 + b_1 k_1^2 k_3 + c_1^2 k_2 k_3)  &= k \alpha_1,   \label{eq: 4-2 eqn3}  \\
k_1 k_2 k_3 (b_1 k_1^2 + c_1^2 k_3) &= k \alpha_0,     \label{eq: 4-2 eqn4}  \\
b_1 (b_1 k_1^3 + b_1 k_1^2 k_2 + b_1 k_1^2 k_3 + c_1^2 k_1 k_3 + c_1^2 k_2 k_3)  &= k \beta_3,  \label{eq: 4-2 eqn5}  \\
b_1 c_1 k_1 k_3^2 &= k \beta_2,   \label{eq: 4-2 eqn6}  \\
k_3 (k_1 + k_2) (b_1 k_1^2 + c_1^2 k_3)   &= k \beta_1,  \label{eq: 4-2 eqn7}
\end{align}
\end{subequations}
where $k > 0$. Then, it follows from \eqref{eq: 4-2 eqn4} and \eqref{eq: 4-2 eqn7} that $1/k_1 + 1/k_2 = \beta_1/\alpha_0$, which implies that $k_2$ can be expressed as
\begin{equation}  \label{eq: 4-2 element values k2 k1}
k_2 = \frac{\alpha_0 k_1}{\beta_1 k_1 - \alpha_0},
\end{equation}
which implies that $\beta_1 k_1 - \alpha_0 > 0$.
Similarly, combining \eqref{eq: 4-2 eqn1} and \eqref{eq: 4-2 eqn3}, one obtains
\begin{equation}  \label{eq: 4-2 element values b1 k3}
b_1 = \frac{\alpha_3}{\alpha_1}k_3.
\end{equation}
Then, it is implied from \eqref{eq: 4-2 eqn2} and \eqref{eq: 4-2 eqn4} that
\begin{equation}  \label{eq: 4-2 element values k3 k1}
k_3 = \frac{\Delta_\alpha k_1}{\alpha_3 (\beta_1 k_1 - \alpha_0)}.
\end{equation}
Then, it follows from \eqref{eq: 4-2 eqn1} and \eqref{eq: 4-2 eqn6} that
\begin{equation}  \label{eq: 4-2 element values c1 2}
c_1^2 = \frac{-\alpha_2\beta_2k_1^2 + \Delta_\alpha k_1}{\alpha_0 \beta_2},
\end{equation}
which together with \eqref{eq: 4-2 eqn5} and \eqref{eq: 4-2 eqn6} implies that $c_1$ can be expressed as in
\begin{equation}  \label{eq: 4-2 element values c1 k1}
c_1 = \frac{ k_1 (-\beta_2(\beta_1 k_1 - \alpha_0) + \alpha_1\beta_1)}{\alpha_0 \alpha_1}.
\end{equation}
Substituting \eqref{eq: 4-2 element values k2 k1}--\eqref{eq: 4-2 element values k3 k1} and \eqref{eq: 4-2 element values c1 k1} into \eqref{eq: 4-2 eqn6}
implies that
\begin{equation}  \label{eq: 4-2 x}
k = \frac{\Delta_\alpha^3 k_1^5  (-\beta_2(\beta_1 k_1 - \alpha_0) + \alpha_1\beta_1)}{\alpha_0  \alpha_1^2 \alpha_3  \beta_2 \beta_3 (\beta_1 k_1 - \alpha_0)^3}.
\end{equation}
By the element value expressions in \eqref{eq: 4-2 element values k2 k1}--\eqref{eq: 4-2 element values k3 k1} and \eqref{eq: 4-2 element values c1 k1} and the value of $k$ in \eqref{eq: 4-2 x}, one can obtain
from \eqref{eq: 4-2 eqn4} and \eqref{eq: 4-2 element values c1 2} that
\begin{equation}  \label{eq: 4-2 equivalent equation 01}
\alpha_3^2 \beta_1^2 \beta_2^3 k_1^3 - 2 \alpha_3^2 \beta_1 \beta_2^2 (\alpha_0\beta_2 + \alpha_1\beta_1) k_1^2 + \beta_2 (\alpha_3^2 (\alpha_0\beta_2 + \alpha_1\beta_1)^2 + \alpha_0\alpha_1^2\alpha_2\beta_3^2) k_1 - \alpha_0 \alpha_1^2 \beta_3^2 \Delta_\alpha = 0,
\end{equation}
and
\begin{equation}  \label{eq: 4-2 equivalent equation 02}
\beta_2 \mathcal{B}_{23} k_1 - (\alpha_1 \mathcal{B}_{23} - \alpha_3 \mathcal{B}_{12}) = 0.
\end{equation}
By \eqref{eq: 4-2 equivalent equation 02}, it is clear that $\mathcal{B}_{23} \neq 0$, which further implies that
$k_1$ can be expressed as in (23). Substituting the expression of $k_1$ into
\eqref{eq: 4-2 equivalent equation 01} implies that (22c) holds.
Then, by   \eqref{eq: 4-2 element values k2 k1}--\eqref{eq: 4-2 element values k3 k1},   \eqref{eq: 4-2 element values c1 k1},
and the expression of $k_1$ as in (23),
the element values of $c_1$, $k_2$, $k_3$, and $b_1$ can be further expressed as in
(23).
The assumption that the element values are positive and finite implies that
(22a) and (22b) hold. The proof of the necessity part has been completed.

\emph{Sufficiency.}
Let  $c_1$, $c_2$, $k_1$, $k_2$, $k_3$, and $b_1$
satisfy (23), where (22a)--(22c) hold.
Then, it can be verified that the element values are positive and finite.  Moreover, together with the expression of $c_2$ in
(23), it can be computed that the admittance of the configuration in Fig.~5(b) is $Y(s) = \tilde{n}(s) / \tilde{d}(s)$, where $\tilde{n}(s) = b_1 c_1 (b_1 k_1^2 k_2 + b_1 k_1^2 k_3 + c_1^2 k_2 k_3) s^3 + b_1 k_1 (k_2 + k_3)  (b_1 k_1^2 + c_1^2 k_3)s^2 + c_1 k_3 (b_1 k_1^2 k_2 + b_1 k_1^2 k_3 + c_1^2 k_2 k_3) s + k_1 k_2 k_3 (b_1 k_1^2 + c_1^2 k_3)$ and $\tilde{d}(s) = b_1 (b_1 k_1^3 + b_1 k_1^2 k_2 + b_1 k_1^2 k_3 + c_1^2 k_1 k_3 + c_1^2 k_2 k_3) s^3 + b_1 c_1 k_1 k_3^2 s^2 + k_3 (k_1 + k_2) (b_1 k_1^2 + c_1^2 k_3)s$. Since
(23) holds,
one can verify that \eqref{eq: 4-2 eqn1}--\eqref{eq: 4-2 eqn7} hold with $k$ satisfying \eqref{eq: 4-2 x}.  Therefore,  $Y(s)$  is realizable by the circuit configuration in Fig.~5(b).

\subsection{The Configuration in Fig.~5(c) (The Proof of Lemma~16)}

\emph{Necessity.} The admittance of the circuit configuration in Fig.~5(c) is calculated as $Y(s) = n(s)/d(s)$, where
$n(s) = b_1 c_1 c_2 s^4 + b_1 (c_1 k_3 + c_2 k_1) s^3 + (b_1 k_1 k_3 + c_1 c_2 k_2 + c_1 c_2 k_3) s^2 + (c_1 k_2 k_3 + c_2 k_1 k_2 + c_2 k_1 k_3) s + k_1 k_2 k_3$ and $d(s) = b_1 c_2 s^4 + (b_1 k_3 + c_1 c_2) s^3 + (c_1 k_3 + c_2 k_1 + c_2 k_2 + c_2 k_3) s^2 + k_3 (k_1 + k_2) s$.
If the given admittance $Y(s)$ of this lemma is realizable by the circuit in Fig.~5(c), then the resultant of $n(s)$ and $d(s)$ in $s$ calculated as $R_0 (n, d, s) = -b_1 c_2^4 k_1 k_2 k_3^3 (c_1 k_3 (b_1 k_1^2 + c_1^2k_2) - c_2 k_1 (b_1 k_1^2 + c_1^2 k_2 + c_1^2 k_3) )^2$ is zero. Therefore, one obtains that $c_2$ satisfies the expression in
(25), which further implies that the admittance of the configuration in Fig.~5(c) becomes
$Y(s) = \tilde{n}(s)/\tilde{d}(s)$, where $\tilde{n}(s) = b_1 c_1 (b_1 k_1^2 + c_1^2 k_2) s^3 + b_1 k_1 (b_1 k_1^2 + c_1^2 k_2 + c_1^2 k_3) s^2 + c_1 (k_2 + k_3) (b_1 k_1^2 + c_1^2 k_2) s + k_1 k_2 (b_1 k_1^2 + c_1^2 k_2 + c_1^2 k_3)$ and $\tilde{d}(s) = b_1 (b_1 k_1^2 + c_1^2 k_2) s^3 + c_1 (b_1 k_1^2 + c_1^2 k_2 + b_1 k_1 k_3) s^2 + (k_1 + k_2) (b_1 k_1^2 + c_1^2 k_2 + c_1^2 k_3) s$.
Therefore, there exists $k > 0$ such that $n(s) = k\alpha(s)$ and $d(s) = k \beta(s)$. Then, it follows that
\begin{subequations}
\begin{align}
b_1 c_1 (b_1 k_1^2 + c_1^2 k_2) &= k \alpha_3,  \label{eq: 4-3 eqn1}  \\
b_1 k_1 (b_1 k_1^2 + c_1^2 k_2 + c_1^2 k_3)  &= k \alpha_2   \label{eq: 4-3 eqn2}  \\
c_1 (k_2 + k_3) (b_1 k_1^2 + c_1^2 k_2)  &= k \alpha_1,   \label{eq: 4-3 eqn3}  \\
k_1 k_2 (b_1 k_1^2 + c_1^2 k_2 + c_1^2 k_3) &= k \alpha_0,     \label{eq: 4-3 eqn4}  \\
b_1 (b_1 k_1^2 + c_1^2 k_2)  &= k \beta_3,  \label{eq: 4-3 eqn5}  \\
c_1 (b_1 k_1^2 + c_1^2 k_2 + b_1 k_1 k_3) &= k \beta_2,   \label{eq: 4-3 eqn6}  \\
 (k_1 + k_2) (b_1 k_1^2 + c_1^2 k_2 + c_1^2 k_3)   &= k \beta_1,  \label{eq: 4-3 eqn7}
\end{align}
\end{subequations}
where $k > 0$. Then, it follows from \eqref{eq: 4-3 eqn1} and \eqref{eq: 4-3 eqn5} that the value of $c_1$ can be expressed as in (25). Let
\begin{equation} \label{eq: 4-3 lambda k1}
x = k_1 > 0.
\end{equation}
Clearly, it follows from \eqref{eq: 4-3 lambda k1} that $k_1$ can be expressed as in (25).
Furthermore, combining \eqref{eq: 4-3 eqn4} and \eqref{eq: 4-3 eqn7}, it is implied that $k_1^{-1} + k_2^{-1} = \beta_1/\alpha_0$, which together with the expression of $k_1$
implies the expression of $k_2$ as in (25). From \eqref{eq: 4-3 eqn2} and \eqref{eq: 4-3 eqn4}, one obtains $b_1 = \alpha_2 k_2/\alpha_0$, which together with the expression of $k_2$ further implies that the value of $b_1$ can be expressed as in (25).
It follows from \eqref{eq: 4-3 eqn3} and \eqref{eq: 4-3 eqn5} that $c_1 (k_2 + k_3) = \alpha_1 b_1 / \beta_3$, which together with the expressions of $c_1$ and $k_2$ implies that $k_3$ can be expressed as in (25). Substituting  the expressions of $c_1$, $k_1$, $k_2$, and $b_1$ into \eqref{eq: 4-3 eqn5} yields
\begin{equation} \label{eq: 4-3 x}
k = \frac{\alpha_2 x^2 (\alpha_2 \beta_3^2 x^2 + \alpha_0 \alpha_3^2)}{\beta_3^3 (\beta_1 x - \alpha_0)^2}.
\end{equation}
Substituting the element values in (25) into \eqref{eq: 4-3 eqn2} and
\eqref{eq: 4-3 eqn6} implies (24d) and $\alpha_2 \beta_3^3 (\alpha_3 \beta_1 - \alpha_2 \beta_2) x^3 + \alpha_2 \beta_3^2 (\alpha_1 \alpha_2 - 2 \alpha_0 \alpha_3) x^2  + \alpha_0 \alpha_3^2 (\alpha_3 \beta_1 - \alpha_2 \beta_2) x - \alpha_0^2 \alpha_3^3 = 0$, which can further imply
(24b). The assumption that the element values are positive and finite implies that
(24a) and (24c) hold. The proof of the necessity part has been completed.

\emph{Sufficiency.}
Let the element values of $c_1$, $c_2$, $k_1$, $k_2$, $k_3$, and $b_1$ satisfy (25), where
(24a) holds,
and $x$ is positive root of (24b), such that (24c) and (24d) hold.
Then, it can be verified that the elements values are positive and finite. Moreover, together with the expression of $c_2$ in
(25), it can be calculated that the admittance of the configuration in Fig.~5(c) is
$Y(s) = \tilde{n}(s) / \tilde{d}(s)$, where $\tilde{n}(s) = b_1 c_1 (b_1 k_1^2 + c_1^2 k_2) s^3 + b_1 k_1 (b_1 k_1^2 + c_1^2 k_2 + c_1^2 k_3) s^2 + c_1 (k_2 + k_3) (b_1 k_1^2 + c_1^2 k_2) s + k_1 k_2 (b_1 k_1^2 + c_1^2 k_2 + c_1^2 k_3)$ and $\tilde{d}(s) = b_1 (b_1 k_1^2 + c_1^2 k_2) s^3 + c_1 (b_1 k_1^2 + c_1^2 k_2 + b_1 k_1 k_3) s^2 + (k_1 + k_2) (b_1 k_1^2 + c_1^2 k_2 + c_1^2 k_3) s$. Since $x$ satisfies (24b) and (24d), it can be verified that
\eqref{eq: 4-3 eqn1}--\eqref{eq: 4-3 eqn7} hold with $k$ satisfying \eqref{eq: 4-3 x}. Therefore,  $Y(s)$  is realizable by the circuit in Fig.~5(c).

\subsection{The Configuration in Fig.~5(d) (The Proof of Lemma~17)}

\emph{Necessity.} The admittance of the circuit configuration in Fig.~5(d) is calculated as $Y(s) = n(s)/d(s)$, where
$n(s) = b_1 b_2 c_1 s^4 + (b_1 b_2 k_1 + b_1 c_1 c_2 + b_2 c_1 c_2) s^3 + (b_1 c_2 k_1 + b_2 c_1 k_2 + b_2 c_2 k_1) s^2 + k_2 (b_2 k_1 + c_1 c_2) s + c_2 k_1 k_2$ and $d(s) = b_1 b_2 s^4 + (b_1 c_2 + b_2 c_1 + b_2c_2) s^3 + (b_2 k_1 + b_2 k_2 + c_1 c_2) s^2 + c_2 (k_1 + k_2) s$.
If the given admittance $Y(s)$ of this lemma is realizable by the circuit in Fig.~5(d), then the resultant of $n(s)$ and $d(s)$ in $s$ calculated as $R_0 (n, d, s) = b_1 b_2^3 c_2^4 k_1 k_2 (b_2 k_1 (b_1 k_1^2 + c_1^2k_2) - c_1c_2 (b_1k_1^2 + b_2 k_1^2 + c_1^2 k_2) )^2$ is zero. Therefore, one obtains that $c_2$ satisfies the expression in
(27), which further implies that the admittance of the configuration in Fig.~5(d) becomes
$Y(s) = \tilde{n}(s)/\tilde{d}(s)$, where $\tilde{n}(s) = b_1 c_1 (b_1 k_1^2 + b_2 k_1^2 + c_1^2 k_2) s^3 + k_1 (b_1 + b_2) (b_1 k_1^2 + c_1^2 k_2) s^2 + c_1 k_2 (b_1 k_1^2 + b_2 k_1^2 + c_1^2 k_2) s + k_1 k_2 (b_1 k_1^2 + c_1^2 k_2)$ and $\tilde{d}(s) = b_1 (b_1 k_1^2 + b_2 k_1^2 + c_1^2 k_2) s^3 + c_1 (b_1 k_1^2 + b_2 k_1^2 + b_2 k_1 k_2 + c_1^2 k_2) s^2 + (k_1 + k_2) (b_1 k_1^2 + c_1^2 k_2) s$. Therefore, there exists $k > 0$ such that $n(s) = k \alpha(s)$ and $d(s) = k \beta(s)$.  Then, it follows that
\begin{subequations}
\begin{align}
b_1 c_1 (b_1 k_1^2 + b_2 k_1^2 + c_1^2 k_2) &= k \alpha_3,  \label{eq: 4-4 eqn1}  \\
k_1 (b_1 + b_2) (b_1 k_1^2 + c_1^2 k_2)   &= k \alpha_2   \label{eq: 4-4 eqn2}  \\
c_1 k_2 (b_1 k_1^2 + b_2 k_1^2 + c_1^2 k_2)  &= k \alpha_1,   \label{eq: 4-4 eqn3}  \\
k_1 k_2 (b_1 k_1^2 + c_1^2 k_2) &= k \alpha_0,     \label{eq: 4-4 eqn4}  \\
b_1 (b_1 k_1^2 + b_2 k_1^2 + c_1^2 k_2)  &= k \beta_3,  \label{eq: 4-4 eqn5}  \\
c_1 (b_1 k_1^2 + b_2 k_1^2 + b_2 k_1 k_2 + c_1^2 k_2) &= k \beta_2,   \label{eq: 4-4 eqn6}  \\
(k_1 + k_2) (b_1 k_1^2 + c_1^2 k_2)   &= k \beta_1,  \label{eq: 4-4 eqn7}
\end{align}
\end{subequations}
where $k > 0$. Then, it follows from \eqref{eq: 4-4 eqn1} and \eqref{eq: 4-4 eqn5} that the value of $c_1$ can be expressed as in (27).
Let
\begin{equation} \label{eq: 4-4 lambda k1}
x = k_1 > 0.
\end{equation}
Clearly, it can be derived from \eqref{eq: 4-4 lambda k1} that $k_1$ can be expressed as in
(27).
From \eqref{eq: 4-4 eqn4} and \eqref{eq: 4-4 eqn7}, one can derive that $k_1^{-1} + k_2^{-1} = \beta_1/\alpha_0$, which together with the expression of $k_1$ further implies the expression of $k_2$ as in
(27). Then, it follows from \eqref{eq: 4-4 eqn3} and \eqref{eq: 4-4 eqn5} that
$b_1 = \beta_3 c_1 k_2/\alpha_1$, which together with the expressions of $c_1$ and $k_2$ can imply the expression of $b_1$ as in (27). Furthermore, one can derive from \eqref{eq: 4-4 eqn2} and \eqref{eq: 4-4 eqn4} that $b_1 + b_2 = \alpha_2 k_2 / \alpha_0$, which together with the expressions of $k_2$ and $b_1$ can further imply that  the value of $b_2$ can be expressed as in (27).
By the element value expressions in (27), it follows from  \eqref{eq: 4-4 eqn4} that
\begin{equation}  \label{eq: 4-4 k}
k = \frac{\alpha_0 \alpha_3 x^3 (\beta_3^2 x^2 + \alpha_1 \alpha_3)}{\alpha_1
\beta_3^2 (\beta_1 x - \alpha_0)^2}.
\end{equation}
Substituting \eqref{eq: 4-4 k} and the element values in (27) into \eqref{eq: 4-4 eqn5} and \eqref{eq: 4-4 eqn6} implies (26d) and $\alpha_0 \beta_2 \beta_3^3 x^4 - \alpha_1 \alpha_2 \beta_1 \beta_3^2 x^3 + \alpha_0 \alpha_3 \beta_3 (\alpha_0 \beta_3 + \alpha_1 \beta_2) x^2 - \alpha_0 \alpha_1 \alpha_3^2 \beta_1 x + \alpha_0^2 \alpha_1 \alpha_3^2 = 0$,
 respectively, which can further imply
(26b). The assumption that the element values are positive and finite can imply that
(26a) and (26c). The proof of the necessity part has been completed.

\emph{Sufficiency.}
Let the element values of $c_1$, $c_2$, $k_1$, $k_2$, $b_1$, and
$b_2$
 satisfy (27), where (26a) holds, and $x$ is a positive root of (26b), such that
(26c) and (26d) hold.
Then, it can be verified that the elements values are positive and finite. Moreover, together with the expression of $c_2$ in
(26c), it can be calculated that the admittance of the configuration in Fig.~5(d) is
$Y(s) = \tilde{n}(s)/\tilde{d}(s)$, where $\tilde{n}(s) = b_1 c_1 (b_1 k_1^2 + b_2 k_1^2 + c_1^2 k_2) s^3 + k_1 (b_1 + b_2) (b_1 k_1^2 + c_1^2 k_2) s^2 + c_1 k_2 (b_1 k_1^2 + b_2 k_1^2 + c_1^2 k_2) s + k_1 k_2 (b_1 k_1^2 + c_1^2 k_2)$ and $\tilde{d}(s) = b_1 (b_1 k_1^2 + b_2 k_1^2 + c_1^2 k_2) s^3 + c_1 (b_1 k_1^2 + b_2 k_1^2 + b_2 k_1 k_2 + c_1^2 k_2) s^2 + (k_1 + k_2) (b_1 k_1^2 + c_1^2 k_2) s$.
Since $x$ satisfies (26b) and (26d), it can be verified that \eqref{eq: 4-4 eqn1}--\eqref{eq: 4-4 eqn7} hold with
$k$ satisfying \eqref{eq: 4-4 k}. Therefore,  $Y(s)$  is realizable by the circuit in Fig.~5(d).

\subsection{The Configuration in Fig.~5(e) (The Proof of Lemma~18)}

\emph{Necessity.} The admittance of the circuit configuration in Fig.~5(e) is calculated as $Y(s) = n(s)/d(s)$, where
$n(s) = b_1 c_1 c_2 s^4 + b_1(c_1 k_2 + c_2 k_1) s^3 + (b_1 k_1 k_2 + c_1 c_2 k_2 + c_1 c_2 k_3) s^2 + (c_1 k_2 k_3 + c_2 k_1 k_2 + c_2 k_1 k_3) s + k_1 k_2 k_3$ and $d(s) = b_1 (c_1 + c_2) s^4 + (b_1 k_1+b_1k_2+c_1 c_2) s^3 + (c_1 k_3 + c_2 k_1 + c_2 k_2 + c_2 k_3) s^2 + k_3 (k_1 + k_2) s$.
If the given admittance $Y(s)$ of this lemma is realizable as in Fig.~5(e), then the resultant of $n(s)$ and $d(s)$ in $s$ calculated as $R_0 (n, d, s) = -b_1 c_2^4 k_1 k_2 k_3^3 (c_2 k_1 (b_1 k_1^2 + c_1^2 k_2 + c_1^2 k_3) - c_1 k_2(b_1 k_1^2 + c_1^2 k_3))^2$ is zero. Therefore, one obtains that $c_2$ satisfies the expression in
(29), which further implies that the admittance of the configuration in Fig.~5(e) becomes
$Y(s) = \tilde{n}(s)/\tilde{d}(s)$, where $\tilde{n}(s) = b_1 c_1 k_2 (b_1 k_1^2 + c_1^2 k_3) s^3 + b_1 k_1 k_2 (b_1 k_1^2 + c_1^2 k_2 + c_1^2 k_3) s^2 + c_1 k_2 (k_2 + k_3) (b_1 k_1^2 + c_1^2 k_3) s + k_1 k_2 k_3  (b_1 k_1^2 + c_1^2 k_2 + c_1^2 k_3)$ and $\tilde{d}(s) = b_1 (b_1 k_1^3 + b_1 k_1^2 k_2 + c_1^2 k_1 k_2 + c_1^2 k_1 k_3 + c_1^2 k_2 k_3) s^3 + c_1 k_2 (b_1 k_1^2 + b_1 k_1 k_2 + c_1^2 k_3) s^2 + k_3 (k_1 + k_2) (b_1 k_1^2 + c_1^2 k_2 + c_1^2 k_3)s$.
Therefore, there exists $k > 0$ such that $n(s) = k \alpha(s)$ and $d(s) = k \beta(s)$.
Then, it follows that
\begin{subequations}
\begin{align}
b_1 c_1 k_2 (b_1 k_1^2 + c_1^2 k_3) &= k \alpha_3, \label{eq: 4-5 eqn1}    \\
b_1 k_1 k_2 (b_1 k_1^2 + c_1^2 k_2 + c_1^2 k_3) &= k \alpha_2,    \label{eq: 4-5 eqn2}    \\
c_1 k_2 (k_2 + k_3) (b_1 k_1^2 + c_1^2 k_3)  &=  k \alpha_1,  \label{eq: 4-5 eqn3}    \\
k_1 k_2 k_3  (b_1 k_1^2 + c_1^2 k_2 + c_1^2 k_3)  &=  k \alpha_0,  \label{eq: 4-5 eqn4}    \\
b_1 (b_1 k_1^3 + b_1 k_1^2 k_2 + c_1^2 k_1 k_2 + c_1^2 k_1 k_3 + c_1^2 k_2 k_3)  &=  k \beta_3,  \label{eq: 4-5 eqn5}    \\
c_1 k_2 (b_1 k_1^2 + b_1 k_1 k_2 + c_1^2 k_3)  &=  k \beta_2,  \label{eq: 4-5 eqn6}    \\
k_3 (k_1 + k_2) (b_1 k_1^2 + c_1^2 k_2 + c_1^2 k_3)  &=  k \beta_1,  \label{eq: 4-5 eqn7}
\end{align}
\end{subequations}
where $k > 0$. Then, it follows from \eqref{eq: 4-5 eqn4} and \eqref{eq: 4-5 eqn7} that $k_2$ can be expressed as in
\begin{equation} \label{eq: 4-5 k2 equivalent}
k_2 = \frac{\alpha_0 k_1}{\beta_1 k_1 - \alpha_0}.
\end{equation}
Similarly, it is derived from \eqref{eq: 4-5 eqn2} and \eqref{eq: 4-5 eqn4} that
\begin{equation}  \label{eq: 4-5 b1 equivalent}
b_1 = \frac{\alpha_2 k_3}{\alpha_0}.
\end{equation}
Together with \eqref{eq: 4-5 k2 equivalent} and \eqref{eq: 4-5 b1 equivalent}, one can obtain from \eqref{eq: 4-5 eqn1} and \eqref{eq: 4-5 eqn3} that
\begin{equation}   \label{eq: 4-5 k3 equivalent}
k_3 = \frac{\alpha_0^2 \alpha_3 k_1}{\Delta_\alpha (\beta_1 k_1 - \alpha_0)}.
\end{equation}
Then, it follows from \eqref{eq: 4-5 eqn1} and \eqref{eq: 4-5 eqn6} that $k \beta_2 b_1 - k \alpha_3 = b_1 c_1 k_1 k_2^2$, which further implies that $c_1$ can be expressed as in
\begin{equation}  \label{eq: 4-5 c1 equivalent}
c_1 = \frac{k \Delta_\alpha (\beta_1 k_1 - \alpha_0)^3(\alpha_0\alpha_2\beta_2 k_1 - \Delta_\alpha (\beta_1k_1 - \alpha_0))}{\alpha_0^4 \alpha_2^2 \alpha_3 k_1^5}
\end{equation}
Since $k_2 > 0$ and $k_3 > 0$, it is implied from \eqref{eq: 4-5 k2 equivalent} and \eqref{eq: 4-5 k3 equivalent} that
$\beta_1 k_1 - \alpha_0 > 0$ and $\Delta_\alpha > 0$. Then, combining \eqref{eq: 4-5 eqn2} and \eqref{eq: 4-5 eqn5}, one can obtain the expression of $k$ as
\begin{equation} \label{eq: 4-5 x}
k =  \frac{\alpha_0^4 \alpha_2^3 \alpha_3 \tilde{\mathcal{B}}_{12} k_1^7}{\Delta_\alpha (\beta_1 k_1 - \alpha_0)^3(\alpha_0\alpha_2\beta_2 k_1 - \Delta_\alpha (\beta_1 k_1 - \alpha_0))^2},
\end{equation}
which implies that (28a) holds.
Substituting \eqref{eq: 4-5 k2 equivalent}--\eqref{eq: 4-5 x} into \eqref{eq: 4-5 eqn1} implies the expression of $k_1$ as in (29). Together with
the expression of $k_1$ and \eqref{eq: 4-5 x}, the element values expressed in \eqref{eq: 4-5 k2 equivalent}--\eqref{eq: 4-5 c1 equivalent} can be further equivalent to the expressions
of $c_1$, $k_2$, $k_3$, and $b_1$ in (29),
which implies that  (28b) holds.
By (29) and
\eqref{eq: 4-5 x}, it is implied that \eqref{eq: 4-5 eqn2} can be equivalent to
(28c). The proof of the necessity part has been completed.

\emph{Sufficiency.}
Let the element values of $c_1$, $c_2$, $k_1$, $k_2$, $k_3$, and $b_1$ satisfy
(29), where (28a)--(28c) hold.
Then, it can be verified that the elements values are positive and finite.
Moreover, together with the expression of $c_2$ in
(29), it can be calculated that the admittance of the configuration in Fig.~5(e) is
$Y(s) = \tilde{n}(s)/\tilde{d}(s)$, where $\tilde{n}(s) = b_1 c_1 k_2 (b_1 k_1^2 + c_1^2 k_3) s^3 + b_1 k_1 k_2 (b_1 k_1^2 + c_1^2 k_2 + c_1^2 k_3) s^2 + c_1 k_2 (k_2 + k_3) (b_1 k_1^2 + c_1^2 k_3) s + k_1 k_2 k_3  (b_1 k_1^2 + c_1^2 k_2 + c_1^2 k_3)$ and $\tilde{d}(s) = b_1 (b_1 k_1^3 + b_1 k_1^2 k_2 + c_1^2 k_1 k_2 + c_1^2 k_1 k_3 + c_1^2 k_2 k_3) s^3 + c_1 k_2 (b_1 k_1^2 + b_1 k_1 k_2 + c_1^2 k_3) s^2 + k_3 (k_1 + k_2) (b_1 k_1^2 + c_1^2 k_2 + c_1^2 k_3)s$.
Since  (28c) holds, it can be verified that
\eqref{eq: 4-5 eqn1}--\eqref{eq: 4-5 eqn7} hold with $k$ satisfying \eqref{eq: 4-5 x}. Therefore,  $Y(s)$ is realizable by the circuit in Fig.~5(e).

\subsection{The Configuration in Fig.~5(f) (The Proof of Lemma~19)}

\emph{Necessity.} The admittance of the circuit configuration in Fig.~5(f) is calculated as $Y(s) = n(s)/d(s)$, where
$n(s) = b_1 c_1 c_2 s^4 + b_1 (c_1 k_2 + c_1 k_3 + c_2 k_1) s^3 + (b_1k_1k_2 + b_1k_1k_3 + c_1c_2k_2) s^2 + k_2(c_1k_3 + c_2k_1) s + k_1k_2k_3$ and
$d(s) = b_1(c_1+c_2)s^4 + (b_1k_1 + b_1k_2 + b_1k_3 + c_1c_2) s^3 + (c_1k_3 + c_2k_1 + c_2k_2) s^2 + k_3(k_1 + k_2) s$.
If the given admittance $Y(s)$ of this lemma is realizable as in Fig.~5(f), then the resultant of $n(s)$ and $d(s)$ in $s$ calculated as $R_0 (n, d, s) = -b_1^3 k_1 k_2 k_3^5(k_1 (b_1 k_1^2 + c_1^2 k_2)c_2 - c_1 (b_1 k_1^2 k_2 + b_1 k_1^2 k_3 + c_1^2 k_2 k_3))^2$ is zero.  Therefore, one obtains that $c_2$ satisfies the expression in
(31),   which further implies that the admittance of the configuration in Fig.~5(f) becomes
$Y(s) = \tilde{n}(s)/\tilde{d}(s)$, where $\tilde{n}(s) = b_1 c_1 (b_1 k_1^2 k_2 + b_1 k_1^2 k_3 + c_1^2 k_2 k_3) s^3 + b_1   k_1  (k_2 + k_3) (b_1 k_1^2 + c_1^2 k_2) s^2 + c_1 k_2 (b_1 k_1^2 k_2 + b_1 k_1^2 k_3 + c_1^2 k_2 k_3) s + k_1 k_2 k_3 (b_1 k_1^2 + c_1^2 k_2)$ and $\tilde{d}(s) = b_1 (b_1 k_1^3 + b_1 k_1^2 k_2 + b_1 k_1^2 k_3 + c_1^2 k_1 k_2 + c_1^2 k_2 k_3) s^3 + c_1 (b_1 k_1^2 k_2 + b_1 k_1^2 k_3 + b_1 k_1 k_2^2 + c_1^2 k_2 k_3) s^2 + k_3 (k_1 + k_2) (b_1 k_1^2 + c_1^2 k_2)s$. Therefore, there exists $k > 0$ such that $n(s) = k \alpha(s)$ and $d(s) = k \beta(s)$.
Then, it follows that
\begin{subequations}
\begin{align}
b_1 c_1 (b_1 k_1^2 k_2 + b_1 k_1^2 k_3 + c_1^2 k_2 k_3) &= k \alpha_3, \label{eq: 4-6 eqn1}    \\
b_1   k_1  (k_2 + k_3) (b_1 k_1^2 + c_1^2 k_2) &= k \alpha_2,    \label{eq: 4-6 eqn2}    \\
c_1 k_2 (b_1 k_1^2 k_2 + b_1 k_1^2 k_3 + c_1^2 k_2 k_3)  &=  k \alpha_1,  \label{eq: 4-6 eqn3}    \\
k_1 k_2 k_3 (b_1 k_1^2 + c_1^2 k_2)  &=  k \alpha_0,  \label{eq: 4-6 eqn4}    \\
b_1 (b_1 k_1^3 + b_1 k_1^2 k_2 + b_1 k_1^2 k_3 + c_1^2 k_1 k_2 + c_1^2 k_2 k_3)  &=  k \beta_3,  \label{eq: 4-6 eqn5}    \\
c_1 (b_1 k_1^2 k_2 + b_1 k_1^2 k_3 + b_1 k_1 k_2^2 + c_1^2 k_2 k_3)  &=  k \beta_2,  \label{eq: 4-6 eqn6}    \\
k_3 (k_1 + k_2) (b_1 k_1^2 + c_1^2 k_2)  &=  k \beta_1,  \label{eq: 4-6 eqn7}
\end{align}
\end{subequations}
where $k > 0$.
Therefore, it can be calculated that $\alpha_1 \beta_1 - \alpha_0\beta_2 = c_1 k_2^2 k_3^2 (b_1k_1^2 + c_1^2k_2)^2/k > 0$, which means that (30b) holds.
Let
\begin{equation}  \label{eq: 4-6 lambda mu}
x = k_1 > 0, ~~~ y = - \frac{k_1}{c_1} < 0.
\end{equation}
Then, it follows from \eqref{eq: 4-6 eqn4} and \eqref{eq: 4-6 eqn7} that $1/k_1 + 1/k_2 = \beta_1/\alpha_0$, which implies that $k_2$
can be expressed as in (31), which implies that $x > \alpha_0/\beta_1$. Then, combining \eqref{eq: 4-6 eqn1} and \eqref{eq: 4-6 eqn3}, one obtains $b_1 = \alpha_3 k_2/\alpha_1$, which implies that $b_1$ can be expressed as in (31). It follows from \eqref{eq: 4-6 eqn2} and \eqref{eq: 4-6 eqn4} that $b_1 (1/k_2 + 1/k_3) = \alpha_2/\alpha_0$, which implies that $k_3$ can be expressed as in (31), which implies that
(30a) holds.
Together with \eqref{eq: 4-6 lambda mu} and the element values of $k_2$, $k_3$, and $b_1$ in (31), it is implied from \eqref{eq: 4-6 eqn1} that $k$ can be expressed as in
\begin{equation} \label{eq: 4-6 x}
k =  - \frac{\alpha_0^3 \alpha_3 x^6 (\alpha_2 y^2 + \alpha_0)}{\alpha_1 \Delta_\alpha y^3 (\beta_1 x - \alpha_0)^3}.
\end{equation}
Together with \eqref{eq: 4-6 lambda mu}, \eqref{eq: 4-6 x}, and the element values of $k_2$, $k_3$, and $b_1$ in (31), it follows that \eqref{eq: 4-6 eqn4}, \eqref{eq: 4-6 eqn5}, and \eqref{eq: 4-6 eqn6} can be equivalent to (30e),
\begin{equation} \label{eq: 4-6 equation mu third-order}
\alpha_3   (\beta_1 \Delta_\alpha x + \alpha_0^2 \alpha_3) y^3 -  \alpha_1\alpha_2 \mathcal{B}_{13} x  y^2 +
\alpha_1   (\beta_1 \Delta_\alpha x - \alpha_0 (\Delta_\alpha - \alpha_0\alpha_3)) y     - \alpha_0 \alpha_1 \mathcal{B}_{13}    = 0,
\end{equation}
and
\begin{equation} \label{eq: 4-6 equation mu}
(\alpha_2  \tilde{\mathcal{B}}_{11} x - \alpha_0^2\alpha_3) y^2 + \alpha_0 (\tilde{\mathcal{B}}_{11} x - \alpha_0\alpha_1) = 0,
\end{equation}
respectively. By \eqref{eq: 4-6 equation mu}, it is clear that that $\alpha_2  \tilde{\mathcal{B}}_{11} x - \alpha_0^2\alpha_3 = 0$ can imply $\tilde{\mathcal{B}}_{11} x - \alpha_0\alpha_1 = 0$, which contradicts the assumption that $\Delta_\alpha \neq 0$. Therefore, it is implied that
$(\tilde{\mathcal{B}}_{11} x - \alpha_0 \alpha_1)(\alpha_0^2\alpha_3 - \alpha_2 \tilde{\mathcal{B}}_{11} x) > 0$, which together with \eqref{eq: 4-6 equation mu} implies that the expression of $y$ can be expressed as in (30f) and together with $\tilde{\mathcal{B}}_{11} > 0$ and $x > \alpha_0/\beta_1$ implies
(30d). By (30e) and (30f), it can be derived that \eqref{eq: 4-6 equation mu third-order} is further equivalent to (30c). The proof of the necessity part has been completed.

\emph{Sufficiency.}
Let the element values of $c_1$, $c_2$, $k_1$, $k_2$, $k_3$, and $b_1$
satisfy (31), where
(30a) and (30b) hold, and $x$ is a positive root of (30c) such that (30d) and (30e) hold
with (30f).
Then, it can be verified that the elements values are positive and finite.
Moreover, together with the expression of $c_2$ in
(31), it can be calculated that the admittance of the configuration in Fig.~5(f) is
$Y(s) = \alpha'(s)/\beta'(s)$, where $\alpha'(s) = b_1 c_1 (b_1 k_1^2 k_2 + b_1 k_1^2 k_3 + c_1^2 k_2 k_3) s^3 + b_1   k_1  (k_2 + k_3) (b_1 k_1^2 + c_1^2 k_2) s^2 + c_1 k_2 (b_1 k_1^2 k_2 + b_1 k_1^2 k_3 + c_1^2 k_2 k_3) s + k_1 k_2 k_3 (b_1 k_1^2 + c_1^2 k_2)$ and $\beta'(s) = b_1 (b_1 k_1^3 + b_1 k_1^2 k_2 + b_1 k_1^2 k_3 + c_1^2 k_1 k_2 + c_1^2 k_2 k_3) s^3 + c_1 (b_1 k_1^2 k_2 + b_1 k_1^2 k_3 + b_1 k_1 k_2^2 + c_1^2 k_2 k_3) s^2 + k_3 (k_1 + k_2) (b_1 k_1^2 + c_1^2 k_2)s$.
Since (30c), (30e), and (30f), it can be verified that
\eqref{eq: 4-6 eqn1}--\eqref{eq: 4-6 eqn7} hold with $k$ satisfying \eqref{eq: 4-6 x}.  Therefore,  $Y(s)$  is realizable by the circuit in Fig.~5(f).

\subsection{The Configuration in Fig.~5(g) (The Proof of Lemma~20)}

\emph{Necessity.}  The admittance of the circuit configuration in Fig.~5(g) is calculated as $Y(s) = n(s)/d(s)$, where
$n(s) = b_1 b_2 c_1 s^4 + b_1 (b_2 k_1 + c_1 c_2) s^3 + (b_1 c_1 k_2 +
b_1 c_2 k_1 + b_2 c_1 k_1) s^2 + k_1 (b_1 k_2 + c_1 c_2) s + c_1 k_1 k_2$
and
$d(s) = b_1 b_2 s^4 + (b_1 c_1 + b_1 c_2 + b_2 c_1) s^3 + (b_1 k_1 + b_1 k_2 +
c_1c_2) s^2 + c_1 (k_1 + k_2) s$. It is clear that
the resultant of $n(s)$ and $d(s)$ in $s$ calculated as $R_0 (n, d, s) = b_1 b_2 c_1^4 k_1 k_2 ( ((b_1 k_2 - b_2 k_1)^2 + b_1 c_2^2 k_1) c_1^2 - b_1 c_2 k_1 (b_1 k_2 + b_2 k_1) c_1 + b_1^2 b_2 k_1^2 k_2 )^2$ can be zero. Assume that the common factor is $(s+y)$ with $y > 0$.
Therefore, there exists $k > 0$ such that $n(s) = k (s+y) \alpha(s)$ and $d(s) = k (s+y) \beta(s)$.
Then, it follows that
\begin{subequations}
\begin{align}
b_1 b_2 c_1 &= k \alpha_3, \label{eq: 5-1 eqn1}    \\
b_1 (b_2 k_1 + c_1 c_2) &= k (\alpha_3 y + \alpha_2),  \label{eq: 5-1 eqn2}    \\
b_1 c_1 k_2 + b_1 c_2 k_1 + b_2 c_1 k_1  &= k (\alpha_2 y + \alpha_1),
\label{eq: 5-1 eqn3}    \\
k_1 (b_1 k_2 + c_1 c_2)  &=  k (\alpha_1 y + \alpha_0),
\label{eq: 5-1 eqn4}    \\
c_1 k_1 k_2  &=  k \alpha_0 y,  \label{eq: 5-1 eqn5}    \\
b_1 b_2  &= k \beta_3,  \label{eq: 5-1 eqn6}    \\
b_1 c_1 + b_1 c_2 + b_2 c_1  &= k (\beta_3 y + \beta_2),  \label{eq: 5-1 eqn7}    \\
b_1 k_1 + b_1 k_2 + c_1 c_2  &= k (\beta_2 y + \beta_1),  \label{eq: 5-1 eqn8}    \\
c_1 (k_1 + k_2)  &=  k \beta_1 y,    \label{eq: 5-1 eqn9}
\end{align}
\end{subequations}
where $k > 0$. Then, it follows from \eqref{eq: 5-1 eqn1} and \eqref{eq: 5-1 eqn6} that
the value of $c_1$ can be expressed as in (33).
Let
\begin{equation} \label{eq: 5-1 lambda k1}
x = k_1 > 0.
\end{equation}
Combining
\eqref{eq: 5-1 eqn5} and \eqref{eq: 5-1 eqn9}, one obtains the expression of $k_2$ as in
(33).
The assumption that $k_2 > 0$ implies that (32d) holds.
Substituting  the expression of $c_1$ and $k_2$ as in (33) into \eqref{eq: 5-1 eqn5} implies that
\begin{equation}  \label{eq: 5-1 x k1}
k = \frac{\alpha_3 k_1^2}{\beta_3 (\beta_1 k_1 - \alpha_0)},
\end{equation}
together with \eqref{eq: 5-1 eqn4} and \eqref{eq: 5-1 eqn8} further implies that $b_1$ can be expressed as in (33).
Therefore, one indicates that (32e) holds.   Substituting the expression of $b_1$ as in (33)
into \eqref{eq: 5-1 eqn6} yields the expression of $b_2$ as in (33).
Then, substituting $k$ in \eqref{eq: 5-1 x k1} and the expressions of $c_1$, $k_2$, and $b_1$ as in (33) into \eqref{eq: 5-1 eqn4}, the expression of $c_2$ can be obtained as in (33),
which implies (32a). Finally, substituting $k$ in \eqref{eq: 5-1 x k1} and the element values expressed in  (33) into \eqref{eq: 5-1 eqn2}, \eqref{eq: 5-1 eqn3} and \eqref{eq: 5-1 eqn7} implies (32b), (32c) and (32f), respectively.
The proof of the necessity part has been completed.

\emph{Sufficiency.}  Let the element values of
$c_1$, $c_2$, $k_1$, $k_2$, $b_1$, and $b_2$
satisfy (33), where
(32a) holds, and  $x>0$ and $y>0$ are positive roots for
(32b)   and  (32c), such that (32d)--(32f) hold.
Then, it can be verified that the element values are positive and finite. Since (32b), (32c), and (32f)
hold, it can be verified that \eqref{eq: 5-1 eqn1}--\eqref{eq: 5-1 eqn9} hold with $k$ satisfying \eqref{eq: 5-1 x k1}.   Therefore, $Y(s)$ is realizable by the circuit in Fig.~5(g).

\subsection{The Configuration in Fig.~5(h) (The Proof of Lemma~21)}

\emph{Necessity.}  The admittance of the circuit configuration in Fig.~5(h) is calculated as $Y(s) = n(s)/d(s)$, where
$n(s) = b_1 b_2 c_1 c_2 s^4 + b_1 b_2 (c_1 k_2 + c_2 k_1) s^3 + (b_1 b_2 k_1 k_2 + b_1 c_1 c_2 k_2 + b_2 c_1 c_2 k_1) s^2 + k_1 k_2 (b_1 c_2 + b_2 c_1) s + c_1 c_2 k_1 k_2$
and
$d(s) = b_1 b_2 (c_1+c_2) s^4 + (b_1 b_2 k_1 + b_1 b_2 k_2 + b_1 c_1 c_2 + b_2 c_1 c_2) s^3 + (k_1 + k_2) (b_1 c_2 + b_2 c_1) s^2 + c_1 c_2 (k_1 + k_2)s$.
It is clear that
the resultant of $n(s)$ and $d(s)$ in $s$ calculated as $R_0 (n, d, s) = b_1 b_2 c_1^4 c_2^4 k_1 k_2 ( (c_2^2 (b_1 k_2 - b_2 k_1)^2 + b_1 b_2^2 k_1 k_2^2) c_1^2 -  b_1 b_2 c_2 k_1 k_2 (b_1 k_2 + b_2 k_1) c_1 + b_1^2 b_2  c_2^2 k_1^2 k_2)^2$ can be zero. Assume that the common factor is $(s+\gamma)$ with $\gamma > 0$.
Therefore, there exists $k > 0$ such that $n(s) = k(s+\gamma) \alpha(s)$ and $d(s) = k(s+\gamma) \beta(s)$.
Then, it follows that
\begin{subequations}
\begin{align}
b_1 b_2 c_1 c_2 &= k \alpha_3, \label{eq: 5-2 eqn1}    \\
b_1 b_2 (c_1 k_2 + c_2 k_1) &= k (\alpha_3 \gamma + \alpha_2),  \label{eq: 5-2 eqn2}    \\
b_1 b_2 k_1 k_2 + b_1 c_1 c_2 k_2 + b_2 c_1 c_2 k_1  &= k (\alpha_2 \gamma + \alpha_1),
\label{eq: 5-2 eqn3}    \\
k_1 k_2 (b_1 c_2 + b_2 c_1)  &=  k (\alpha_1 \gamma + \alpha_0),
\label{eq: 5-2 eqn4}    \\
c_1 c_2 k_1 k_2  &=  k \alpha_0 \gamma,  \label{eq: 5-2 eqn5}    \\
b_1 b_2 (c_1+c_2)  &= k \beta_3,  \label{eq: 5-2 eqn6}    \\
b_1 b_2 k_1 + b_1 b_2 k_2 + b_1 c_1 c_2 + b_2 c_1 c_2  &= k (\beta_3 \gamma + \beta_2),  \label{eq: 5-2 eqn7}
\\
(k_1 + k_2) (b_1 c_2 + b_2 c_1)  &= k (\beta_2 \gamma + \beta_1),  \label{eq: 5-2 eqn8}    \\
c_1 c_2 (k_1 + k_2)  &=  k \beta_1 \gamma,    \label{eq: 5-2 eqn9}
\end{align}
\end{subequations}
where $k > 0$.
Let
\begin{equation} \label{eq: 5-2 lambda1 lambda2 lambda3}
x = c_1 > 0, ~~~ y = k_1 > 0, ~~~ z = b_1 > 0.
\end{equation}
It is clear from \eqref{eq: 5-2 lambda1 lambda2 lambda3} that the element values of $c_1$, $k_1$, and $b_1$ can be expressed as in
(35).
Then, it follows from \eqref{eq: 5-2 eqn1} and \eqref{eq: 5-2 eqn6} that $c_2$ can be expressed as in  (35), which implies that (34e) holds.
Combining \eqref{eq: 5-2 eqn5} and \eqref{eq: 5-2 eqn9}, one obtains the expression of $k_2$ as in (35), which implies that (34f) holds. By the element value expression of $k_2$, it is implied from \eqref{eq: 5-2 eqn4} and \eqref{eq: 5-2 eqn8} that (34a) holds. Substituting the element values of $c_1$, $c_2$, $k_1$, and $k_2$ into \eqref{eq: 5-2 eqn5} implies
\begin{equation} \label{eq: 5-2 x}
k = \frac{\alpha_3^2 x^3 y^2}{(\beta_3 x - \alpha_3) \Theta_0},
\end{equation}
which implies that $\Theta_0 > 0$.
Then, it is implied from \eqref{eq: 5-2 eqn1}, the element values of $c_1$, $c_2$, and $b_1$ in
(35), and $k$ in \eqref{eq: 5-2 x} that $b_2$ can be expressed in
(35).  Finally, substituting the element values in (35), and $k$ in \eqref{eq: 5-2 x} into \eqref{eq: 5-2 eqn3}, \eqref{eq: 5-2 eqn4}, and \eqref{eq: 5-2 eqn7} implies (34b)--(34d), respectively.
The proof of the necessity part has been completed.

\emph{Sufficiency.}
Let the element values of
$c_1$, $c_2$, $k_1$, $k_2$, $b_1$, and $b_2$
 satisfy (35), where
(34a) holds, and
$x > 0$, $y > 0$, and $z > 0$ are positive roots of  (34b)--(34d), such that (34e)--(34g) hold. Since (34a)--(34d) hold,  it can be verified that  \eqref{eq: 5-2 eqn1}--\eqref{eq: 5-2 eqn9} hold with $k$ satisfying
\eqref{eq: 5-2 x}.   Therefore,  $Y(s)$  is realizable by the circuit in Fig.~5(h).

\section{Side-View Train Suspension Model in Fig.~8}

By Newton's Second Law, the motion equations of the side-view train suspension model can be formulated as in (36), that is,
\begin{equation*}
M_g \ddot{z}_g + C_g \dot{z}_g + K_g z_g = E_g u + K_r z_r,
\end{equation*}
where $z_g  = \left[
           z_s, \theta_s, z_{b1}, \theta_{b1}, z_{b2}, \theta_{b2}, z_{w1}, z_{w2}, z_{w3}, z_{w4}
       \right]^\mathrm{T}$,
$u = \left[ F_1, F_2 \right]^\mathrm{T}$, $z_r = \left[ z_{r1}, z_{r2}, z_{r3}, z_{r4} \right]^\mathrm{T}$,
and
\begin{equation}
M_g = \left[
           \begin{array}{cccccccccc}
           m_s &   &  &   &   &   &   &   &   &    \\
             & I_s &  &   &   &  &   &   &   &    \\
             &   & m_{b} &   &   &  &   &  &   &    \\
             &   &   & I_{b} &   &   &  &   &   &   \\
            &   &    &   & m_{b}  &   &   &   &   &   \\
             &   &   &   &   & I_{b} &   &   &  &   \\
             &   &   &   &   &   &  m_w &  &   &   \\
           &   &   &   &   &   &   & m_w &   &   \\
           &   &   &   &   &   &   &  &  m_w &   \\
            &   &   &   &   &   &   &  &   &  m_w \\
           \end{array}
         \right],
\end{equation}
\begin{equation}
C_{g}  = \left[
           \begin{array}{cccccccccc}
           0 & 0 & 0 & 0 & 0 & 0 & 0 & 0 & 0 & 0  \\
           0 & 0 & 0 & 0 & 0 & 0 & 0 & 0 & 0 & 0  \\
           0 & 0 & 2c_p & 0 & 0 & 0 & -c_p & -c_p & 0 & 0  \\
           0 & 0 & 0 & 2 l_b^2 c_p & 0 & 0 & -l_b c_p & l_b c_p & 0 & 0 \\
           0 & 0 & 0 & 0 & 2 c_p & 0 & 0 & 0 & -c_p & -c_p \\
           0 & 0 & 0 & 0 & 0 & 2 l_b^2 c_p & 0 & 0 & -l_b c_p & l_b c_p \\
           0 & 0 & -c_p & -l_b c_p & 0 & 0 & c_p & 0 & 0 & 0 \\
           0 & 0 & -c_p &  l_b c_p & 0 & 0 &  0  & c_p & 0 & 0  \\
           0 & 0 & 0 & 0 & -c_p & -l_b c_p & 0 & 0 & c_p & 0 \\
           0 & 0 & 0 & 0 & -c_p &  l_b c_p & 0 & 0 & 0 & c_p  \\
           \end{array}
         \right],
\end{equation}
\begin{equation}
K_g  = \left[
           \begin{array}{cccccccccc}
            0 & 0 & 0 & 0 & 0 & 0 & 0 & 0 & 0 & 0  \\
           0 & 0 & 0 & 0 & 0 & 0 & 0 & 0 & 0 & 0  \\
           0 & 0 & 2k_p & 0 & 0 & 0 & -k_p & -k_p & 0 & 0  \\
           0 & 0 & 0 & 2 l_b^2 k_p & 0 & 0 & -l_b k_p & l_b k_p & 0 & 0 \\
           0 & 0 & 0 & 0 & 2 k_p & 0 & 0 & 0 & -k_p & -k_p \\
           0 & 0 & 0 & 0 & 0 & 2 l_b^2 k_p & 0 & 0 & -l_b k_p & l_b k_p \\
           0 & 0 & -k_p & -l_b k_p & 0 & 0 & k_p+k_w & 0 & 0 & 0 \\
           0 & 0 & -k_p &  l_b k_p & 0 & 0 &  0  & k_p+k_w & 0 & 0  \\
           0 & 0 & 0 & 0 & -k_p & -l_b k_p & 0 & 0 & k_p+k_w & 0 \\
           0 & 0 & 0 & 0 & -k_p &  l_b k_p & 0 & 0 & 0 & k_p+k_w  \\
           \end{array}
         \right],
\end{equation}
\begin{equation}
E_g  = \left[
         \begin{array}{cc}
           -1 & -1 \\
           -l_s & l_s \\
           1 & 0 \\
           0 & 0 \\
           0 & 1 \\
           0 & 0 \\
           0 & 0 \\
           0 & 0  \\
           0 & 0  \\
           0 & 0  \\
         \end{array}
       \right],
\end{equation}
and
\begin{equation}
K_{r}  = \left[
           \begin{array}{cccc}
             0 & 0 & 0 & 0 \\
             0 & 0 & 0 & 0 \\
             0 & 0 & 0 & 0 \\
             0 & 0 & 0 & 0 \\
             0 & 0 & 0 & 0 \\
             0 & 0 & 0 & 0 \\
             k_w & 0 & 0 & 0 \\
             0 & k_w & 0 & 0 \\
             0 & 0 & k_w & 0 \\
             0 & 0 & 0 & k_w \\
           \end{array}
         \right].
\end{equation}

\section{Conclusion}

In this report, the proofs of some results  in the original
paper \cite{IJCTA_sub} as well as some other supplementary material have been presented,  which are omitted from the paper for brevity.



\vspace{0.2cm}

\section*{References}


\begin{thebibliography}{00}

\bibitem{AV06}
B.D.O. Anderson, S. Vongpanitlerd, Network Analysis and Synthesis: A Modern Systems Theory
Approach, 3rd Edition,  Dover Publication, New York, 2006.

\bibitem{Bah84}
H. Baher, Synthesis of Electrical Networks,  Wiley, New York, 1984.

\bibitem{You15}
D.C. Youla,  Theory and Synthesis of Linear Passive Time-invariant Networks, Cambridge University Press, Cambridge, 2015.

\bibitem{MS19}
A. Morelli, M.C. Smith,  Passive Network Synthesis: An Approach to Classification,
SIAM, Philadelphia, 2019.

\bibitem{CWC19}
M.Z.Q. Chen, K. Wang, G. Chen. Passive Network Synthesis: Advances with Inerter,
 World Scientific, Singapore, 2019.

\bibitem{BD49}
R. Bott, R.J. Duffin, Impedance synthesis without use of transformers,  Journal of Applied Physics
20 (8) (1949) 816.


\bibitem{HS17}
T.H. Hughes, M.C. Smith. Controllability of linear passive network behaviors,
System and Control Letters  101 (2017) 58--66.


\bibitem{Smi02}
M.C. Smith,  Synthesis of mechanical networks: The inerter,  IEEE Trans. Automatic Control
47 (10) (2002) 1648--1662.



\bibitem{ELSS06}
S. Evangelou, D.J.N. Limebeer, R.S. Sharp, M.C. Smith,
Control of motorcycle steering instabilities, IEEE Control Systems Magazine 26 (5) (2006) 78--88.


\bibitem{PS06}
C. Papageorgiou, M.C. Smith, Positive real synthesis using matrix inequalities for mechanical networks:
Application to vehicle suspension,  IEEE Trans. Control Systems Technology 14 (3) (2006) 423--435.



\bibitem{WLLSC09}
F.C. Wang, M.K. Liao, B.H. Liao, W.J. Su,  H.A. Chan, The performance improvements of train suspension systems with mechanical networks employing inerters,  Vehicle System Dynamics  47 (7) (2009) 805--830.




\bibitem{JMGS12}
J.Z. Jiang, A.Z. Matamoros-Sanchez,  R.M. Goodall,  M.C. Smith,
Passive suspension incorporating inerters for railway vehicles,  Vehicle System Dynamics
50 (2012) 263--276.




\bibitem{CHW15}
M.Z.Q. Chen,  Y.  Hu,  F.C.  Wang, Passive mechanical control with a special class of positive-real controllers: Application to passive vehicle suspensions,  Journal of Dynamic Systems, Measurement, and Control
137 (12) (2015) 121013.




\bibitem{YS16}
K. Yamamoto,   M.C. Smith,
Bounded disturbance amplification for mass chains with passive interconnection,
IEEE Trans. Automatic Control 61 (6) (2016) 1565--1574.




\bibitem{CLLNSC17}
L. Chen, C. Liu,  W. Liu,  J. Nie,  Y. Shen,  G. Chen, Network synthesis and parameter optimization for vehicle
suspension with inerter,  Advances in Mechanical Engineering  9 (1) (2017) 1--7.




\bibitem{CSW17}
H.J. Chen, W.J. Su,   F.C.  Wang, Modeling and analyses of a connected multi-car train system employing the inerter,  Advances in Mechanical Engineering  9 (8) (2017) 1--13.







\bibitem{TGP19}
A.A. Taflanidis,  A. Giaralis,   D. Patsialis,  Multi-objective optimal design of inerter-based
vibration absorbers for earthquake protection of multi-storey building structures,
Journal of the Franklin Institute 356 (14) (2019) 7754--7784.



\bibitem{WNZ19}
X. Wei,  B.F. Ng,  X. Zhao,  Aeroelastic load control of large and flexible wind turbines through mechanically driven flaps,
Journal of the Franklin Institute  356 (14) (2019) 7810--7835.

\bibitem{LCYZY19}
C. Liu, L. Chen, X. Yang, X. Zhang, and Y. Yang, General theory of skyhook control and its application to semi-active suspension control strategy design, IEEE Access 7 (2019) 101552--101560.

\bibitem{LCZYN20}
C. Liu, L. Chen, X. Zhang, Y. Yang, and J. Nie,
Design and tests of a controllable inerter
with fluid-air mixture condition, IEEE Access 8 (2020) 125620--125629.






\bibitem{NDZSL20}
D. Ning,  H. Du,  N. Zhang,  S. Sun,  W. Li, Controllable electrically interconnected suspension system for improving vehicle vibration performance,  IEEE/ASME Trans. Mechatronics
25 (2) (2020) 859--871.

\bibitem{SJNC20}
Y. Shen, J.Z. Jiang, S.A. Neild, L. Chen,
Vehicle vibration suppression using an inerter-based mechatronic device,
Proceedings of the Institution of Mechanical Engineers,
Part D: Journal of Automobile Engineering 234 (10--11) 2592--2601.



\bibitem{BK21}
M. Baduidana,  A. Kenfack-Jiotsa,  Optimal design of inerter-based isolators minimizing the compliance and
mobility transfer function versus harmonic and random ground acceleration excitation,
Journal of Vibration and Control 27 (11--12) (2021) 1297--1310.




\bibitem{YWDLZ21}
L. Yang,  R. Wang,  R. Ding,  W. Liu,  Z. Zhu, Investigation on the dynamic performance of a new semi-active
hydro-pneumatic inerter-based suspension system with MPC control strategy,
Mechanical Systems and Signal Processing 154 (2021) 107569.




\bibitem{DWSZ21}
N. Duan, Y. Wu,  X.M. Sun,  C.  Zhong,
Vibration control of conveying fluid pipe based on inerter enhanced nonlinear energy sink,
IEEE Trans. Circuits and Systems I: Regular Papers 68 (4)  (2021) 1610--1623.




\bibitem{CS09(2)}
M.Z.Q. Chen,  M.C. Smith,  A note on tests for positive-real functions,
IEEE Trans. Automatic Control 54 (2) (2009) 390--393.







\bibitem{JS11}
J.Z. Jiang, M.C. Smith,  Regular positive-real functions and five-element network synthesis for electrical and mechanical networks,
IEEE Trans. Automatic Control 56 (6) (2011) 1275--1290.




\bibitem{JS12}
J.Z. Jiang,  M.C. Smith, Series-parallel six-element synthesis of biquadratic impedances,
IEEE Trans. Circuits and Systems I: Regular Papers 59 (11) (2012) 2543--2554.




\bibitem{CWSL13}
M.Z.Q. Chen,  K. Wang, Z.  Shu,  C.  Li, Realizations of a special
class of admittances with strictly lower complexity than canonical forms,
IEEE Trans. Circuits and Systems I: Regular Papers 60 (9) (2013) 2465--2473.




\bibitem{YKKP14}
B.S. Yarman, R.  Kopru,  N. Kumar,   C. Prakash, High precision synthesis of a Richards immittance via parametric approach,   IEEE Trans. Circuits and Systems I: Regular Papers  61 (4) (2014) 1055--1067.




\bibitem{WC15}
K. Wang, M.Z.Q. Chen, Y. Hu, Synthesis of biquadratic impedances with at most four passive elements,
Journal of the Franklin Institute 351 (3) (2014) 1251--1267.







\bibitem{ST17}
M.S. Sarafraz,  M.S. Tavazoei, Passive realization of fractional-order impedances by a fractional element and RLC components: Conditions and procedure,  IEEE Trans. Circuits and Systems I: Regular Papers 64 (3) (2017)
585--595.


\bibitem{Hug17}
T.H. Hughes,  Why RLC realizations of certain impedances need many more energy storage elements than expected,
IEEE Trans. Automatic Control
62 (9) (2017) 4333--4346.



\bibitem{ZJWN17}
S.Y. Zhang, J.Z.  Jiang,  H.L. Wang,   S. Neild,
Synthesis of essential-regular bicubic impedances,
International Journal of Circuit Theory and Applications 45 (11) (2017) 1482--1496.






\bibitem{LQ18}
G. Liang,  Z. Qi, Synthesis of passive fractional-order LC n-port with three element orders,
IET Circuits, Devices and Systems 13 (1) (2018) 61--72.

\bibitem{WCLC18}
K. Wang, M.Z.Q. Chen, C. Li, G. Chen, Passive controller realization of a biquadratic impedance with
double poles and zeros as a seven-element series-parallel network for effective mechanical control, IEEE
Trans. Automatic Control 63 (9) (2018) 3010--3015.



\bibitem{HMS19}
T.H. Hughes,  A. Morelli,  M.C. Smith, On a concept of genericity for RLC networks,  Systems and Control Letters   134 (2019) 104562.





\bibitem{WJ19}
K. Wang, X. Ji, Passive controller realization of a bicubic admittance containing a pole at $s=0$ with no more than five
elements for inerter-based mechanical control,
Journal of the Franklin Institute  356 (14) (2019) 7896--7921.


\bibitem{WC21_JFI}
K. Wang, M.Z.Q. Chen,
Passive mechanical realizations of bicubic impedances with no more than five elements for inerter-based control design,
Journal of the Franklin Institute 358 (10) (2021) 5353--5385.







\bibitem{Hug20}
T.H. Hughes,  Minimal series-parallel network realizations of bicubic impedances,  IEEE Trans. Automatic Control  65 (12) (2020) 4997--5011.



\bibitem{WC21}
K. Wang,   M.Z.Q.  Chen,  On realizability of specific biquadratic
impedances as three-reactive seven-element series-parallel networks for
inerter-based mechanical control,  IEEE Trans. Automatic Control
66 (1) (2021) 340--345.






\bibitem{Kal10}
R. Kalman,   Old and new directions of research in system theory,
Perspectives in Mathematical System Theory, Control, and Signal Processing
398 (2010) 3--13.



\bibitem{Smi17}
M.C. Smith,  Kalman's last decade: Passive network synthesis,  IEEE Control Systems Magazine
37 (2) (2017) 175--177.





\bibitem{RV21}
A. Recski,   \'{A}.  V\'{e}k\'{a}ssy,   Interconnection, reciprocity and a hierarchical
classification of generalized multiports,  IEEE Trans. Circuits and Systems I: Regular Papers
68 (9) (2021) 3682--3692.




\bibitem{LBHD11}
J. Lavaei,  A. Babakhani,  A. Hajimiri,  J.C.   Doyle, Solving large-scale
hybrid circuit-antenna problems,
IEEE Trans. Circuits and Systems I: Regular Papers 58 (2) (2011) 374--387.





\bibitem{DGYM20}
R. Deaton,   M. Garzon,  R.  Yasmin,   T.  Moorse,  A model for self-assembling circuits with voltage-controlled growth,
International Journal of Circuit Theory and Applications 48 (7) (2020) 1017--1031.








\bibitem{Kes17}
A.\"{U}. Keskin,  Electrical circuits in biomedical engineering: Problems with solutions, Springer, Berlin, 2017.


\bibitem{Tav20}
M.S. Tavazoei,  Conditions on polynomials involved in admittance functions passively realizable by using RLC
and two fractional elements,   IEEE Trans. Circuits and Systems II: Express Briefs 67 (6) (2020) 999--1003.


\bibitem{LLDYYIFL21}
D. Lin,  X. Liao,  L.  Dong,  R. Yang,  S.S. Yu,   H.H.C. Iu,   T. Fernando,   Z. Li,
Experimental study of fractional-order RC circuit model using the Caputo and Caputo-Fabrizio derivatives,
IEEE Trans. Circuits and Systems I: Regular Papers 68 (3) (2021) 1034--1044.





\bibitem{XLP16}
J. Xiong,  A. Lanzon,  I.R.   Petersen, Negative imaginary lemmas for descriptor systems,
IEEE Trans. Automatic Control 61 (2) (2016) 491--496.











\bibitem{ZGSX19}
D. Zhao,  K. Galkowski,  B.  Sulikowski,   L.  Xu, 3-D modelling of rectangular circuits as the particular class
of spatially interconnected systems on the plane,
Multidimensional Systems and Signal Processing 30 (2019) 1583--1608.



\bibitem{ZDG95}
K. Zhou,  J.C. Doyle,   K.  Glover,  Robust and Optimal Control,  Prentice Hall, New Jersey, 1995.


\bibitem{Van60}
M.E. Van Valkenburg,  Introduction to Modern Network Synthesis,   Wiley, New York, 1960.








\bibitem{Fuh12}
P.A. Fuhrmann,  A Polynomial Approach to Linear Algebra, Second
Edition,  Springer, New York, 2012.




\bibitem{Ses59}
S. Seshu,   Minimal realizations of the biquadratic minimum function,
IRE Trans. Circuit Theory 6 (4) (1959) 345--350.



\bibitem{SR61}
S. Seshu,  M.B. Reed,  Linear Graphs and Electrical Networks,  Addison-Wesley, Boston, 1961.



\bibitem{Lad64}
E.L. Ladenheim,  Three-reactive five-element biquadratic structures,
IEEE Trans. Circuit Theory 11 (1) (1964) 88--97.




\bibitem{Wang_sup}
K. Wang,  M.Z.Q. Chen,  F.   Liu,   Supplementary material to: Synthesis of a bicubic admittance containing a pole at the origin as a passive six-element series-parallel mechanical circuit for inerter-based control (technical report to be available in arXiv.org).






\end{thebibliography}

\begin{thebibliography}{99}


\bibitem{IJCTA_sub}
K. Wang, M. Z. Q. Chen, and F. Liu,  ``Series-parallel mechanical circuit synthesis of a positive-real third-order admittance using at most six passive elements for inerter-based control,'' \emph{Journal of the Franklin Institute}, under review.







\end{thebibliography}
\end{document}